\documentclass[letterpaper,10pt]{article}
\usepackage{amsfonts}
\usepackage{amssymb, amsmath, amsthm,  graphicx}
\usepackage[margin=1in]{geometry}
\usepackage[bottom]{footmisc}
\usepackage{commath}
\usepackage[utf8]{inputenc}
\usepackage{natbib}
\usepackage{booktabs, multirow}
\usepackage{url}

\renewcommand{\baselinestretch}{1.5}

\DeclareMathOperator*{\argmin}{argmin}

\newtheorem{thm}{Theorem}

\newtheorem{lemma}{Lemma}

\newtheorem{assum}{Assumption}
\newtheorem{remark}{Remark}
\newtheorem{definition}{Definition}
\begin{document}

\title{High Dimensional Linear GMM 
}
\author{\textsc{Mehmet Caner\thanks{%
Ohio State University, 452 Arps Hall, Department of Economics, Translational Data Analytics, Department of Statistics, OH 43210. Email:caner.12@osu.edu.}} \and \textsc{Anders Bredahl Kock\thanks{%
University of Oxford, Aarhus University and CREATES, Department of Economics. Email: anders.kock@economics.ox.ac.uk. The authors would like to thank Yanqin Fan, Eva Janssens, Adam McCloskey and Jing Tao for helpful comments and suggestions.}}
}
\date{\today}

\maketitle

\begin{abstract}
This paper proposes a desparsified GMM estimator for estimating high-dimensional regression models allowing for, but not requiring, many more endogenous regressors than observations. We provide finite sample upper bounds on the estimation error of our estimator and show how asymptotically uniformly valid inference can be conducted in the presence of conditionally heteroskedastic error terms. We do not require the projection of the endogenous variables onto the linear span of the instruments to be sparse; that is we do not impose the instruments to be sparse for our inferential procedure to be asymptotically valid. Furthermore, the variables of the model are not required to be sub-gaussian and we also explain how our results carry over to the classic linear dynamic panel data model. Simulations show that our estimator has a low mean square error and does well in terms of size and power of the tests constructed based on the estimator. 

\vspace{0.2cm}

\noindent\textit{Keywords:} GMM, Desparsification, Uniform inference, High-dimensional models, Linear regression, Dynamic panel data.


\end{abstract}

\section{Introduction}

GMM is one of the cornerstones of modern econometrics. It has been used to link economic theories to estimation of structural parameters as well as testing. It has also infused other fields such as finance, marketing and accounting. The popularity of GMM stems furthermore from its role in instrumental variable estimation in the presence of endogenous regressors. 

Until recently, the validity of GMM based inference had only been established in asymptotic regimes with a fixed number of instruments and endogenous regressors as sample size tended to infinity. For example, \cite{c09} proposed a Bridge type of penalty on the GMM estimator with a fixed number of parameters and analyzed its model selection properties. Furthermore, \cite{cf15} used an adaptive Lasso type penalty to select instruments --- again this was done in a setting with a fixed number of instruments. 

The setting of an increasing number of endogenous variables was analyzed by \cite{cz14}. These authors considered the adaptive elastic net penalty and studied estimation and variable selection consistency. Next, \cite{chl18} proposed an adaptive elastic net based estimator which can simultaneously select the model and the valid instruments while at the same time estimating the structural parameters. However, the asymptotic framework in all of the above papers is pointwise and the sample size is always larger than the number of instruments (albeit this is allowed to diverge with the sample size).

In a seminal paper, \cite{bcch12} proposed a heteroskedasticity robust procedure for inference in IV estimation valid in settings with many more instruments than observations. Their results do not rely on the data being sub-gaussian, making it useful for many economic applications. \cite{gt14} also consider high-dimensional instrumental variable estimation. 
Furthermore, \cite{bch14} developed the first uniformly valid confidence intervals for the treatment coefficient in the presence of a high dimensional vector of control variables.

Recently, \cite{z15, z18} introduced new oracle inequalities for high dimensional two stage least squares estimators. Based on that, work \cite{glt18} proposd a debiased version of a lasso based two stage least squares estimator with sub gaussian data, and homoskedastic errors. \cite{nnll15} also considered estimating equations and confidence regions. Simultaneously with \cite{glt18},  \cite{BCHN} proposed a new instrumental variable estimator satisfying empirical orthogonality conditions in high dimensions in the presence of heteroskedastic data. A useful feature of their approach is that it does not involve tuning parameters.  Another relevant paper is by \cite{belloni19}, which provides a new way of handling linear and nonlinear instrumental variables regression as well as relaxing the sparsity assumption. Under some high level assumptions, they provide general results. They introduce double/de-biased regularized GMM which starts with a Dantzig estimator for structural parameter estimation and debiases that again by means of Dantzig based estimation.  Furthermore, they provide a result for homoskedastic linear instrumental variable estimation under primitive conditions. 
The limit of their estimator is standard normal. A further related paper is by \cite{breunig18}. They use the lasso as the first step estimator and debias it twice  through nodewise regressions.  They provide inferential results when the variables are subgaussian, and approximate sparsity conditions on the structural parameters, on the inverse of the standard GMM variance, as well as on another matrix related to combination of the moments, are holding simultaneously. They were able to provide inferential results for linear combination of parameters and also consider the case of weak identification.  Note that we are able to avoid using subgaussian data due to \cite{cck16}.

 
Our approach is based on debiasing a two-step Lasso-GMM estimator. Thus, our estimator is related to \cite{van2014} who proposed a desparsified Lasso estimator and established that the confidence intervals based on it are asymptotically uniformly valid. Simultaneously, similar advancements were made in the papers by \cite{jm14} and \cite{zz14}. \cite{ck18} proposed debiasing the conservative Lasso in the context of a plain linear regression model without endogenous covariates and showed how it can be used to construct uniformly valid confidence intervals in the presence of heteroskedasticity. In addition, the asymptotic inference can simultaneously involve a large number of coefficients.

This paper proposes a high dimensional penalized GMM estimator where the number of instruments and explanatory variables are both allowed, yet not required, to be larger than the sample size. We do not impose sparsity on the instruments; that is, we do not require only a small subset of the instruments to be valid. While we develop the theory in the context of cross sectional data, we also explain how the theory is valid in dynamic panel data models upon taking first differences. The error terms are allowed to be heteroskedastic conditionally on the instruments. Our approach does not impose the data to be sub-gaussian as we benefit from concentration inequalities by \cite{cck16}. For debiasing our estimator we need an approximate inverse of a certain singular sample covariance matrix, cf Section \ref{sec:desp}. Our construction of this approximate inverse relies on the CLIME estimator of \cite{cai11}. Uniformly valid confidence intervals for the debiased estimator are developed. The tuning parameter present is chosen by cross-validation. Finally, the finite sample properties of our estimator are investigated through simulations and we compare it to the estimator in \cite{glt18}. In the presence of heteroskedasticity, our estimator performs very well in terms of size, power, length of the confidence interval and mean square error. 
   
Section \ref{model} introduces the model and estimator. Next, Section \ref{assumpt} lays out the used assumptions and and oracle inequality for the penalized two-step GMM estimator.  Section \ref{sec:desp1} develops the approximate inverse used for debiasing the penalized two-step GMM estimator. Section \ref{inferen} establishes the asymptotically uniform validity of our inference procedure and Section \ref{tuning} explains the tuning parameter choice.  Finally, the Monte Carlo simulations are contained in Section \ref{mc}.

\section{Notation, model and estimator}\label{model}

Prior to introducing the model and the ensuing inference problem we present notation used throughout the paper.

\subsection{Notation}
For any $x\in \mathbb{R}^n$, let $\enVert[0]{x}_1, \enVert[0]{x}_2$ and $\enVert[0]{x}_\infty$ denote its $l_1$-, $l_2$-, and the $l_\infty$-norm, respectively. Also, we shall let $\enVert[0]{x}_{l_0}$ be the $l_0$-``norm'' counting the number of non-zero entries in $x$. For an $m\times n$ matrix $A$, we define $\enVert[0]{A}_\infty=\max_{1\leq i\leq m, 1\leq j\leq n}|A_{ij}|$. $\enVert[0]{A}_{l_\infty}=\max_{1\leq i\leq m}\sum_{j=1}^n|A_{ij}|$ denotes the induced $l_\infty$-norm of $A$. Similarly, $\enVert[0]{A}_{l_1}=\max_{1\leq j\leq n}\sum_{i=1}^m|A_{ij}|$ denotes the induced $l_1$-norm.
For any symmetric matrix $B$, let $\ Eigmin (B)$ and $Eigmax (B)$ denote the smallest and largest eigenvalues of $B$, respectively. For $S\subseteq\cbr[0]{1,...,n}$, we let $x_S$ be the modification of $x$ that places zeros in all entries of $x$ whose index does not belong to $S$. $|S|$ denotes the cardinality of $S$. For any $n\times n$ matrix $C$ let $C_S$ denote the $|S|\times |S|$ submatrix of $C$ consisting only of the rows and columns indexed by $S$. 
$diag(x)$ denotes the diagonal matrix having $x_j$ as its $j$ diagonal element.
$e_j$ will denote the $j$th canonical basis vector for $\mathbb{R}^n$. $\stackrel{d}{\to}$  indicates convergence in distribution. 

\subsection{The Model}
We consider the linear model
\begin{align}
Y=X\beta _{0}+u\label{eq:mod},
\end{align}
where $X$ is the $n\times p$ matrix of potentially endogenous explanatory variables and $u$ is an $n\times 1$ vector of error terms. $\beta _{0}$ is the $p\times 1$ population vector of coefficients, which we shall assume to be sparse. Thus, $Y$ is $n\times 1$.  However, the location of the non-zero coefficients is unknown. Let $S_0=\cbr[0]{j:\beta_{0j}\neq 0}$ denote the set of relevant regressors and $s_0=|S_0|$ their cardinality. In this paper we study the high-dimensional case where $p$ is much greater than $n$ but our results actually only require $n\to \infty$ --- thus $p\leq n$ is covered as well. All regressors are allowed to be endogenous but are not required to be. In particular, this means that upon taking first differences, the classic linear dynamic panel data model can be cast in our framework. We provide more details on this in Section \ref{sec:dynpan} of the supplementary appendix.
%
%
%

We assume that $q$ instruments are available and let $Z$ denote the $n\times q$ matrix of instruments. Exogenous variables can instrument themselves as usual. The regime under investigation is $q \ge p>n$ where there are many instruments and regressors compared to the sample size. However, our results can easily be adapted to any regime of orderings and growth rates of $p,\ q$ and $n$ (as long as we have at our disposal at least as many instruments as endogenous variables, i.e. $q\geq p$) \footnote{For details on arbitrary growth rates of $p,\ q$ and $n$ we refer to Remark \ref{rem:genres} in the beginning of the appendix.}. Letting $X_i$  and $Z_i$ denote the $i$th row of $X$ and $Z$, respectively, $i=1,...,n$, written as column vectors, we assume that
\begin{align}
E Z_i u_i = 0, \label{eq:GMMcond}
\end{align}
for all $i=1,..., n$ amounting to the the instruments being uncorrelated with the error terms.

The goal of this paper is to construct valid tests and confidence intervals for the entries of $\beta_0$. We do not impose that the columns of $A$ are sparse in a first step equation of the type $X=ZA+\epsilon$ (put differently, the $L_2$-projection of the covariates on the linear span of the instruments is not assumed to be sparse) and also allow $u_i$ to be heteroskedastic conditionally on $Z_i$ for each $i=1,...,n$. In addition, we do not impose the random variables in the model (\ref{eq:mod}) to be sub-gaussian. 

 Based on (\ref{eq:GMMcond}), we propose the following penalized \emph{first-step} Lasso GMM estimator.
\begin{equation}
\hat{\beta}_F = \argmin_{\beta \in \mathbb{R}^p} \left[ \frac{(Y - X \beta)'ZZ'(Y - X \beta)}{n^2q} + 2 \lambda_n \|\beta \|_1\right],\label{eq:FSGMM}
\end{equation}
where $\lambda_n$ is a positive tuning parameter sequence defined in (\ref{eq:lambda_def}) in the appendix. While we shall later see that this estimator is consistent under suitable regularity conditions, the main focus of this paper is a generalization of the classic GMM estimator to the high-dimensional setting, allowing for a $q\times q$ weight matrix $\hat{W}_d=diag(1/\hat{\sigma}_1^2,...,1/\hat{\sigma}_q^2)$ with $\hat{\sigma}_l^2 = \frac{\sum_{i=1}^n Z_{il}^2 \hat{u}_i^2}{n}$ and $\hat{u}_i = Y_i - X_i' \hat{\beta}_F$. This \emph{two-step} Lasso GMM estimator is defined as
\begin{align}
\hat{\beta} = \argmin_{\beta \in \mathbb{R}^p} \left[ \frac{(Y - X \beta)'Z}{n} \frac{\hat{W}_d}{q} \frac{Z'(Y - X \beta)}{n} + 2 \lambda_n^* \|\beta \|_1\right].\label{eq:1}
\end{align}
For the two-step GMM estimator we shall use $\lambda_n^*$ as defined in (\ref{eq:lambdastar_def}) in the appendix. While the exact form of $\lambda_n^*$ is rather involved we note that under Assumption \ref{2} below one has that $\lambda_n^*=O(\sqrt{\ln q/n})$. 

\begin{remark}\label{rem:generalW}
Although we focus on the case of a diagonal weight matrix $\hat{W}_d$, it is worth mentioning that our results can be shown to remain valid in case of a general weight matrix $\hat{W}$ if there exists (a sequence of) non-random matrices $W$ such that
\begin{align*}
\| \hat{W} - W \|_{l_{\infty}} = o_p (1) \quad \text{and} \quad  \| W \|_{l_{\infty}} \le C < \infty.
\end{align*}
for some universal $C>0$. However, since $W$ is $q\times q$, assuming it to have uniformly bounded $l_\infty$-norm is restrictive. Thus, even though $\| \hat{W} - W \|_{l_{\infty}} = o_p (1)$ and $\| W \|_{l_{\infty}} \le C < \infty$ can be relaxed at the expense of strengthening some of our other assumptions, we shall focus on the case of a diagonal weight matrix as this is enough to handle conditionally heteroskedastic error terms.

Note also that the classic choice of weighting matrix in low-dimensional GMM, $\hat{W} = [ n^{-1} \sum_{i=1}^n Z_i Z_i' \hat{u}_i^2 ]^{-1}$, is not applicable since it is not well-defined for $q > n$ due to the reduced rank of $n^{-1} \sum_{i=1}^n Z_i Z_i' \hat{u}_i^2 $. 
\end{remark}

\section{Assumptions and oracle inequalities}\label{assumpt}
Throughout we assume that $X_i, Z_i$ and $u_i$ are independently and identically distributed across $i=1,...,n$. Before stating our first assumption, we introduce the following notation. First, let 
\begin{equation}
\Sigma_{xz} = E X_1 Z_1',\label{sxz}
\end{equation}
and, with $\sigma_l^2 =  E  Z_{1l}^2 u_1^2$, $l=1,...,q$, set 
\begin{align}
W_d
=
diag(1/\sigma_1^2,...,1/\sigma_q^2) \label{wd}.
\end{align}
Next,
define the population adaptive restricted eigenvalue of $\Sigma_{xz} W_d \Sigma_{xz}'$
\begin{equation}
\phi_{\Sigma_{xzw}}^2 (s) = \min 
\cbr[3]{  \frac{\delta' (\Sigma_{xz} W_d \Sigma_{xz}') \delta}{q \|\delta_S\|_2^2}:\delta \in \mathbb{R}^p  \setminus \cbr[0]{0},\ \| \delta_{S^c} \|_1 \le 3 \sqrt{s} \|\delta_S\|_2,\ |S|\leq s}\label{pev1}
\end{equation}
which is the relevant extension of the classic adaptive restricted eigenvalue from the linear regression model with exogenous regressors (which only involves $E(X_1X_1')$). Verifying that the sample counterpart of (\ref{pev1}) is bounded away from zero, which is an important step in establishing the oracle inequalities in Theorem \ref{thm1} below, becomes more challenging than in the classic setting, cf. Lemma \ref{evalue2} in the Appendix.

\begin{assum}\label{1}
Assume that $EZ_1 u_1=0$. Furthermore, $\max_{1 \le j \le p} E |X_{1j} |^{r_x}$, 
 $\max_{1 \le l \le q } E |Z_{1l}|^{r_z}$, and 
$ E |u_1|^{r_u}$ are uniformly bounded from above (over $n$) for $r_z, r_x, r_u \ge 4$. Finally, $\phi_{\Sigma_{xzw}}^2 (s_0)$ and $\min_{1 \le l \le q} \sigma_l^2= \min_{1 \le l \le q}  E Z_{1l}^2 u_1^2$ are bounded away from zero uniformly over $n$.
\end{assum}
Note that Assumption \ref{1} does not impose sub-gaussianity of the random variables. Assumption \ref{1} is used to establish the oracle inequality in Theorem \ref{thm1} below, which in turn plays an important role for proving the asymptotic gaussianity of the (properly centered and scaled) desparsified two-step GMM estimator. Furthermore, Assumption \ref{1} does not require $\Sigma_{xz} W_d \Sigma_{xz}'/q$ to be full rank.  In other words, we allow for ill-posedness due to many endogenous regressors and instruments.
Thus, we allow for some of the instruments to be weakly correlated with the explanatory variables.


While Assumption \ref{1} imposes restrictions for each $n\in\mathbb{N}$, the following assumption only imposes asymptotic restrictions. It restricts the growth rate of the moments of certain maxima as well as the number of non-zero entires of $\beta_0$, i.e. $s_0$. Prior to stating the assumption, we introduce the following maxima.


\begin{definition}
\[ M_1 = \max_{1 \le i \le n} \max_{1 \le l \le q} | Z_{il} u_i |,\]
\[ M_2 = \max_{ 1 \le i \le n} \max_{1 \le j \le p} \max_{1 \le l \le q} |Z_{il} X_{ij}- E Z_{il} X_{ij} |.\]
\[ M_3 = \max_{1 \le i \le n} \max_{1 \le l \le q} |Z_{il}^2 u_i^2 - E Z_{il}^2 u_i^2|.\]
\[M_4 = \max_{1 \le i \le n} \max_{1 \le j \le p} \max_{1 \le l \le q} |Z_{il}^2 u_i X_{ij} - E Z_{il}^2 u_i X_{ij}|.\]
\[ M_5 = \max_{1 \le i \le n} \max_{1 \le j \le p} \max_{1 \le l \le q} |Z_{il}^2 X_{ij} X_{il} - E Z_{il}^2 X_{ij} X_{il}|.\] 
\end{definition}
\begin{assum}\label{2}

(i). \[ s_0^2 \sqrt{\frac{\ln q}{n}} \to 0.\]

(ii).  
\[ \frac{\sqrt{\ln q}}{\sqrt{n}} \max [(E M_1^{2})^{1/2},  (E M_2^{2})^{1/2},(E M_3^{2})^{1/2},(E M_4^{2})^{1/2}
,(E M_5^{2})^{1/2}] \to 0.\]

\end{assum}
Assumption \ref{2} restricts the number of non-zero entries of $\beta_0$. No sparsity is assumed on the instruments. However, the dimensionality of the model, as measured by the number of instruments $q$, does influence how fast $s_0$ can increase. Part (ii) is similar to assumptions made in \cite{clc17} in the context of establishing the validity of cross validation to choose the tuning parameter in the context of a linear regression mode with exogenous regressors. 
Essentially, Assumption \ref{2} (ii) restricts the growth rate of the second moments of the maxima $M_1,...,M_5$. Note that one can provide primitive sufficient conditions for Assumption \ref{2}(ii) by introducing conditions on the maximum of moments rather than moments of maximum of random variables as in Assumption \ref{2}(ii). 


%

\begin{thm}\label{thm1}
Under Assumptions \ref{1} and \ref{2}, with $r_z \ge 12, r_x \ge 6, r_u \ge 6$, we have 
with probability at least $1 -  \frac{21}{q^C} - \frac{K [ 5 E M_1^{2}+ 10 E M_2^{2} + 2E M_3^{2} + 2E M_4^{2} + 2 E M _5^{2}]}{n \ln q}$


(i).\[ \| \hat{\beta} - \beta_0 \|_1 \le \frac{24 \lambda_n^*  s_0}{ \phi_{\Sigma_{xzw}}^2 (s_0)}\]
for $n$ sufficiently large. The above bound is valid uniformly over ${\cal B}_{l_0}(s_0) = \{ \|\beta_0 \|_{l_0} \le s_0 \}$. 

(ii). Furthermore, $\lambda_n^* = O (\sqrt{\frac{\ln q}{n}})$ and the probability of (i) being valid tends to one. 
\end{thm}

\textbf{Some remarks}

1. While we mainly use Theorem \ref{thm1} as a stepping stone towards testing hypotheses about the elements of $\beta_0$, it may be of interest in its own right as it guarantees that the two-step GMM estimator estimates $\beta_0$ precisely. In particular, we see that $\| \hat{\beta} - \beta_0 \|_1=O_p(s_0\sqrt{\ln q/n})$ allowing $q$ to increase very quickly in $n$ without sacrificing consistency of $\hat{\beta}$ since the number of instruments only enters the upper bound on the $l_1$ estimation error through its logarithm.

2. Lemma 4.4 of \cite{gold2017inference} and the line immediately below it provide an upper bound on the $l_1$ estimation error of their estimator which is of order $O_p (s_0 s_{A}^2 \ln q/n + s_0 s_A \sqrt{\ln q/n})$ where $s_A<q$ is the number of relevant instruments and the remaining quantities are as in the present paper. Recall that we do not impose any restrictions on the number of relevant instruments. 

3. Note that the result of part (i) of Theorem \ref{thm1} is uniform over the $l_0$ ball ${\cal B}_{l_0} (s_0)=\cbr[0]{x\in\mathbb{R}^p: ||x||_{l_0}\leq s_0}$ since the bounds depend on $\beta_0$ only through $s_0$. In particular, the non-zero entries of $\beta_0$ can drift to zero at any rate. 


Having established that the two-step GMM estimator estimates $\beta_0$ precisely, we turn towards desparsifying it in order to construct tests and confidence intervals. Furthermore, in Section \ref{recipe}, we provide a choice of tuning parameters.

\section{Desparsification}\label{sec:desp1}
\subsection{The desparsified two-step GMM estimator}\label{sec:desp}
We now introduce the desparsified two-step GMM estimator that we use to construct tests and confidence intervals. To this end, consider the Karush-Kuhn-Tucker first order conditions for the problem in (\ref{eq:1}) 
\begin{equation}
\frac{-X'Z}{n} \frac{\hat{W}_d}{q} \frac{Z' (Y - X \hat{\beta})}{n} + \lambda_n^* \hat{\kappa} = 0, \label{3}
\end{equation}
where $\| \hat{\kappa} \|_{\infty} \le 1$, and $\hat{\kappa}_j = sgn (\hat{\beta}_j)$ when $\hat{\beta}_j \neq 0$. for $j=1,...,p$. Since $Y = X \beta_0+u$,  
\begin{equation}
\left[ \frac{X'Z}{n} \frac{\hat{W}_d}{q} \frac{Z'X}{n}\right] (\hat{\beta} - \beta_0) + \lambda_n^* \hat{\kappa} = 
\frac{X'Z}{n} \frac{\hat{W}_d}{q} \frac{Z'u}{n}.\label{eq:4}
\end{equation}
Next, since $\hat{\Sigma}:=\left[ \frac{X'Z}{n} \frac{\hat{W}_d}{q} \frac{Z'X}{n}\right]$ is of reduced rank, it is not possible to left-multiply by its inverse in the above display in order to isolate $\sqrt{n}(\hat{\beta}-\beta_0)$. Instead, we construct an approximate inverse, $\hat{\Gamma}$, of $\hat{\Sigma}$ and control the error resulting from this approximation. We shall be explicit about the construction of $\hat{\Gamma}$ in the sequel (cf Section \ref{clime}) but first highlight which properties it must have in order to conduct asymptotically valid inference based on it. Left multiply (\ref{eq:4}) by $\hat{\Gamma}$ to obtain
\begin{equation}
\hat{\Gamma}\left[ \frac{X'Z}{n} \frac{\hat{W}_d}{q} \frac{Z'X}{n}\right] (\hat{\beta} - \beta_0) +\hat{\Gamma} \lambda_n^* \hat{\kappa} = 
\hat{\Gamma}\left( \frac{X'Z}{n} \frac{\hat{W}_d}{q} \frac{Z'u}{n} \right).\label{5}
\end{equation}
Add $(\hat{\beta} - \beta_0)$ to both sides of (\ref{5}) and rearrange to get 
\begin{equation}
(\hat{\beta} - \beta_0)+\hat{\Gamma} \lambda_n^* \hat{\kappa} = 
\hat{\Gamma}\left(\frac{X'Z}{n} \frac{\hat{W}_d}{q} \frac{Z'u}{n}\right) -  \left(  \hat{\Gamma}\left[ \frac{X'Z}{n} \frac{\hat{W}_d}{q} \frac{Z'X}{n}\right]- I_p \right)  (\hat{\beta} - \beta_0) .\label{7}
\end{equation}
Upon defining $\Delta  = \sqrt{n} \left(  \hat{\Gamma}\left[ \frac{X'Z}{n} \frac{\hat{W}_d}{q} \frac{Z'X}{n}\right]- I_p \right)  (\hat{\beta} - \beta_0)$, which can be interpreted as the approximation error due to using an approximate inverse of $\left[ \frac{X'Z}{n} \frac{\hat{W}_d}{q} \frac{Z'X}{n}\right]$ instead of an exact inverse, (\ref{7}) can also be written as  
\begin{equation}
(\hat{\beta} - \beta_0)+\hat{\Gamma} \lambda_n^* \hat{\kappa} = 
\hat{\Gamma}\left(\frac{X'Z}{n} \frac{\hat{W}_d}{q} \frac{Z'u}{n}\right) - \frac{\Delta}{\sqrt{n}}.\label{8}
\end{equation}
Thus, 
\begin{align*}
\hat{\beta}=\beta_0-\hat{\Gamma} \lambda_n^* \hat{\kappa}+\hat{\Gamma}\left(\frac{X'Z}{n} \frac{\hat{W}_d}{q} \frac{Z'u}{n}\right) - \frac{\Delta}{\sqrt{n}} 
\end{align*}
where $\hat{\Gamma} \lambda_n^* \hat{\kappa}$ is the shrinkage bias introduced to $\hat{\beta}$ due to penalization in (\ref{eq:1}). By removing this, we define the \emph{two-step desparsified GMM} estimator
\begin{align}
\hat{b}
=
\hat{\beta} + \hat{\Gamma} \lambda_n^* \hat{\kappa}
=
\beta_0 + \hat{\Gamma}\left(\frac{X'Z}{n} \frac{\hat{W}_d}{q} \frac{Z'u}{n}\right) - \frac{\Delta}{\sqrt{n}}\label{10}.
\end{align} 
Note that by (\ref{3}) one can calculate $\hat{b}$ in terms of observable quantities as
\begin{align}
\hat{b}=\hat{\beta}+\hat{\Gamma}\frac{X'Z}{n} \frac{\hat{W}_d}{q} \frac{Z' (Y- X \hat{\beta})}{n}.\label{13a}
\end{align}

Thus, to conduct inference on the $j$th component of $\beta_0$ we consider
\begin{align}
\sqrt{n}(\hat{b}_j-\beta_{0j})
=
\sqrt{n}e_j'(\hat{b}-\beta_0)
=
\hat{\Gamma}_j\left(\frac{X'Z}{n} \frac{\hat{W}_d}{q} \frac{Z'u}{\sqrt{n}}\right) - \Delta_j\label{eq:inf}
\end{align}
where $\hat{\Gamma}_j$ denotes the $j$th row of $\hat{\Gamma}$. Hence, in order to conduct asymptotically valid gaussian inference, it suffices to establish a central limit theorem for $\hat{\Gamma}_j\left(\frac{X'Z}{n} \frac{\hat{W}_d}{q} \frac{Z'u}{\sqrt{n}}\right)$ as well as asymptotic negligibility of $\Delta_j$. To achieve these two goals, we need to construct an approximate inverse $\hat{\Gamma}$ and develop its properties. The subsequent subsection is concerned with these issues.

\subsection{Constructing $\hat{\Gamma}$}\label{clime}
In Section \ref{sec:desp} we assumed the existence of a an approximate inverse $\hat{\Gamma}$ of $\hat{\Sigma}=   \frac{X'Z}{n} \frac{\hat{W}_d}{q} \frac{Z'X}{n}$. We now turn towards the construction of $\hat{\Gamma}$. Our construction builds on the CLIME estimator of \cite{cai11} (which was further refined in \cite{glt18}). We establish how our estimator can provide a valid approximate inverse allowing for conditional heteroskedasticity. First, define
\begin{align*}
\Sigma = \Sigma_{xz} \frac{W_d}{q} \Sigma_{xz}',
\end{align*}
as well as its inverse $\Gamma=\Sigma^{-1}$, which is guaranteed to exist by Assumption \ref{5cl} below.
We shall assume that for some $0\leq f<1,\ m_{\Gamma}>0$ and $s_{\Gamma}>0$,
\begin{align*}
\Gamma
\in
U (m_{\Gamma}, f, s_{\Gamma}) 
:= 
\cbr[2]{A\in\mathbb{R}^{p\times p}: A > 0, \enVert[0]{A}_{l_1} \le m_{\Gamma}, \max_{1 \le j \le p} \sum_{k=1}^p | A_{jk}|^f \le s_{\Gamma}},
\end{align*} 
where $m_{\Gamma}$ and $s_{\Gamma}$ regulate the sparsity of the matrices in $U (m_{\Gamma}, f, s_{\Gamma})$ and their potential dependence on $n$ is suppressed (in particular we allow $s_\Gamma,m_\Gamma\to\infty$). Note that $f=0$ amounts to assuming that $\Gamma$ has exactly sparse rows.

Our proposed approximate inverse for $\hat{\Sigma}$ is found by the following variant of the CLIME estimator: The $j$th row of $\hat{\Gamma}$, denoted $\hat{\Gamma}_j$, is found as
\begin{align}
 \hat{\Gamma}_j=\argmin_{a\in \mathbb{R}^p} \|a \|_1 \quad \text{s.t.} \quad \| a\hat{\Sigma} - e_j' \|_{\infty} \le \mu,\label{eq:CLIME}
\end{align}
where $\mu>0$ and the dependence of $\hat{\Gamma}_j$ on $\mu$ is suppressed. The exact expression for $\mu$ is involved and given in the statement of Lemma \ref{cl-l3} in the appendix which also establishes that $\mu=O(m_{\Gamma}s_0 \frac{\sqrt{\ln q}}{\sqrt{n}})=o(1)$, cf. (\ref{rocmu}). Furthermore, Lemma \ref{cl-l3} in the appendix shows that with probability converging to one, $ \enVert[0]{\Gamma \hat{\Sigma} -I_p}_{\infty} \le \mu$ implying that the problem in (\ref{eq:CLIME}) is well-defined (with probability approaching one) since $\Gamma$ satisfies the constraint. This is noteworthy since $\Gamma$ is not required to be strictly sparse.

\begin{assum}\label{4}

(i). $\Gamma \in U (m_{\Gamma}, f, s_{\Gamma})$.

(ii). $m_{\Gamma} s_0 \frac{\sqrt{\ln q}}{\sqrt{n}}= o(1).$

\end{assum}
Assumption \ref{4} (i) restricts the structure of $\Gamma$ by imposing it to belong to $U (m_{\Gamma}, f, s_{\Gamma})$. Note that for $0<f<1$, $\Gamma$ is not required to be (exactly) sparse as opposed to much previous work. Part (ii) restricts the growth rates of $m_\Gamma, s_0$ and $q$. Note that by Assumption \ref{4}, we restrict the growth rate of the absolute sums of the coefficients in rows of the precision matrix $\Gamma$. This will also restrict the relation between $X_i, Z_i,$, but we could not come up with a primitive on the data for Assumption \ref{4}.



\section{Testing and uniformly valid confidence intervals}\label{inferen}
In this section we show how to conduct asymptotically valid gaussian inference on each entry of $\beta_0$. It is a technical exercise to extend this to joint inference, by e.g. Wald-type tests, on any fixed and finite number of elements of $\beta_0$. At the price of more technicalities and more stringent assumptions we also conjecture that it is possible to conduct joint inference on a subvector of $\beta_0$ of slowly increasing dimension. We refer to \cite{ck18} for details in the case of the conservative Lasso applied to the high-dimensional plain linear regression model and do not pursue these extensions further here. To conduct inference on $\beta_{0j}$ we consider the studentized version of (\ref{eq:inf}):

\begin{equation}
t_{W_d} = \frac{n^{1/2} e_j' (\hat{b} - \beta_0)}{\sqrt{e_j' \hat{\Gamma} \hat{V}_d \hat{\Gamma}' e_j}},\label{5.0}
\end{equation}
where 
\begin{align}
\hat{V}_d = \left( \frac{X'Z}{n} \frac{\hat{W}_d}{q} \hat{\Sigma}_{Zu} \frac{\hat{W}_d}{q} \frac{Z'X}{n}   \right) \quad \text{and} \quad \hat{\Sigma}_{Zu} = \frac{1}{n} \sum_{i=1}^n Z_i Z_i' \hat{u}_i^2, \label{5.1}
\end{align}
and $\hat{u}=Y-X\hat{\beta}_F$ are the first step Lasso-GMM residuals.

To state the next assumption, define the $q\times q$ matrix $\Sigma_{z u} =  E  Z_1  Z_1' u_1^2$ as well as the $p\times p$ matrices $V_1 = \Sigma_{x z} W_d  \Sigma_{z u }  W_d \Sigma_{x z}'$ and $V_d=\frac{1}{q^2}V_1$. Finally, let 
\[ M_6 = \max_{1 \le i \le n} \max_{1 \le j \le p}| X_{ij} u_i - E X_{ij} u_i|,\]
\[ M_{7} = \max_{1 \le i \le n} \max_{1 \le l \le q} \max_{1 \le m \le q} |Z_{il} Z_{im} u_i^2 - E Z_{il} Z_{im} u_i^2|.\]

In order to establish the asymptotic normality of $\sqrt{n}(\hat{b}_j-\beta_{0j})$ when $\hat{\Gamma}$ is the CLIME estimator in (\ref{eq:CLIME}), we impose the following assumptions.

\begin{assum}\label{5cl}
(i). $m_{\Gamma}^{r_u/2}/n^{r_u/4-1} \to 0$ and
\[ s_{\Gamma} (m_{\Gamma} \mu)^{1-f} \sqrt{\ln q} = O \left( \frac{s_{\Gamma} m_{\Gamma} ^{2-2f}s_0^{1-f} (\ln q)^{1 - f/2}}{n^{(1-f)/2}}\right) = o(1).\]

In addition, $ \frac{m_{\Gamma}  s_0^2  \ln q}{\sqrt{n}} = o(1)$.

(ii) $Eigmax (V_d)$ and $Eigmax (\Sigma)$ are bounded from above. $Eigmin (V_d)$ and $Eigmin(\Sigma)$ are bounded away from zero

(iii). $\sqrt{\frac{\ln q}{n}} \max[ (E M_6^{2})^{1/2}, ( E M_{7}^{2})^{1/2}] \to 0.$

(iv). $r_z > 12$ and
$\frac{ m_{\Gamma}^2  s_0 q^{4/r_z} n^{2/r_z} \sqrt{\ln q}}{n^{1/2}} \to 0.$

\end{assum}
Assumption \ref{5cl} governs the permissible growth rates of the number of instruments, $q$, the sparsity imposed on $\Gamma$ via $s_{\Gamma}$ and $m_{\Gamma}$, as well as the number of non-zero entries in $\beta_0$, $s_0$. For example, assuming that $f=1/2$, $r_u=10$ and $r_z=16$ along with $s_{\Gamma} = O (\ln n),\ m_{\Gamma} = O (\ln n),\ s_0 = n^{1/10}, q = 2n$ is in accordance with Assumption \ref{5cl}. Assumption \ref{5cl}(ii) is a standard assumption on population matrices. Note that the requirement $Eigmin(\Sigma)$ being bounded away from zero implies the adaptive restricted eigenvalue being bounded away from zero (as required in Assumption \ref{1}). The considered largest and smallest eigenvalues can be allowed to be unbounded and approach zero, respectively, at the expense of strengthening other assumptions. Thus, instruments that are weakly correlated with the explanatory variables can be allowed for. Part (iii) is similar to assumptions imposed in \cite{clc17}, cf. also the discussion of Assumption \ref{2} above.


The following theorem establishes the validity of asymptotically gaussian inference for the desparsified two-step GMM estimator.

\begin{thm}\label{thmcl1}
Let $j\in\cbr[0]{1,...,p}$. Then, under Assumptions \ref{1},\ref{2},\ref{4} and \ref{5cl} with $r_z  > 12, r_x \ge 6, r_u > 8$ 

(i). \begin{align}
\frac{n^{1/2} (\hat{b}_j - \beta_{0j})}{\sqrt{e_j' \hat{\Gamma} \hat{V}_d \hat{\Gamma}' e_j}} \stackrel{d}{\to} N(0,1)\quad \text{uniformly over  } \beta_0\in {\cal B}_{l_0}(s_0)\label{eq:test} 
\end{align}

(ii). \[ \sup_{\beta_0 \in {\cal B}_{l_0}(s_0)} | e_j' \hat{\Gamma} \hat{V}_d \hat{\Gamma}' e_j - e_j' \Gamma V_d \Gamma' e_j | = o_p (1),\].
\end{thm}
Part (i) of Theorem \ref{thmcl1} establishes asymptotic normality of $\hat{b}_j$ for every $j\in\cbr[0]{1,...,p}$  uniformly over $\beta_0\in\mathcal{B}_{l_0}(s_0)$. 

Part (ii) of Theorem \ref{thmcl1} provides a uniformly consistent estimator of the asymptotic variance of $n^{1/2} (\hat{b}_j - \beta_{0j})$. This is valid even for conditionally heteroskedastic $u_i,\ i=1,...,n$ and even though the dimension ($p\times p$) of the involved matrices diverges with the sample size.

Next, we show that the confidence bands resulting from Theorem \ref{thmcl1} have asymptotically uniformly correct coverage over $\mathcal{B}_{l_0}(s_0)$. Furthermore, the bands contract uniformly at the optimal $\sqrt{n}$ rate. Let $\Phi(\cdot)$ denote the cdf of the standard normal distribution and let $z_{1 - \alpha/2}$ be its $1-\alpha/2$ quantile. For brevity, let $\hat{\sigma}_{bj}=
\sqrt{e_j' \hat{\Gamma} \hat{V}_d \hat{\Gamma}' e_j}$ while $diam([a,b])=b-a$ denotes the length of the interval $[a,b]$ in the real line.
\begin{thm}\label{thm2c}
Let $j\in\cbr[0]{1,...,p}$. Then, under Assumptions \ref{1}-\ref{5cl} with $r_z  > 12, r_x \ge 6, r_u > 8$ 

(i). \[ \sup_{t \in \mathbb{R}} \sup_{ \beta_0 \in {\cal B}_{l_0}(s_0)} \left| P \left( 
\frac{n^{1/2} (\hat{b}_j - \beta_{0j})}{\sqrt{e_j' \hat{\Gamma} \hat{V}_d \hat{\Gamma}' e_j}} \le t \right) - \Phi (t) \right| \to 0
.\]

(ii). \[ \lim_{n \to \infty} \inf_{\beta_0 \in {\cal B}_{l_0}(s_0)} P \left( \beta_{0j} \in 
[ \hat{b}_j - z _{1 - \alpha/2} \frac{\hat{\sigma}_{bj}}{n^{1/2}}, \hat{b}_j + z_{1 - \alpha/2} \frac{\hat{\sigma}_{bj}}{n^{1/2}}] \right)
= 1 - \alpha.\]

(iii). \[ \sup_{\beta_0 \in {\cal B}_{l_0}(s_0)} diam \left( [ \hat{b}_j - z_{1 - \alpha/2} \frac{\hat{\sigma}_{bj}}{n^{1/2}}, \hat{b}_j + 
z_{1 - \alpha/2} \frac{\hat{\sigma}_{bj}}{n^{1/2}}]\right) = O_p \del[2]{\frac{1}{n^{1/2}}}.\]
\end{thm}
Part (i) of Theorem \ref{thm2c} asserts the uniform convergence to the normal distribution of the properly centered and scaled $\hat{b}_j$. Part (ii) is a consequence of  (i) and yields the asymptotic uniform validity of confidence intervals based on $\hat{b}_j$. Finally, (iii) asserts that the confidence intervals contract at rate $\sqrt{n}$ uniformly over $\mathcal{B}_{l_0}(s_0)$. We also stress that the above results do \emph{not} rely on a $\beta_{\min}$-type condition requiring the non-zero entries of $\beta_0$ to be bounded away from zero.  In other words, all our results allow for local to zero structural parameters.

\subsection{Linear dynamic panel data models as a special case}\label{sec:dynpan}
In this section we show how the classic dynamic linear panel data model as studied in, e.g., \cite{arellano1991some} is covered by our framework as a special case upon taking first differences. To be precise, we consider the model
\[ y_{it} = \rho_0 y_{i t-1} + x_{it}' \delta_0 + \mu_i + u_{it},\ i=1,..., n,\ t=1,..., T\]
where $|\rho_0|<1$. $y_{it}$ is a scalar, $x_{it}$ is a $K \times 1$ vector of strictly exogenous variables and $\mu_i$ is the unobserved effect for individual $i$ which can be correlated with $y_{i t-1}$ and $x_{it}$. Assume, for concreteness, that $y_{i0}=0$ for $i=1,...,n$. Since $(\rho_0,\delta_0)$ is the parameter of interest we have $p=K+1$ in the terminology of our paper. The $\mu_i$ can be removed by taking first differences, arriving at
\begin{equation}
\Delta y_{it} = \rho_0 \Delta y_{i t-1} + \Delta x_{it}' \delta_0 + \Delta u_{it},\ i=1,...,n,\ t=2,....,T,\label{dp1}
\end{equation}
Upon stacking the observations across individuals and time, (\ref{dp1}) is of the form (\ref{eq:mod}). Next, imposing 
\begin{align}
E [ u_{it} | \mu_i, y_i^{t-1}, x_i^T]=0,\ i=1,...,n,\ t=1,...,T,\label{eq:dynpanmom}
\end{align}
where $y_i^{s} = (y_{i1},..., y_{is})'$  and $x_i^T = (x_{i1}',..., x_{iT}')'$ implies that for each $i\in\cbr[0]{1,...,n}$ 
\begin{equation}
 E [ y_i^{t-2} \Delta u_{it} ]= 0,\ t=3,...,T,\label{dp2}
 \end{equation}
 \begin{equation}
 E [x_{it} \Delta u_{is	} ] = 0,\ t=1,...,T, s=2,....,T.\label{dp3}
 \end{equation}
This results in $q=(T-2)(T-1)/2+T(T-1)K$ moment inequalities for each $i=1,...,n$ thus fitting into (\ref{eq:GMMcond}). In particular we note that the number of instruments $q$ can be larger than the sample size $n(T-2)$ even for moderate values of $T$ and $K$ thus resulting in a setting with many moments/instruments compared to the number of observations as studied in this paper

\section{Tuning Parameter Choice}\label{tuning}
In this section we explain how we choose the tuning parameter sequences $\lambda_n$ and $\lambda_n^*$. We use cross validation since this has recently been shown by \cite{clc17} to result in Lasso estimators with guaranteed low finite sample estimation and prediction error in the context of high-dimensional linear regression models. While the theoretical guarantees are for the linear regression model without endogenous regressors and sub-gaussian error terms we still use cross validation here and are content to leave the big task of establishing theoretical guarantees of cross validated two-step desparsified Lasso GMM in the presence of endogenous regressors for future work.

The exact implementation of the cross validation used is as follows: Let $\hat{W}\in\cbr[0]{I_q, \hat{W}_d}$ (indicating whether the first step or second step Lasso GMM estimator is used) and $K\in\mathbb{N}$ be the number of cross validation folds. Assuming, for simplicity, that $n/K$ is an integer we let $I_k=\cbr[0]{\frac{k-1}{K}n +1,\frac{k}{K}n},\ k=1,...,K$ be a partition of $\cbr[0]{1,...,n}$ consisting of ``consecutive'' sets. Fix a $\lambda\in\Lambda_n\subseteq \mathbb{R}$ where $\Lambda_n$ is the candidate set of tuning parameters. With $\underline{n}$ being the cardinality of the $I_k$ define 
\begin{equation}
 \hat{\beta}_{-k} (\lambda) = 
\argmin_{b \in \mathbb{R}^p}  \left\{  \left[  \frac{1}{n- \underline{n}} \sum_{ i \notin I_k}  Z_i ( Y_i - X_i' b)\right]' \frac{\hat{W}}{q} \left[  \frac{1}{n- \underline{n}} \sum_{ i \notin I_k} Z_i ( Y_i - X_i' b) \right] + \lambda  \| b \|_1 \right\}, \label{6.1}
\end{equation}
for $k=1,...,K$ and choose $\lambda$ as
\begin{equation}
 \hat{\lambda}_{CV} = \argmin_{ \lambda  \in \Lambda_n} \sum_{k=1}^K 
\sbr[2]{ \sum_{i \in I_k}  Z_i (Y_i - X_i'  \hat{\beta}_{-k}  (\lambda)}' \frac{\hat{W}}{q}  \sbr[2]{ \sum_{ i \in I_k} Z_i (Y_i - X_i'  \hat{\beta}_{-k} (\lambda)) } .\label{6.2}
\end{equation}
The concrete choices of $K$ and $\Lambda_n$ are given in Section \ref{mc}.

We also tried a  a modified BIC to choose $\lambda_n^*$. Recently, only in the case of $n>q>p$, where $q,p$ grow with sample size, \cite{chl18} provided a theorem that shows that this type of choice leads to selection consistency with adaptive elastic net penalty for GMM. However, here we consider $q>p>n$ and the Lasso so this result does not apply. Our preliminary simulations showed poor performance for the modified BIC in our context so we did not pursue it further.

We also used a tuning parameter choice based on our theorems. We start with first step GMM estimator. The issue is where we use this first step gmm, and tuning parameter associated with that in proofs. Our interest is in inference. There are two main issues with tuning parameter. First one: it provides an upper bound on the noise term in the oracle inequality with probability approaching one. The second one is: it has to show that asymptotic bias of the first step GMM lasso estimates converge in probability to zero. In that sense there is a tradeoff. The first issue forces choice of tuning parameter to be large and the second bias concern forces the tuning parameter to be small. In our preliminary simulation, we tried a theoretically oriented choice of tuning parameter that satisfy these two criteria, but this also resulted in poor finite sample results in inference, so we did not recommend that.

\subsection{ A Recipe for  Using High Dimensional Linear GMM}\label{recipe}

In this subsection we provide a step by step guide to implement the de-biased estimator. The recipe consists of three parts. In the part A we provide steps to implement Lasso-GMM, then in part B, we show how to use CLIME, and in part C we form the de-biased GMM and set up the test statistic.

{\bf A: Lasso-GMM}

1. Use (\ref{6.1}) and (\ref{6.2}) with $\hat{W}= I_q$ to choose the tuning parameter $\hat{\lambda}_{CV}$ for the first step Lasso GMM estimator.

2. Use (\ref{eq:FSGMM}) with $\hat{\lambda}_{CV}$ from step 1 to get $\hat{\beta}_F$.

3. Define the residuals $\hat{u}_i := Y_i - X_i' \hat{\beta}_F$, for $i=1,\cdots, n$, and define for each $l=1,\cdots, q$
$\hat{\sigma}_l^2 := \frac{1}{n} \sum_{i=1}^n Z_{il} \hat{u}_i^2$; here $Z_{il}$ is the $l$th instrument in $i$ th cross section unit. Form the diagonal 
$q \times q$ matrix
\[ \hat{W}_d = diag (\frac{1}{\hat{\sigma}_1^2}, \cdots, \frac{1}{\hat{\sigma}_l^2}, \cdots, \frac{1}{\hat{\sigma}_q^2}).\]

4. Use (\ref{6.1}) and (\ref{6.2}) with $\hat{W}= \hat{W}_d$ to choose the tuning parameter $\hat{\lambda}_{CV}^*$ for two-step Lasso GMM estimator.

5. Use (\ref{eq:1}) with $\hat{\lambda}_{CV}^*$ from step 4 to get $\hat{\beta}$.

{\bf B: CLIME}

1. For $j=1,\cdots, p$, set 
\[ \hat{\mu}_j = 1.2 inf_{a \in R^p} \| a \hat{\Sigma} - e_j' \|_{\infty},\]
where the minimization problem is solved by the MOSEK optimizer.

2. After obtaining $\hat{\mu}_j$ for all $j=1,\cdots, p$, solve (\ref{eq:CLIME}) to obtain $\hat{\Gamma}_j$ to form $\hat{\Gamma}$ from the rows 
$\hat{\Gamma}_j$. This is done by the MOSEK optimizer.

{\bf C: De-Biased GMM}

1. Calculate $\hat{b}$ as in (\ref{13a}).

2. Calculate $\hat{V}_d$ as in (\ref{5.1}).

3. Form the test statistic as in (\ref{5.0}) to test the desired hypothesis.

\section{Monte Carlo}\label{mc}
In this section we investigate the finite sample properties of the desparsified two-step GMM Lasso (DGMM) estimator and compare it to the desparsified two stage least squares (D2SLS) estimator of \cite{glt18}. All designs are repeated $B=100$ times as the procedures are computationally demanding. Before discussing the results, we explain how the data was generated and the performance measures used to compare DGMM to D2SLS.

\subsection{Implementation details}
The implementation of the D2SLS of \cite{glt18} is inspired by the publicly available code at \texttt{https://github.com/\allowbreak LedererLab/\allowbreak HDIV/blob/master/\allowbreak src/estimation.r}. 
We use five fold cross validation, $K=5$, to select $\lambda_n$ and $\lambda_n^*$. $\Lambda_N$ is chosen by the \texttt{glmnet} package in R. As in \cite{glt18} we choose  $\mu_j =  1.2\cdot\inf_{a \in \mathbb{R}^p} \|  a\hat{\Sigma} - e_j' \|_{\infty},$ for ($a$ is a row vector) $j=1,2,...,p$ and the minimization problem is solved by the MOSEK optimizer for R, \cite{mosek}.  We also tried replacing 1.2  by 1.5 in the calculation of $\mu_j$ but this did not affect the inferential results. 

\subsection{Design 1}
This design is inspired by the heteroskedastic design in \cite{chl18}. We choose
\begin{align*}
Z_i\sim N_q(0,\Omega) \quad \text{with}\quad \Omega_{j,k}=\rho_z^{|j-k|} 
\end{align*}
for $\rho_z=0.5$ and set
\begin{align*}
X_i=\pi'Z_i +v_i
\end{align*}
where, for $\iota_k$ being a $k\times 1$ vector of ones, we choose the $q\times p$ matrix $\pi = [ 2+ 2 \rho_z^{q/2}]^{-1/2} ( \iota_2 \otimes I_{q/2})$. Thus, $p=q/2$. Furthermore,
\begin{align*}
Y_i=X_i'\beta_0+u_i
\end{align*}
with $\beta_0=(1,1,0_{p-8}',0.5,0_5')'$ and $0_k$ being a $k\times 1$ vector of zeros.  Thus, $\beta_0$ has three non-zero entries. The following notation is introduced to define the joint distribution of $v_i$ and $u_i$: Let $\epsilon_i=(\epsilon_{i1},\epsilon_{i2}, \epsilon_{i3})\sim N_{p+2}(0, I_{p+2})$ where $\epsilon_{i1}$ and $\epsilon_{i2}$ are scalars and $\epsilon_{i3}$ is $p\times 1$. 
\begin{align*}
\tilde{u}_i = \sqrt{\rho_{uv}} \epsilon_{1i} + \sqrt{1 - \rho_{uv}} \epsilon_{2i} \qquad \text{and}\qquad v_i = \sqrt{\rho_{uv}} \epsilon_{1i} \iota_p + \sqrt{1 - \rho_{uv}} \epsilon_{3i}
\end{align*}
with $\rho_{uv}=0.25$. Then, in order to introduce conditional heteroskedasticity into $u_i$, we set $u_{i} = \tilde{u}_i \| Z_i \|_2/\sqrt{q}$. The following combinations of $n,p$ and $q$ are considered:
\begin{align*}
(n,p,q)\in \bigl\{&(50,50,100),(75,50,100),(75,10,20), (75,100,200), (150,100,200), (150,10,20), (150,200,400),\\ &(300,200,400), (300,10,20)\bigr\}.
\end{align*}
Note that these designs are in three categories: i) many moments/instruments and variables; $q>p>n$, ii) many moments/instruments; $q>n\geq p$ and iii) standard asymptotics; $n>q>p$. In the Tables \ref{tab:1}-\ref{tab:3} the results for these three settings can be found in columns i), ii) and iii), respectively.

%

%
%
%

\subsection{Design 2}
Everything is as in Design 1 except for $\pi=1_{q,p}/q$ where $\pi=1_{q,p}$ denotes a $q\times p$ matrix of ones. Thus, all instruments are (weakly) relevant. 

\subsection{Design 3}
Everything is as in Design 1 except for $\pi=(0.25\cdot 1_{p,q/4},0_{p,3/4\cdot q})'$ where $1_{p,q/4}$ is a $p\times q/4$ matrix of ones and $0_{p,3/4\cdot q}$ is a $p\times 3/4\cdot q$ matrix of zeros.

\subsection{Performance measures}

The performance of D2SLS and DGMM are measured along the following dimensions.
\begin{enumerate}
\item Size:  The size of the test in (\ref{eq:test}) is gauged by considering a test on $\beta_{0,2}$ as in applied work interest often centers on a single coefficient (of the policy variable). The null hypotheses is always that this coefficient equals the true value assigned to (here the true value is always one). The nominal size of the test is 5\%.
\item Power: To gauge the power of the test we test whether $\beta_{0,j}$ equals its assigned value plus $1/2$ in Design 1. In Designs 2 and 3 we test whether $\beta_{0,j}$ equals its assigned value (which is 1) plus 1.5. The difference in alternatives is merely to obtain non-trivial power comparisons (i.e. to avoid either the power of all tests being (very close to) zero or (very close to) one).
\item Coverage rate: Let  $\hat{C}_j(\alpha)=\sbr[1]{ 
\hat{b}_j - z _{1 - \alpha/2} \frac{\hat{\sigma}_{bj}}{n^{1/2}}, \hat{b}_j + z_{1 - \alpha/2}\frac{\hat{\sigma}_{bj}}{n^{1/2}}},\ j=1,...,p$ be the confidence intervals from Theorem \ref{thm2c}. We calculate the average coverage rate across all $p$ entries of $\beta_0$ and $B=100$ Monte Carlo replications. We use $\alpha=0.05$ throughout.
%
\item Length of confidence interval: We report the average length of the confidence intervals from Theorem \ref{thm2c} over all $p$ entries of $\beta_0$ and $B=100$ Monte Carlo replications.
\item MSE: We calculate the mean square error of $\hat{b}$ across all $B=100$ Monte Carlo replications
\end{enumerate}

%

\subsection{Results of simulations}
In this section we report the results of our simulation study.

\subsubsection{Design 1}
Table \ref{tab:1} contains the results of Design 1. Our DGMM procedure is oversized (size above 5\%) in 2 out of the 9 panels while the D2SLS is oversized in 1 out of 9 panels. In general, both procedures tend to be slightly undersized, however. Our DGMM procedure is non-inferior in terms of power in 7 out of 9 panels and achieves power advantages of up to 38\%-point. Both procedures always have at least 95\% coverage but the intervals produced by the DGMM procedure are more narrow in 7 out of 9 panels. Thus, the intervals are more informative.  Finally, the MSE of the  DGMM estimator is lower than the one of the D2SLS estimator in 8 out 9 panels; sometimes by more than a factor 10.

\subsubsection{Design 2}
Table \ref{tab:2} contains the results of Design 2. Recall that this is a setting where $\pi$ is not sparse. While our DGMM estimator is oversized in 2 out 9 panels, the D2SLS is oversized in 7 out of 9 panels with sizes of up to 75\%. Despite having generally lower size, the DGMM procedure has higher power than the D2SLS procedure in 8 out of 9 panels. Furthermore, the DGMM procedure does not exhibit undercoverage in terms of its confidence intervals for any of the 9 panels while D2SLS undercovers in 4 out 9 panels. The higher coverage of DGMM does not come at the price of longer confidence intervals as DGMM confidence intervals are always more narrow than the ones stemming from D2SLS. Finally, the MSE is always lower for DGMM.

\subsubsection{Design 3}
Table \ref{tab:3} contains the results of Design 3. This design strikes a middle ground between Designs 1 and 2 in terms of the sparsity of $\pi$. The tests based on DGMM and D2SLS are both oversized in 1 out of 9 panels. However, the former procedure results in more powerful tests than the latter in 6 out 9 panels. The largest power advantage of DGMM over D2SLS is 66\%-point while the largest advantage of D2SLS over DGMM is 18\%-point. The DGMM procedure always has at least 95\% coverage while this is the case for 8 out of 9 panels for the D2SLS procedure. However, the DGMM procedure has a tendency to overcover. This tendency is less pronounced for the D2SLS procedure. Despite this fact, the confidence intervals resulting from the DGMM procedure are shorter than the ones stemming from the D2SLS procedure in 5 out of 9 panels. The DGMM procedure always has lower MSE than D2SLS.

  \begin{table}[h]
\centering
\begin{tabular}{|c|cc|cc|cc|}
\toprule 
 & \multicolumn{6}{|c|}{Design 1}\\ \hline
\hline &\multicolumn{2}{c|}{$n=50,p=50,q=100$}
& \multicolumn{2}{c|}{$n=75,p=50,q=100$}& \multicolumn{2}{|c|}{$n=75, p=10, q=20$} \\ \hline
&D2SLS & D2GMM
& D2SLS & D2GMM& D2SLS  & D2GMM \\ \hline
Size  &17\%& 21\% &2\% & 17\%& 4\% & 2\%\\ 
Power &  7\% & 32\% &19\% & 47\%& 47\% & 46\%\\ 
  Coverage & 0.9670& 0.9574&0.9771& 0.9992& 0.9710 & 0.9820\\ 
  Length & 34.8783& 8.9940 &1.8287&0.9578& 0.9629 & 0.9786 \\ 
  MSE & 701.667&1.9198 &0.1985& 0.0315& 0.0545 & 0.0333 \\  \hline \hline
  &  \multicolumn{2}{c|}{$n=75,p=100,q=200$}
     & \multicolumn{2}{c|}{$n=150,p=100,q=200$} &\multicolumn{2}{|c|}{$n=150, p=10, q=20$}\\ \hline   
  &D2SLS & D2GMM & D2SLS & D2GMM & D2SLS & D2GMM \\ \hline 
Size & 0\% & 0\% &  1\% & 0\%& 2\% & 3\%\\ 
Power  & 40\%  & 78\% &36\% & 70\%& 95\% & 82\%\\ 
  Coverage & 0.9741 & 0.9897 &0.9640& 0.9957& 0.9720 & 0.9630\\ 
  Length & 1.6973 & 0.7616 &1.1276&0.8289& 0.5705 & 0.6255 \\ 
  MSE&  0.1615 & 0.0160 &0.0759& 0.0173 & 0.0169 & 0.0188 \\  \hline \hline
  & \multicolumn{2}{c|}{$n=150,p=200,q=400$} 
    & \multicolumn{2}{c|}{$n=300,p=200,q=400$} &\multicolumn{2}{|c|}{$n=300, p=10, q=20$}\\ \hline   
   & D2SLS & D2GMM
   & D2SLS & D2GMM & D2SLS & D2GMM\\ \hline
Size   & 2\% & 0\% &4\% & 0\% & 1\% & 3\%\\ 
Power  &  62\% & 97\% & 83\% & 97\%& 100\% & 100\%\\ 
  Coverage&  0.9709 & 0.9963 & 0.9537& 0.9973& 0.9650 & 0.9710\\ 
  Length & 0.9611 & 0.5521 &0.6792&0.5750& 0.3531 & 0.3941 \\ 
  MSE &  0.0477 & 0.0080 &0.0289& 0.0080 & 0.0065 & 0.0037\\   
  \hline   
\end{tabular}
\caption{}
 \label{tab:1}
\end{table}

  \begin{table}[h]
\centering
\begin{tabular}{|c|cc|cc|cc|}
\toprule 
 & \multicolumn{6}{|c|}{Design 2}\\ \hline
\hline 
&  \multicolumn{2}{c|}{$n=50,p=50,q=100$}
& \multicolumn{2}{c|}{$n=75,p=50,q=100$}& \multicolumn{2}{|c|}{$n=75, p=10, q=20$} \\ \hline
& D2SLS & D2GMM
& D2SLS & D2GMM& D2SLS  & D2GMM \\ \hline
Size & 56\% & 21\%  &57\% & 16\%& 32\% & 1\%\\ 
Power & 59\% & 36\%  &66\% & 93\%& 47\% & 66\%\\ 
  Coverage & 0.9572 & 0.9634 &0.9576& 0.9692& 0.9040 & 0.9890\\ 
  Length & 51.8943 & 14.0387 &57.6295&1.3173& 10.6126 & 2.3726 \\ 
  MSE & 9354.21 & 26.4915 &10472.80& 0.0670& 46.2285 & 0.2014 \\  \hline \hline
   &  \multicolumn{2}{c|}{$n=75,p=100,q=200$}
       & \multicolumn{2}{c|}{$n=150,p=100,q=200$} &\multicolumn{2}{|c|}{$n=150, p=10, q=20$}\\ \hline   
   & D2SLS & D2GMM 
   & D2SLS & D2GMM & D2SLS & D2GMM \\ \hline
Size & 61\% & 5\%  &74\% & 0\%& 9\% & 0\%\\ 
Power& 63\% & 100\%   &76\% & 100\%& 33\% & 85\%\\ 
  Coverage& 0.9523 & 0.9848   &0.9362& 0.9930& 0.9580 & 0.9890\\ 
  Length & 73.0592 & 0.9154  &72.2950&1.0717& 5.5107 & 2.2228 \\ 
  MSE & 30676.17 & 0.0303 &323884.4& 0.0339 & 15.1127 & 0.1712 \\  \hline \hline
  & \multicolumn{2}{c|}{$n=150,p=200,q=400$}   
  & \multicolumn{2}{c|}{$n=300,p=200,q=400$} &\multicolumn{2}{|c|}{$n=300, p=10, q=20$}\\ \hline   
   & D2SLS & D2GMM  
   & D2SLS & D2GMM & D2SLS & D2GMM\\ \hline
Size& 66\% & 2\%   &74\% & 0\% & 1\% & 0\%\\ 
Power& 67\% & 100\%   &75\% & 100\%& 63\% & 83\%\\ 
  Coverage& 0.8230 & 0.9921 &0.8126& 0.9947& 0.9830 & 0.9820\\ 
  Length& 106.34 & 0.6914 &145.48&0.7522& 2.9207 & 2.0897 \\ 
  MSE& 110527.2 & 0.0150 &137180.5& 0.0172 & 0.5160 & 0.1907\\   
  \hline 
\end{tabular}
\caption{}
 \label{tab:2}
\end{table}

  \begin{table}[h]
\centering
\begin{tabular}{|c|cc|cc|cc|}
\toprule 
& \multicolumn{6}{|c|}{Design 3-Semi-Sparse No of Instruments}\\ \hline
\hline
 & \multicolumn{2}{c|}{$n=50,p=50,q=100$}
 & \multicolumn{2}{c|}{$n=75,p=50,q=100$}& \multicolumn{2}{|c|}{$n=75, p=10, q=20$} \\ \hline
& D2SLS & D2GMM
& D2SLS & D2GMM& D2SLS  & D2GMM \\ \hline
Size & 8\% & 10\%  &1\% & 0\%& 1\% & 1\%\\ 
Power & 13\% & 19\%   &45\% & 71\%& 29\% & 68\%\\ 
  Coverage& 0.9404 & 0.9802 &0.9646& 0.9956& 0.9760 & 0.9860\\ 
  Length& 32.0919 & 14.1922 &3.5628&2.4496& 11.3496 & 2.3582 \\ 
  MSE& 550.36& 4.3972 &0.9067& 0.0912& 53.7831 & 0.2082 \\  \hline \hline
   & \multicolumn{2}{c|}{$n=75,p=100,q=200$}
      & \multicolumn{2}{c|}{$n=150,p=100,q=200$} &\multicolumn{2}{|c|}{$n=150, p=10, q=20$}\\ \hline   
   & D2SLS & D2GMM 
   & D2SLS & D2GMM & D2SLS & D2GMM \\ \hline
Size& 2\% & 0\%   &4\% & 0\%& 0\% & 0\%\\ 
Power& 99\% & 81\%   &78\% & 81\%& 18\% & 83\%\\ 
  Coverage& 0.9514 & 0.9995 &0.9706& 0.9999& 0.9740 & 0.9870\\ 
  Length& 1.0409 & 2.2356 &1.9492&2.2473& 9.0591 & 2.2711 \\ 
  MSE& 0.0715 & 0.0372 &0.2123& 0.0451 & 16.0593 & 0.1971 \\  \hline
  & \multicolumn{2}{c|}{$n=150,p=200,q=400$}
    & \multicolumn{2}{c|}{$n=300,p=200,q=400$} &\multicolumn{2}{|c|}{$n=300, p=10, q=20$}\\ \hline   
   & D2SLS & D2GMM
   & D2SLS & D2GMM & D2SLS & D2GMM\\ \hline
Size& 4\% & 0\%   &3\% & 0\% & 3\% & 2\%\\ 
Power& 100\% & 92\%   &98\% & 93\%& 19\% & 85\%\\ 
  Coverage& 0.9543 & 1.00&0.9701& 1.00& 0.9660 & 0.9730\\ 
  Length& 0.6188 & 2.2290&1.1898&2.1920& 7.3244 & 2.0423 \\ 
  MSE& 0.0266 & 0.0192 &0.0716& 0.0223 & 8.6713 & 0.2151\\   
  \hline 
\end{tabular}
\caption{}
 \label{tab:3}
\end{table}

\section{Conclusion}
This paper proposes a desparsified GMM estimator for estimating high-dimensional linear models with more endogenous variables than the sample size. The inference based on the estimator is shown to asymptotically uniformly valid even in the presence of conditionally heteroskedastic error terms. We do not impose the variables of the model to be sub-gaussian  nor do we impose sparsity on the instruments. Finally, our results are shown to apply also to linear dynamic panel data models. Future work includes investigating the effect of  the presence of (many) invalid instruments and potential remedies to this.


\setcounter{equation}{0}\setcounter{lemma}{0}\setcounter{assum}{0}\renewcommand{\theequation}{A.%
\arabic{equation}}\renewcommand{\thelemma}{A.\arabic{lemma}}%
\renewcommand{\theassum}{A.\arabic{assum}}%
\renewcommand{\baselinestretch}{1}\baselineskip=15pt

\section*{Appendix}
This appendix consists of three parts. The first part is related to Theorem \ref{1}. The second part is related to estimation of the precision matrix. The third part  considers the asymptotic properties of the new de-sparsified high dimensional GMM estimator.

Lemmas A.1-A.5 establish results that are used in the proof of Theorem \ref{1} and parts of other proofs.
 
Occasionally, we allow constants such as $C$ and $K$ to change from line to line and display to display. Except for in Lemmas \ref{su2} and \ref{su3} $\kappa_n=\ln q$ throughout this appendix, cf. also the remark prior to Assumption \ref{s3}.
 
\subsection{Two auxilliary lemmas}
We first provide concentration inequalities for maxima of centered iid sums. These are taken directly from Lemmas E.1 and E.2 of \cite{cck16} specialized to our iid setting to simplify the used conditions slightly. They can be found in Lemmas \ref{su2} and \ref{su3}. To set the stage assume that $F_i = (F_{i1},..., F_{ij},..., F_{id})'\in \mathbb{R}^d$ and that the vectors are iid across $i=1,...,n$. 
Define 
\[ \max_{1 \le j \le d} | \sum_{i=1}^n (F_{ij} - E F_{ij})| = n \max_{1 \le j \le d} | \hat{\mu}_j - \mu_j |,\]
where $\hat{\mu}_j = n^{-1} \sum_{i=1}^n F_{ij}, \mu_j = E F_{ij}$. Next, define
\[ M_F = \max_{1 \le i \le n} \max_{1 \le j  \le d} | F_{ij} - E F_{ij}|,\]
and 
\[ \sigma_F^2 = \max_{1 \le j \le d} \sum_{i=1}^n E [ F_{ij} - E F_{ij}]^2 = n \max_{1 \le j \le d} var(F_{ij}).\]
From Lemma E.1 of \cite{cck16} one has that there exists a universal constant $K>0$ such that
\[ n E \max_{1 \le j \le d} | \hat{\mu}_j - \mu_j| 
\le K [ \sqrt{n \max_{1 \le j \le d} var (F_{ij})} \sqrt{\ln d} + \sqrt{E M_F^2} \ln d],\]
which implies
\begin{equation}
E \max_{1 \le j \le d} | \hat{\mu}_j - \mu_j | \le K [ \frac{\sqrt{\max_{1 \le j \le d} var (F_{ij}) \ln d}}{\sqrt{n}} + \frac{\sqrt{E M_F^2} \ln d }{n}]
.\label{emax}
\end{equation}
Next, Lemma E.2(ii) of \cite{cck16} states that for all $ \eta >0, t>0, \gamma \ge 1$
\begin{align}
P \sbr[2]{ \max_{1 \le j \le d}  | \hat{\mu}_j - \mu_j | \ge 2 E \max_{1 \le j \le  d} | \hat{\mu}_j - \mu_j | + \frac{t}{n}}
&=
 P \sbr[2]{ n \max_{1 \le j \le d}  | \hat{\mu}_j - \mu_j |\ge (1 + \eta) n E \max_{1 \le j \le  d} | \hat{\mu}_j - \mu_j | + t}\notag\\
&\le 
 \exp(-t^2/3 \sigma_F^2) + K \frac{E M_F^{\gamma}}{t^{\gamma}}.\label{expineq1}
\end{align}

\begin{remark}\label{rem:genres}
For ease of reference, we state the versions of (\ref{emax}) and (\ref{expineq1}) appropriate for our purpose as a lemma. In the rest of the paper we shall use $\eta=1$ and $\gamma=2$. Furthermore, for the purpose of proving our theorems in the case of $q\geq p>n$, we introduce the sequence $\kappa_n$. This sequence will be chosen to equal $\ln(q)$. At the expense of a bit more involved notation and slightly altered assumptions, one can also set $\kappa_n=\max(\ln q,\ln n)$ in order to handle all possible regimes/orderings of $p,q$ and $n$. The following lemma is stated for a maximum over $d$ terms, where often in the sequel we will have $d=q$.
\end{remark}

\begin{assum}
Assume $F_i$ are iid random $d \times 1$ vectors across $i=1,..., n$ with $\max_{1 \le j \le d}  var(F_{ij})$ bounded from above. Finally, let $\kappa_n=\ln d$. \label{s3}
\end{assum}

\begin{lemma}\label{su2}
Under Assumption \ref{s3}

(i). \begin{eqnarray*}
E \max_{1 \le j \le d} | \hat{\mu}_j - \mu_j | &\le& K [ \frac{\sqrt{\max_{1 \le j \le d} var (F_{ij}) \ln d}}{\sqrt{n}} + \frac{\sqrt{E M_F^2} \ln d }{n}]
\\
& \le & K [ \sqrt{\frac{\ln d}{n}} + \frac{ \sqrt{E M_F^2} \ln d }{n}] .
\end{eqnarray*}

(ii) Set $t=t_n= (n \kappa_n)^{1/2}=(n \ln d)^{1/2}$. There exist constants $C,K>0$ such that
\begin{align*}
P \sbr[2]{\max_{1 \le j \le d}  | \hat{\mu}_j - \mu_j | \ge 2E \max_{1 \le j \le  d} | \hat{\mu}_j - \mu_j | + \frac{\kappa_n^{1/2}}{n^{1/2}}}
 \le  \exp(-C \kappa_n) + K \frac{E M_F^2}{n \kappa_n}
 =  \frac{1}{d^C} + K \frac{E M_F^2}{n \ln d}
\end{align*}

\end{lemma}

Note that Nemirowski's inequality could have  been used as in Lemma 14.24 of \cite{bvdg2011}, where the iid case translates to
\[ E \max_{1 \le j \le d} | \hat{\mu}_j - \mu_j| \le \sqrt{8 ln d/n} \sqrt{E M_F^2},\]
which is less sharp than the results above. If we have used Nemirowski's result, we could have needed $E M_F^2 \le C < \infty$ which is very strong.

For the purpose of obtaining asymptotic results, we introduce the following assumption.

\begin{assum}
\[ \frac{(E M_F^{ 2})^{1/2}\sqrt{ \ln d}}{\sqrt{n}} \to 0.\] 
\label{s4}
\end{assum}

\begin{lemma}\label{su3}
Under Assumptions \ref{s3} \ref{s4}

(i).  
\begin{align*}
&P \left( \max_{1 \le j \le d} | \hat{\mu}_j - \mu_j | \le 2E \max_{1 \le j \le d} | \hat{\mu}_j - \mu_j| + 
(\frac{\kappa_n}{n})^{1/2}\right) \\
=
 &P \left( \max_{1 \le j \le d} | \hat{\mu}_j - \mu_j | \le 2E \max_{1 \le j \le d} | \hat{\mu}_j - \mu_j| + 
(\frac{\ln d}{n})^{1/2}\right)\to 1.
\end{align*}

(ii). \[ E \max_{1 \le j \le d} | \hat{\mu}_j - \mu_j|=O (\sqrt{\ln d/n}).\]

(iii). Thus,
\begin{align*}
\max_{1 \le j \le d} | \hat{\mu}_j - \mu_j|
=
O_p\del[2]{2E \max_{1 \le j \le d} | \hat{\mu}_j - \mu_j| + 
(\frac{\kappa_n}{n})^{1/2}}
=
O_p\del[2]{\sqrt{\ln \kappa_n/n}}
=
O_p\del[2]{\sqrt{\ln d/n}}
\end{align*}


\end{lemma}

In the sequel we use the above two lemmata with $\kappa_n=\ln q$. 
\subsection{Some useful events} 
 
 To establish the desired oracle inequality for the estimation error of our estimator we need to bound certain moments. Let $M_1, M_2$ be as defined before Assumption \ref{2}. Define
  \begin{equation}
  {\cal{A}}_1 = \{ \|Z'u/n\|_{\infty} \le t_1/2 \},\label{a7}
  \end{equation}  
 where
 \begin{equation}
 t_1 =  2K [ \frac{C \sqrt{\ln q}}{\sqrt{n}} + \frac{\sqrt{E M_1^2} \ln q}{n}] + \sqrt{\frac{\kappa_n}{n}}.\label{t1n}
  \end{equation}
  for some $C>0$ made precise below. Next, define the set
  \begin{equation}
{\cal{A}}_2 =\{   \|\frac{Z'X}{n} \|_{\infty} \le t_2 \},\label{a10} 
\end{equation} 
with $t_2 = t_3 + C$, where
\begin{equation}
t_3=2K [ \frac{C \sqrt{\ln pq}}{\sqrt{n}} + \frac{\sqrt{E M_2^2} \ln (pq)}{n}]+ \sqrt{\frac{\kappa_n}{n}}.\label{t2s}
\end{equation}
Now we provide probabilities on the bounds and asymptotic rates.

\begin{lemma}\label{bound}

(i) Under Assumption \ref{1},
\[ P ( {\cal A}_1) \ge 1 - \exp(-C \kappa_n) - \frac{K E M_1^{2}}{n \kappa_n},\]
where $C>0$ is the constant from Lemma \ref{su2}

(ii). Adding Assumption \ref{2} to (i)
\[ \| \frac{Z'u}{n} \|_{\infty} = O_p  ( \frac{\sqrt{\ln q}}{\sqrt{n}}).\]

(iii). Under Assumption \ref{1},
\[ P ( {\cal A}_2) \ge 1 - \exp(-C \kappa_n) - \frac{K E M_2^{2}}{n \kappa_n},\]

(iv). Adding Assumption \ref{2} to (iii)
\[  \| \frac{Z'X}{n} - E \left[ \frac{Z'X}{n}  \right]\|_{\infty} = O_p ( \frac{\sqrt{\ln q}}{\sqrt{n}}).\]
\[ \| \frac{Z'X}{n} \|_{\infty} = O_p (1).\]
\end{lemma}

\begin{proof}[Proof of Lemma \ref{bound}]

(i)-(ii). First, note that by Lemma \ref{su2}, replacing $F_i$ with $Z_i u_i $ which is a $q \times 1$ vector, under Assumption \ref{1},  $\mathcal{A}_1$ has probability at least $1 -\exp(-C \kappa_n) - \frac{K E M_1^{2}}{n \kappa_n} $. 
Adding Assumption 2, via Lemma \ref{su3},
    \begin{equation}
  \| \frac{Z'u}{n} \|_{\infty}  = O_p (\frac{\sqrt{\ln q}}{\sqrt{n}}),\label{zu}
  \end{equation}
  
  (iii)-(iv).

Next, consider $ \| \frac{Z'X}{n}\|_{\infty} $. By Lemma \ref{su2}
\begin{equation}
P \left( \frac{\max_{1 \le l \le q} \max_{1 \le j \le p}  \left|  \sum_{i=1}^n [Z_{il} X_{ij} - E Z_{il} X_{ij}] \right| }{n}> t_3\right) \le 
\exp(-C \kappa_n)+ \frac{K E M_2^{2}}{n \kappa_n},\label{a8}
\end{equation}
In conjunction with Assumption \ref{2} (\ref{a8}) implies, via Lemma \ref{su3}, and using $p \le q$
\begin{equation}
\max_{1 \le l \le q} \max_{1 \le j \le p} \left| n^{-1}  \sum_{i=1}^n [Z_{il} X_{ij} - E Z_{il} X_{ij}] \right| = O_p ( \sqrt{\frac{\ln(pq)}{n}})
= O_p ( \sqrt{\frac{\ln q}{n}}) .\label{a8.1}
\end{equation}
Also, by Assumption 1 and Cauchy-Schwarz inequality
\begin{equation}
\max_{1 \le l \le q} \max_{1 \le j \le p} |  E Z_{il} X_{ij} | = O(1).\label{a9}
\end{equation}
Combining (\ref{a8.1}) with (\ref{a9}) we have that $\mathcal{A}_2$ occurs with probability at least $1- \exp(-C \kappa_n) - \frac{K E M_2^{2}}{n \kappa_n}$ 
\end{proof}

\subsection{Oracle inequality for the first step estimator}

Lemmata \ref{l1} and \ref{l2} below are needed  for the proof of Theorem \ref{thm1}. 
Define the norm $\enVert[0]{x}_n=\frac{(x'x)^{1/2}}{n\sqrt{q}}=\frac{\enVert[0]{x}_2}{n\sqrt{q}}$ on $\mathbb{R}^q$. One can thus write
\begin{align*}
\hat{\beta}_F = \argmin_{\beta \in\mathbb{R}^p} \left[ \|Z' (Y - X \hat{\beta}_F)\|_n^2 + 2 \lambda_n \| \beta \|_1 \right].
\end{align*}
Define also the sample covariance between regressors and instruments:
\[ \hat{\Sigma}_{xz} = \frac{X'Z}{n}.\]
With this notation in place we can introduce the concept of empirical adaptive restricted eigenvalue in GMM:
\begin{equation}
\hat{\phi}_{\hat{\Sigma}_{xz}}^2 (s)  = \min\cbr[3]{ 
\frac{\delta' \hat{\Sigma}_{xz}  \hat{\Sigma}_{xz}'\delta}{q \|\delta_S\|_2^2}:  \delta \in \mathbb{R}^p\setminus \cbr[0]{0},\ \| \delta_{S^c} \|_1 \le 3 \sqrt{s} \|\delta_S\|_2,\ |S|\leq s},\label{aregmm}
\end{equation}

We also define the population adaptive restricted eigenvalue for the first step GMM: $\phi_{\Sigma_{xz}}^2 (s)$, as (\ref{pev1}) evaluated at $W_d=I_q$. 
In the sequel we shall choose 
\begin{align}
&\lambda_n
=
t_1t_2=\notag\\
&
\sbr[2]{2K [ \frac{C \sqrt{\ln q}}{\sqrt{n}} + \frac{\sqrt{E M_1^2} \ln q}{n}] + \sqrt{\frac{\kappa_n}{n}}}\sbr[2]{2K [ \frac{C \sqrt{\ln pq}}{\sqrt{n}} + \frac{\sqrt{E M_2^2} \ln (pq)}{n}]+ \sqrt{\frac{\kappa_n}{n}}+C}
 \label{eq:lambda_def}
\end{align}
and note that under Assumption \ref{2}, $\lambda_n=O(\sqrt{\ln q/n})$.

\begin{lemma}\label{l1}
Under Assumptions \ref{1} and \ref{2}, for universal positive constants $K, C$, for $n$ sufficiently large one has with probability at least 
$1 - 3 \exp(-C \kappa_n) - K \frac{ E M_1^{2} + 2 E M_2^{2}}{n \kappa_n} $


(i). \[ \| \hat{\beta}_F - \beta_0 \|_1 \le \frac{24 \lambda_n s_0}{ \phi_{\Sigma_{xz}}^2 (s_0)}.\]
  
(ii).  
The  result in (i) holds with probability approaching one and we have $\lambda_n = O (\sqrt{\ln q/n})$ as seen in (\ref{lambdar}) below. (i) is valid uniformly over ${\cal B}_{l_0}(s_0) = \{ \|\beta_0 \|_{l_0} \le s_0 \}$.
\end{lemma}

\begin{proof}[Proof of Lemma \ref{l1}] 
(i). Since
\[ \|Z' (Y - X \hat{\beta}_F)\|_n^2 = \frac{1}{n}  \left[(Y - X \hat{\beta}_F)' \frac{Z Z'}{nq} (Y- X \hat{\beta}_F)  \right].\]
the minimizing property of $\hat{\beta}_F$ implies that
\begin{equation}
\| Z' (Y - X \hat{\beta}_F) \|_n^2 + 2 \lambda_n \sum_{j=1}^p | \hat{\beta}_{F,j} | \le 
\| Z' (Y - X \beta_0) \|_n^2 + 2 \lambda_n \sum_{j=1}^p | \beta_{0,j} |.\label{a1}
\end{equation}
Next use that $Y = X \beta_0 +u$ and simplify to get
\begin{equation}
\| Z' X (\hat{\beta}_F - \beta_0)\|_n^2 + 2 \lambda_n \sum_{j=1}^p | \hat{\beta}_{F,j} | \le
2 |\frac{u'Z}{n} \frac{Z'X}{nq} (\hat{\beta}_F - \beta_0)| + 2 \lambda_n \sum_{j=1}^p |\beta_{0,j}|.\label{a2}
\end{equation}
Consider the first term on the right side of (\ref{a2}),  and denote the $l$th row of $Z'X$ by $(Z'X)_l$, $l=1,2,...,q$:
\begin{eqnarray}
2 |\frac{u'Z}{n} \frac{Z'X}{nq} (\hat{\beta}_F - \beta_0)| &\le &2 \| \frac{u'Z}{n} \|_{\infty} \|\frac{Z'X}{nq} (\hat{\beta}_F - \beta_0) \|_1 \label{a3}\\
& \le & 2 \| \frac{u'Z}{n} \|_{\infty} \left[ (nq)^{-1} \sum_{l=1}^q \|(Z'X)_l \|_{\infty} \right] \|(\hat{\beta}_F - \beta_0) \|_1\label{a5} \\
& \le &2  \| \frac{u'Z}{n} \|_{\infty} \left[ \max_{1 \le l \le q}  \|\frac{(Z'X)_l}{n} \|_{\infty} \right] \|(\hat{\beta}_F - \beta_0) \|_1, \label{a6}
\end{eqnarray}
where we use Hölder's inequality in (\ref{a3}) and (\ref{ineq1}) in (\ref{a5}).

Assume that ${\cal{A}}_1 \cap {\cal A}_2$ occurs (we shall later provide a lower bound on the probability of this). By (\ref{a7})(\ref{a10}), in (\ref{a6}) we have on ${\cal{A}}_1 \cap {\cal A}_2$
\begin{equation}
2  \| \frac{u'Z}{n} \|_{\infty} \left[ \max_{1 \le l \le q}  \|\frac{(Z'X)_l}{n} \|_{\infty} \right] \|(\hat{\beta}_F - \beta_0) \|_1
\le \lambda_n  \| \hat{\beta}_F - \beta_0 \|_1,
 \label{a11}
\end{equation}
We note that by Assumption \ref{2}, Lemma \ref{su3}, (\ref{zu}) and Lemma \ref{bound} (iv)
\begin{equation}
\lambda_n = O (\sqrt{\frac{\ln q}{n}}). \label{lambdar}
\end{equation}
In combination with (\ref{a2}) we get:
\begin{equation}
\| Z'X (\hat{\beta}_F - \beta_0) \|_n^2 + 2 \lambda_n \sum_{j=1}^p |\hat{\beta}_{F,j}| \le
\lambda_n \|\hat{\beta}_F  - \beta_0 \|_1 + 2 \lambda_n \sum_{j=1}^p |\beta_{0,j}|.\label{a12}
\end{equation}
Next, use that $\| \hat{\beta}_F\| = \| \hat{\beta}_{F,S_0} \|_1 + \| \hat{\beta}_{F,S_0^c} \|_1$ on the second term on the left side of (\ref{a12})
\begin{equation}
\| Z'X (\hat{\beta}_F - \beta_0) \|_n^2 + 2 \lambda_n  \sum_{j \in S_0^c} |\hat{\beta}_{F,j}| \le
\lambda_n \|\hat{\beta}_F  - \beta_0 \|_1 + 2 \lambda_n \sum_{j \in S_0} |\hat{\beta}_{F,j}- \beta_{0,j}|,\label{a13}
\end{equation}
where we used the reverse triangle inequality to get the last term on the right side of (\ref{a13}) and $ \| \beta_{0, S_0^c} \|_1 =0$. Using that $\| \hat{\beta}_F - \beta_0 \|_1 = \| \hat{\beta}_{F,S_0} - \beta_{0, S_0} \|_1 + \| \hat{\beta}_{F,S_0^c} \|_1$ on the first right hand side term in (\ref{a13}) yields
\begin{equation}
\| Z'X (\hat{\beta}_F - \beta_0) \|_n^2 +  \lambda_n  \sum_{j \in S_0^c} |\hat{\beta}_{F,j}| \le
3 \lambda_n \sum_{j \in S_0} |\hat{\beta}_{F,j}- \beta_{0,j}|.\label{a14}
\end{equation}
Furthermore, using that $\enVert[0]{\hat{\beta}_{F,S_0}- \beta_{0,S_0}}_1\leq \sqrt{s_0}\enVert[0]{\hat{\beta}_{F,S_0}- \beta_{0,S_0}}_2$ in (\ref{a14})
\begin{equation}
\| Z'X (\hat{\beta}_F - \beta_0) \|_n^2 +  \lambda_n  \sum_{j \in S_0^c} |\hat{\beta}_{F,j}| \le
3 \lambda_n \sqrt{s_0}   \|\hat{\beta}_{F,S_0}- \beta_{0,S_0}\|_2.\label{a15}
\end{equation}
We see that the restricted set condition in (\ref{aregmm}) is satisfied by ignoring the first term on the right side of (\ref{a15}) and dividing each side by $\lambda_n$ 
\[ \| \hat{\beta}_{F, S_0^c} \|_1 \le 3 \sqrt{s_0} \|\hat{\beta}_{F,S_0} - \beta_{0, S_0} \|_2.\]
Thus, the empirical adaptive restricted eigenvalue condition in (\ref{aregmm}) can be used in (\ref{a15})
\begin{equation}
\| Z'X (\hat{\beta}_F - \beta_0) \|_n^2 +  \lambda_n  \sum_{j \in S_0^c} |\hat{\beta}_{F,j}| \le
3 \lambda_n \sqrt{s_0}   \frac{\| Z'X (\hat{\beta}_F - \beta_0) \|_n}{\hat{\phi}_{\hat{\Sigma}_{xz}} (s_0)}
\label{a16}
\end{equation}
Next, use $3uv \le u^2/2 + 9v^2/2$ with $u = \| Z'X (\hat{\beta}_F - \beta_0) \|_n$ and $v=\lambda_n \sqrt{s_0}/\hat{\phi}_{\hat{\Sigma}_{XZ}} (s_0)$ to get 
\begin{equation}
\| Z'X (\hat{\beta}_F - \beta_0) \|_n^2 +  \lambda_n  \sum_{j \in S_0^c} |\hat{\beta}_{F,j}| \le
\frac{\| Z'X (\hat{\beta}_F - \beta_0) \|_n^2}{2}+
\frac{9}{2}  \frac{\lambda_n^2 s_0}{   \hat{\phi}_{\hat{\Sigma}_{xz}}^2 (s_0)}.\label{a17}
\end{equation}
Multiply each side of (\ref{a17}) by 2 and simplify to get 
\begin{equation}
\| Z'X (\hat{\beta}_F - \beta_0) \|_n^2 + 2  \lambda_n  \sum_{j \in S_0^c} |\hat{\beta}_{F,j}| \le
  \frac{9 \lambda_n^2 s_0}{   \hat{\phi}_{\hat{\Sigma}_{xz}}^2 (s_0)}.\label{a18}
\end{equation}
Next, assume that  ${\cal A}_3=
\{ \hat{\phi}_{\hat{\Sigma}_{xz}}^2 (s_0) \ge \phi_{\Sigma_{xz}}^2 (s_0)/2\}$ such that we are working on ${\cal{A}}_1 \cap {\cal{A}}_2 \cap {\cal{A}}_3$. Thus,
\begin{equation}
\| Z'X (\hat{\beta}_F - \beta_0) \|_n^2 + 2  \lambda_n \sum_{j \in S_0^c} |\hat{\beta}_{F,j}| \le
  \frac{  18 \lambda_n^2 s_0}{    \phi_{\Sigma_{xz}}^2 (s_0)}.\label{a19}
\end{equation}
So we have the oracle inequality on set ${\cal{A}}_1 \cap {\cal{A}}_2 \cap {\cal{A}}_3$.

To get the $l_1$ error bound ignore the first term in (\ref{a15}) above and add $\lambda_n \|\hat{\beta}_{S_0} - \beta_{0, S_0} \|_1$ to both sides to get 
\begin{equation}
\lambda_n \| \hat{\beta}_F - \beta_0 \|_1 \le \lambda_n \| \hat{\beta}_{F,S_0} - \beta_{F,S_0}\|_1 + 3 \lambda_n \sqrt{s_0}
\| \hat{\beta}_{F,S_0} - \beta_{S_0}\|_2.\label{a20}
\end{equation}
Then use $\enVert[0]{\hat{\beta}_{F,S_0}- \beta_{0,S_0}}_1\leq \sqrt{s_0}\enVert[0]{\hat{\beta}_{F,S_0}- \beta_{0,S_0}}_2$ that as well as the empirical adaptive restricted eigenvalue condition for GMM in (\ref{aregmm})
\begin{equation}
\lambda_n \| \hat{\beta}_F - \beta_0 \|_1 \le  4 \lambda_n \sqrt{s_0} \frac{\| Z'X (\hat{\beta}_F - \beta_0)\|_n}{\hat{\phi}_{\Sigma_{xz}} (s_0)}
\label{a21}
\end{equation}
Next, as ${\cal{A}}_3$ is assumed to occur and using the prediction norm upper bound established in (\ref{a19}) results in
\begin{equation}
\| \hat{\beta}_F - \beta_0 \|_1 \le  \frac{24 \lambda_n s_0}{ \phi_{\Sigma_{xz}  }^2  (s_0) }.
\label{a22}
\end{equation}
Thus, in total, by Lemma S.2
\[  1 - P({\cal{A}}_1^c) - P ( {\cal{A}}_2^c) - P ( {\cal{A}}_3^c) = 1 - 3\exp(- C \kappa_n) - \frac{K (E M_1^{2} + 2E M_2^{2})}{n \kappa_n}\to 1,\]
by Assumption 2 where the convergence to 1 establishes ii).
\end{proof}


\subsection{Controlling $\min_{1 \le l \le q} \hat{\sigma}_l^2$ for $\hat{W}_d$}\label{8.4}
\begin{lemma}\label{l2}
Under Assumptions \ref{1} and \ref{2} as well as $r_z \ge 12, r_x \ge 6, r_u \ge 8$, we have that for $n$ sufficiently large there exists a $C>0$ such that
\[ \min_{1 \le l \le q} \hat{\sigma}_l^2 \ge  \min_{1 \le l \le q} \sigma_l^2 /2,\]
with probability at  least $1 - 9 \exp(-C \kappa_n)- \frac{K [ 2 E M_1^{2} + 4 E M_2^{2} + E M_3^{2} + E M_4^{2} + E M_5^{2}]}{n \kappa_n}\to 1$.
The result is valid uniformly over ${\cal B}_{l_0} = \{ \|\beta_0 \|_{l_0} \le s_0 \}$. \end{lemma}

\begin{remark}\label{rem:sigmas}
In the course of the proof of Lemma \ref{l2} we actually establish that under the assumptions of said lemma,
\begin{align*}
P\del[2]{\max_{1 \le l \le q} | \hat{\sigma}_l^2 - \sigma_l^2 |\geq c_n}
\leq
9 \exp(-C \kappa_n)+ \frac{K [ 2 E M_1^{2} + 4 E M_2^{2} + E M_3^{2} + E M_4^{2} + E M_5^{2}]}{n \kappa_n}\to 0
\end{align*}
for a sequence $c_n\to 0$ (where $c_n$ is defined precisely in (\ref{a25})). To be precise, the proof reveals that $c_n=O(s_0\sqrt{\ln q/n})$.
\end{remark}



\begin{proof}[Proof of Lemma \ref{l2}] 
First, note that
\begin{equation}
\min_{1 \le l \le q} \hat{\sigma}_l^2 \ge \min_{1 \le l \le q} \sigma_l^2 - \max_{1 \le l \le q} | \hat{\sigma}_l^2 - \sigma_l^2 |.\label{a23}
\end{equation}
We start by upper bounding $\max_{1 \le l \le q} | \hat{\sigma}_l^2 - \sigma_l^2 |$. To this end, note that we can write 
\[\hat{\sigma}_l^2 =  \frac{1}{n} \sum_{i=1}^n Z_{il}^2 \hat{u}_i^2 = \frac{1}{n} \sum_{i=1}^n Z_{il}^2 [ u_i - X_i' (\hat{\beta}_F - \beta_0)]^2.\]
Then, by the triangle inequality 
\begin{eqnarray}
\max_{1 \le l \le q} | \hat{\sigma}_l^2 - \sigma_l^2| & \le & 
\max_{1 \le l \le q} | \frac{1}{n} \sum_{i=1}^n (Z_{il}^2 u_i^2 - E Z_{il}^2 u_i^2)| 
+ \max_{1 \le l \le q} | \frac{2}{n} \sum_{i=1}^n Z_{il}^2 u_i X_i' (\hat{\beta}_F - \beta_0)| \nonumber \\
& + & \max_{1 \le l \le q} | (\hat{\beta}_F - \beta_0)' \frac{ \sum_{i=1}^n Z_{il}^2 X_i X_i'}{n} (\hat{\beta}_F - \beta_0)| .\label{a24}
 \end{eqnarray}
Define the following events in order to upper bound the right hand side of \ref{a24}
\[ {\cal{B}}_1  = \{ \max_{1 \le l \le q}  | n^{-1}  \sum_{i=1}^n (Z_{il}^2 u_i^2 - E Z_{il}^2 u_i^2)| \le t_4 \}.\]
\[ {\cal{B}}_2 = \{  \max_{1 \le l \le q} \max_{1 \le j \le p} | n^{-1}  \sum_{i=1}^n Z_{il}^2 u_i X_{ij} | \le t_5 \}.\]
\[ {\cal{B}}_3 = \{ \max_{1 \le l \le q} \max_{1 \le j \le p } \max_{1 \le k \le p}  |  n^{-1} \sum_{i=1}^n Z_{il}^2 X_{ij} X_{ik}| \le t_6\}.\]
\[ {\cal{B}}_4 = \{ \| \hat{\beta}_F - \beta_0 \|_1 \le C \lambda_n s_0 \}.\]
where $t_4,t_5$ and $t_6$ will be specified as we analyze the individual events.
We will show that, with probability at least $1 - 9 \exp(-C \kappa_n)- \frac{K [ 2 E M_1^{2} + 4 E M_2^{2} + E M_3^{2} + E M_4^{2} + E M_5^{2}]}{n \kappa_n} $
\begin{equation}
\max_{1 \le l \le q} | \hat{\sigma}_l^2 - \sigma_l^2| \le C[t_4 + t_5 \lambda_n s_0 + t_6 \lambda_n^2 s_0^2] =:c_n\leq \min_{1 \le l \le q} \sigma_l^2/2\label{a25}
\end{equation}
Consider $\mathcal{B}_1$ first. By Lemma \ref{su2}, with $F_{il} = Z_{il}^2 u_i^2$ and $d=q$ via Assumption \ref{1}, and $r_z \ge 8, r_u \ge 8$  as well as the Cauchy-Schwarz inequality
\begin{equation}
P \left( \max_{1 \le l \le q} \frac{ | \sum_{i=1}^n (Z_{il}^2 u_i^2 - E Z_{il}^2 u_i^2)|}{n} > t_4   \right)
\le \exp(-C \kappa_n) + \frac{K E M_3^{2}}{n\kappa_n},\label{a26}
\end{equation}
with $t_4 = 2K [ C_U \frac{\sqrt{\ln q}}{\sqrt{n}} + \frac{ \sqrt{E M_3^2} \ln q}{n}] + \sqrt{\frac{\kappa_n}{n}}.$ Note that by Assumption \ref{2} one has $t_4 \to 0$. 
Next consider the second term on the right side of (\ref{a24}).  By Hölder's inequality
\begin{equation}
\max_{1 \le l \le q} | \frac{1}{n} \sum_{i=1}^n Z_{il}^2 u_i X_i' (\hat{\beta}_F - \beta_0) | \le 
[\max_{1 \le l \le q} \| \frac{1}{n} \sum_{i=1}^n Z_{il}^2 u_i X_i \|_{\infty} ] \| \hat{\beta}_F - \beta_0\|_1.\label{a27}
\end{equation}
We analyze the first term on the right side  of (\ref{a27}). Next by Lemma \ref{su2}
\begin{equation}
P \left(\frac{ \max_{1 \le l \le q} \max_{1 \le j \le p} | \sum_{i=1}^n (Z_{il}^2 u_i X_{ij} - E Z_{il}^2 u_i X_{ij})|}{n}>t_5^*
\right) \le \exp(-C \kappa_n)+  \frac{K E M_4^{2}}{n\kappa_n},\label{a28}
\end{equation}
where 
\begin{equation}
t_5^{*}= 2K [ C_U  \sqrt{\frac{\ln pq}{n}} + \frac{ \sqrt{E M_4^2} \ln pq}{n}] + \sqrt{\frac{\kappa_n}{n}}\to 0.\label{t2s}
\end{equation}
by Assumption \ref{2} and $p \le q$. By Assumption 1 with $r_z \ge 12, r_u \ge 8, r_x \ge 6$ 
\begin{eqnarray}
E | Z_{il}^2 u_i X_{ij}|^2 & = & E | Z_{il}^4 u_i^2 X_{ij}^2| \nonumber \\
& \le & [E | Z_{il}|^{12}]^{1/3} [ E |u_i|^3 E |X_{ij}|^3]^{2/3} \nonumber \\
& \le & [ E | Z_{il}|^{12}]^{1/3} [ E |u_i|^6]^{1/3} [ E |X_{ij}|^6]^{1/3} \le C < \infty.\label{dcs}
\end{eqnarray}
such that
\begin{equation}
\max_{1 \le l \le q} \max_{1 \le j \le p}  | E (Z_{il}^2 u_i X_{ij})| \le C.\label{a29}
\end{equation}
Thus, with $t_5=t_5^*+C=O(1)$ by Assumption \ref{2} we have that $P(\mathcal{B}_2)\leq \exp(-C \kappa_n)+  \frac{K E M_4^{2}}{n\kappa_n}$.
Using this in (\ref{a27}) together with the upper bound on $\| \hat{\beta}_F - \beta_0\|_1$ in Lemma \ref{l1} yields
\begin{equation}
 \max_{1 \le l \le q} | \frac{1}{n} \sum_{i=1}^n Z_{il}^2 u_i X_i' (\hat{\beta}_F - \beta_0)|  \le C[ t_5 \lambda_n s_0] = C [ C + t_5^*]
 \lambda_n s_0\to 0,\label{a30}
\end{equation}
with probability at least $1 - 4 \exp (- C \kappa_n) - \frac{ K (E M_1^{2} + 2 E M_2^{2} + E M_4^{2})}{n\kappa_n}$ and where the convergence to 0 is by Assumption \ref{2}. 
We now turn to ${\cal{B}}_3$. By Assumption 1, and similar analysis in (\ref{dcs}) gives
{\small
\begin{equation}
P \left( \frac{\left[ \max_{1 \le l \le q} \max_{1 \le j \le p} \max_{1 \le k \le p} | \frac{1}{n} \sum_{i=1}^n (Z_{il}^2 X_{ij} X_{ik} - E Z_{il}^2 X_{ij} X_{ik})| \right]}{n} > t_6^* \right)
\le  \exp(-C \kappa_n) + \frac{ K E M_5^{2}}{n\kappa_n},\label{a32}
\end{equation}
}
where 
\begin{equation}
t_6^*= 2 K [ C_U \sqrt{\frac{\ln p^2 q}{n}} + \frac{\sqrt{E M_5^2} \ln p^2 q}{n}] + \sqrt{\frac{\kappa_n}{n}}.\label{t3s}\to 0
\end{equation}
by Assumption \ref{2}. Furthermore, applying the Cauchy-Schwarz inequality twice and using  Assumption 1,
\begin{equation}
 \max_{1 \le l \le q} \max_{1 \le j \le p} \max_{1 \le k \le p}  E |Z_{il}^2 X_{ij} X_{ik}|   \le C.\label{a33}
\end{equation}
Thus, with probability at least $ \exp(-C \kappa_n) + \frac{ K E M_5^{2}}{n\kappa_n}$ we have that $\frac{1}{n} \sum_{i=1}^n Z_{il}^2  X_i X_i'\leq t_6$ where $t_6=t_6^*+C=O(1)$ by Assumption 2. Thus, using also Lemma \ref{l1} one has
\begin{align}
\max_{1 \le l \le q} | (\hat{\beta}_F - \beta_0)' \frac{ \sum_{i=1}^n Z_{il}^2 X_i X_i'}{n} (\hat{\beta}_F - \beta_0)| 
&\le 
\|\hat{\beta}_F - \beta_0 \|_1^2 \left[ \max_{1 \le l \le q} \max_{1 \le j \le p} \max_{1 \le k \le p} | \frac{1}{n} \sum_{i=1}^n Z_{il}^2 X_{ij} X_{ik}| \right]\notag\\
&\leq
C t_6 \lambda_n^2 s_0^2\to 0,\label{a34}
\end{align}
with probability at least $1 - 4 \exp(-C \kappa_n) - \frac{K [ E M_1^{2} + 2E M_2^{2} + E M_5^{2}]}{n \kappa_n}$ and where the convergence to zero is by Assumption \ref{2}. The above results are valid uniformly over $l_0$ ball: ${\cal B}_{l_0} = \{ \|\beta_0 \|_{l_0} \le s_0 \}$. This can be seen by (\ref{a30}) and (\ref{a34}) since the dependence on $\beta_0$ in the bounds is through $s_0$ only. 
\end{proof}

\vspace{1cm}

\subsection{Proof of Theorem \ref{thm1}}
For the purpose of proving Theorem \ref{thm1} below we introduce the empirical version of the adaptive restricted eigenvalue condition for GMM at $s=s_0$:
\begin{align}
 \hat{\phi}_{\hat{\Sigma}_{xz \hat{w}}}^2 (s_0)
&=
 \min\cbr[3]{ \frac{|\delta' \frac{X'Z}{n} \hat{W}_d \frac{Z'X}{n} \delta|}{q \|\delta_{S_0} \|_2^2}:\ \delta\in\mathbb{R}^p \setminus \cbr[0]{0},\ \| \delta_{S_0^c} \|_1 \le 3 \sqrt{s_0} 
\| \delta_{S_0} \|_2,\ |S_0|\leq s_0 }\notag\\
&= 
\min\cbr[3]{  \frac{ \| \hat{W}_d^{1/2} \frac{Z'X}{n} \delta \|_2^2}{q \| \delta_{S_0} \|_2^2}:\  \delta\in\mathbb{R}^p \setminus \cbr[0]{0},\ \| \delta_{S_0^c} \|_1 \le 3 \sqrt{s_0} 
\| \delta_{S_0} \|_2,\ |S_0|\leq s_0}.
\label{eev1}
\end{align}

Furthermore, we shall choose $\lambda_n^*$ as follows:
\begin{align}
&\lambda_n^* = \lambda_n \frac{2}{\min_{1 \le l \le q} \sigma_l^2}=\notag\\
&\frac{2}{\min_{1 \le l \le q} \sigma_l^2}\sbr[2]{2 K [ \frac{C \sqrt{\ln q}}{\sqrt{n}} + \frac{\sqrt{E M_1^2} \ln q}{n}] + \sqrt{\frac{\kappa_n}{n}}}\sbr[2]{2 K [ \frac{C \sqrt{\ln pq}}{\sqrt{n}} + \frac{\sqrt{E M_2^2} \ln (pq)}{n}]+ \sqrt{\frac{\kappa_n}{n}}+C}
\label{eq:lambdastar_def}.
\end{align}
where we used the definition of $\lambda_n$ in (\ref{eq:lambda_def}). Recall that $C$ and $K$ are universal constants guaranteed to exist by Lemma \ref{su2}. Note that under Assumptions \ref{1} and \ref{2}, one has that $\lambda_n^*=O(\sqrt{\ln q/n})$. 		
\begin{proof}[Proof of Theorem \ref{thm1}]
i) The proof is very similar to the one of Lemma \ref{l1} above. Thus, we only point out the differences. There are four main differences. The first one is the set up of instruments, the second one is the noise term, and the third one is the empirical adaptive restricted eigenvalue condition, the fourth one is the tuning parameter. We will show how each component changes the proof. First, the instrument matrix is transformed from $Z$ to $\tilde{Z}= Z \hat{W}_d^{1/2}$, which is again a $n \times q$ matrix but 
\begin{equation}
 \tilde{Z} = [ Z_1, \cdots, Z_l, \cdots, Z_q] \left[ \begin{array}{ccc}
														1/\hat{\sigma}_1 & \cdots & 0 \\
														\vdots  & 1/\hat{\sigma}_l & \vdots \\
														0  & \cdots & 1/\hat{\sigma}_q \end{array} \right].\label{t1.0}
														\end{equation}   
$\hat{W}_d^{1/2}$ is a diagonal matrix and $\hat{\sigma}_l^2 = n^{-1} \sum_{i=1}^n Z_{il}^2 \hat{u}_i^2$, with $\hat{u}_i = Y_i - X_i' \hat{\beta}_F$. Note that,
$\| \tilde{Z}' (Y - X \hat{\beta}) \|_n^2 = \frac{(Y - X \hat{\beta})' Z \hat{W}_d Z' (Y - X \hat{\beta})}{n^2 q}$. Using the definition of $\hat{\beta}$ in (\ref{eq:1}) yields
\begin{equation}
\| \tilde{Z}' (Y - X \hat{\beta}) \|_n^2 + 2 \lambda_n^* \sum_{j=1}^p | \hat{\beta}_{j} | \le 
\| \tilde{Z}' (Y - X \beta_0) \|_n^2 + 2 \lambda_n^* \sum_{j=1}^p | \beta_{0,j} |.\label{t1.1}
\end{equation}
After (\ref{t1.1}) we continue as in (\ref{a1})-(\ref{a6})  with (\ref{t1.0}) and remembering that $(\hat{W}_d Z'X)_l$ is $l$th row of $\hat{W}_d Z'X$.
{\small
\begin{eqnarray}
\|\tilde{Z}' X (\hat{\beta} - \beta_0) \|_n^2 + 2 \lambda_n^* \sum_{j=1}^p |\hat{\beta}_j | & \le & 
2 \left|  \frac{u' \tilde{Z}}{n} \frac{\tilde{Z}'X}{nq} (\hat{\beta} - \beta_0)
\right| + 2 \lambda_n^* \sum_{j=1}^p | \beta_{0,j}| \label{t1.2} \\
& \le & 2 \| \frac{u' Z}{n} \|_{\infty} \| \hat{W}_d  \frac{Z'X}{nq} ( \hat{\beta} - \beta_0) \|_1 + 2 \lambda_n^* \sum_{j=1}^p | \beta_{0,j}|\nonumber \\
& \le & 2 \| \frac{u' Z}{n} \|_{\infty} \left[ \sum_{l=1}^q \| (\frac{\hat{W}_d Z'X}{nq})_l \|_{\infty} \right]  [\| \hat{\beta} - \beta_0 \|_1] + 2 \lambda_n^* \sum_{j=1}^p | \beta_{0,j}|\nonumber \\
& \le & 2 \| \frac{u' Z}{n} \|_{\infty} \left[ q \max_{1 \le l \le q}  \| (\frac{\hat{W}_d Z'X}{nq})_l \|_{\infty} \right]  [\| \hat{\beta} - \beta_0 \|_1] + 2 \lambda_n^* \sum_{j=1}^p | \beta_{0,j}|\nonumber \\
& = &  2 \| \frac{u' Z}{n} \|_{\infty} \| \frac{\hat{W}_d Z'X}{n} \|_{\infty} \| \hat{\beta} - \beta_0 \|_1+ 2 \lambda_n^* \sum_{j=1}^p | \beta_{0,j}| \nonumber \\
& \le & 2 \| \frac{u' Z}{n} \|_{\infty} \| \hat{W}_d \|_{l_{\infty}}  \| \frac{Z'X}{n} \|_{\infty} \| \hat{\beta} - \beta_0 \|_1 + 2 \lambda_n^* \sum_{j=1}^p | \beta_{0,j}|\nonumber \\
& =& 2 \| \frac{u' Z}{n} \|_{\infty} \left[ \| \frac{(Z'X)}{n}\|_{\infty} \right] \| \hat{\beta} - \beta_0\|_1
\left[ \frac{1}{\min_{1 \le l \le q} \hat{\sigma}_l^2} \right] \nonumber \\
& + &  2 \lambda_n^* \sum_{j=1}^p |\beta_{0,j}|,\label{t1.3}
\end{eqnarray} 
}
where Hölder's inequality is used for the second inequality, Lemma \ref{lineq}(i) for the third inequality, simple manipulations for the fourth one, and equation (\ref{ineq3}) for the fifth inequality.

So the difference from (\ref{a6}) is the presence of $1/\min_{1 \le l \le q} \hat{\sigma}_l^2$ in  (\ref{t1.3}). By Lemma \ref{l2} on the set ${\cal{B}}_1 \cap {\cal{B}}_2 \cap 
{\cal{B}}_3 \cap {\cal{B}}_4$ for $n$ sufficiently large
\begin{equation}
\frac{1}{\min_{1 \le l \le q} \hat{\sigma}_l^2} \le \frac{2}{\min_{1 \le l \le q} \sigma_l^2},\label{t1.4}
\end{equation}
We proceed again as in (\ref{a7})-(\ref{a10}) to get (on the set ${\cal{A}}_1 \cap {\cal{A}}_2 \cap {\cal{B}}_1 \cap {\cal{B}}_2 \cap 
{\cal{B}}_3 \cap {\cal{B}}_4$)
\begin{equation}
2 \| \frac{u'Z}{n}\|_{\infty} \left[  \frac{\|(Z'X)\|_{\infty}}{n}\right] \|\hat{\beta} - \beta_0 \|_1 \frac{1}{\min_{1 \le l \le q} \hat{\sigma}_l^2}
\le \lambda_n^* \| \hat{\beta} - \beta_0 \|_1,\label{t1.5}
\end{equation}
where 
\begin{equation}
\lambda_n^* = \lambda_n \frac{2}{\min_{1 \le l \le q} \sigma_l^2}.\label{t1.6}
\end{equation}
Note that by Assumption \ref{2}, we have get $\lambda_n^* = O(\lambda_n)=O (\sqrt{\frac{\ln q}{n}}) = o(1)$. Next, proceed as in (\ref{a12})-(\ref{a15}) (replacing $Z$ by $\tilde{Z}$ and $\hat{\beta}_F$ by $\hat{\beta}$)
\begin{equation}
\| \tilde{Z}' X (\hat{\beta}- \beta_0) \|_n^2 + \lambda_n^* \sum_{j \in S_0^c} | \hat{\beta}_j | \le 3 \lambda_n^* \sqrt{s_0} \|\hat{\beta}_{S_0} - \beta_{0,S_0} \|_2.\label{t1.7}
\end{equation}
Ignoring the first term in (\ref{t1.7}), the restricted set condition is satisfied for the eigenvalue condition. Use (\ref{eev1}) and proceed as in (\ref{a16})-(\ref{a17}) to get 
\[ \| \tilde{Z}' X (\hat{\beta} - \beta_0) \|_n^2 + 2 \lambda_n^* \sum_{j \in S_0^c} | \hat{\beta}_j| \le 
\frac{9 (\lambda_n^*)^2 s_0}{\hat{\phi}_{\hat{\Sigma}_{xz \hat{w}} }^2 (s_0)}.\]
Now proceed as in the proof of Lemma \ref{l1}
\begin{equation}
\| \tilde{Z}' X (\hat{\beta} - \beta_0) \|_n^2 +2  \lambda_n^* \sum_{j \in S_0^c} | \hat{\beta}_j| \le 
\frac{18 (\lambda_n^*)^2 s_0}{\phi_{\Sigma_{xz w } }^2 (s_0) },\label{t1.9}
\end{equation}
where we used the fact that we are also on ${\cal{A}}_4 = \{ \hat{\phi}_{\hat{\Sigma}_{xz\hat{w}}}^2 (s_0) \ge \phi_{\Sigma_{xzw}}^2 (s_0) /2\}$.
%

We now turn to upper bounds on the $l_1$ estimation error. Instead of (\ref{a21}) we have 
\[ \| \hat{\beta} - \beta_0 \|_1 \le \frac{4 \sqrt{s_0} \| \tilde{Z}' X (\hat{\beta} - \beta_0) \|_n}{\hat{\phi}_{\hat{\Sigma}_{xz \hat{w}}}^2 (s_0)}.\]
Using (\ref{t1.9}) on the right side of the above equation yields
\begin{equation}
\| \hat{\beta} - \beta_0 \|_1 \le \frac{24\lambda_n^* s_0 }{\phi_{\Sigma_{xzw}}^2 (s_0)},\label{t1.10}
\end{equation}
where we used the fact that we are on ${\cal{A}}_4 = \{ \hat{\phi}_{\hat{\Sigma}_{xz\hat{w}}}^2 (s_0) \ge \phi_{\Sigma_{xzw}}^2 (s_0)/2\}$. 
These upper bounds are valid uniformly over ${\cal B}_{l_0} = \{ \|\beta_0 \|_{l_0} \le s_0 \}$. Note that the upper bounds are valid on the event ${\cal{A}}_1 \cap {\cal{A}}_2 \cap {\cal{A}}_4
\cap {\cal{B}}_1 \cap {\cal{B}}_2 \cap {\cal{B}}_3 \cap {\cal{B}}_4$. 
We lower bound the probability of $ {\cal{B}}_1 \cap {\cal{B}}_2 \cap {\cal{B}}_3 \cap {\cal{B}}_4$. 
by Lemma \ref{l2}. The probability of ${\cal{A}}_1 \cap {\cal{A}}_2 \cap {\cal{A}}_4$ is lower bounded by Lemma \ref{bound} and Lemma \ref{evalue2}. 

ii) By Assumption 2 the probability of ${\cal{A}}_1 \cap {\cal{A}}_2 \cap {\cal{A}}_4
\cap {\cal{B}}_1 \cap {\cal{B}}_2 \cap {\cal{B}}_3 \cap {\cal{B}}_4$ can then be shown to tend to 1. 
\end{proof}

\vspace{1cm}
\subsection{Properties of the CLIME estimator $\hat{\Gamma}$}
\noindent We next establish three lemmata on the properties of the CLIME estimator. The first two lemmata are adapted from \cite{glt18} and applied to our case. We provide the proofs of them so that it is easy to establish the third lemma we develop for GMM case. Prior to the first lemma, define the event
\[ T_{\Gamma} ( \mu) = \{\| \Gamma \hat{\Sigma} -I_p \|_{\infty} \le \mu \}.\] 

\begin{lemma}\label{cl-l1}
Assume that $ \| \Gamma \|_{l_1}$ is bounded from above by $m_{\Gamma} < \infty.$  Suppose that the rows of $\hat{\Gamma}$, which are denoted $\hat{\Gamma}_j$, are obtained by the CLIME program in section \ref{clime}. Then, on the set $T_{\Gamma} (\mu)$, for each $j=1,2,..., p$
\[ \| \hat{\Gamma}_j - \Gamma_j \|_{\infty} \le 2  m_{\Gamma} \mu.\]
\end{lemma}
\begin{remark}
As the result is for a fixed sample size, one can choose a different $m_\Gamma$ for each $n$. We shall utilize this in the sequel. 
\end{remark} 

\begin{proof}[Proof of Lemma \ref{cl-l1}] 
By $\Gamma= \Sigma^{-1}$, adding and subtracting $\hat{\Gamma} \hat{\Sigma}$ in the second equality
 \begin{eqnarray*}
 \hat{\Gamma} - \Gamma  & = & ( \hat{\Gamma} \Sigma - I_p ) \Gamma  = [ \hat{\Gamma} \hat{\Sigma} + \hat{\Gamma} ( \Sigma - \hat{\Sigma}) - I_p ] \Gamma \\
 & = & [ \hat{\Gamma} \hat{\Sigma} - I_p ] \Gamma + \hat{\Gamma} ( \Sigma - \hat{\Sigma}) \Gamma\\
 & = & [ \hat{\Gamma} \hat{\Sigma} - I_p ] \Gamma + \hat{\Gamma} ( I_p - \hat{\Sigma} \Gamma )
 \end{eqnarray*}
Next, by the definition of the CLIME program in section \ref{clime} $\| \hat{\Gamma} \hat{\Sigma} - I_p \|_{\infty} \le   \mu$, and using that we are on $T_{\Gamma} (\mu)$
 \begin{eqnarray*}
  \| \hat{\Gamma} - \Gamma \|_{\infty} &  \le  &
\|  [ \hat{\Gamma} \hat{\Sigma} - I_p ] \Gamma  \|_{\infty}+ \| \hat{\Gamma} ( I_p - \hat{\Sigma} \Gamma )\|_{\infty} \\
& \le & \|  \hat{\Gamma} \hat{\Sigma} - I_p  \|_{\infty}  \| \Gamma \|_{l_1} + 
\| \hat{\Gamma} \|_{l_{\infty}}  \|( I_p - \hat{\Sigma} \Gamma )\|_{\infty} \\
& \le & 2 m_{\Gamma} \mu,
\end{eqnarray*}
where we used dual norm inequality on p.44 of \cite{vdG16}, (\ref{ineq3}) and that $\| \hat{\Gamma} \|_{l_{\infty}} \le \| \Gamma\|_{l_{\infty}}$ on $T_\Gamma(\mu)$. Furthermore, since $\Gamma$ is symmetric, we have $\| \Gamma \|_{l_{\infty}} = \| \Gamma \|_{l_1} \le m_{\Gamma}$. 
\end{proof}

Recall from Section \ref{clime} that for $f \in [0,1)$ 
\[ U (m_{\Gamma}, f, s_{\Gamma}) = 
\{A\in\mathbb{R}^{p\times p}: A> 0, \|A \|_{l_1} \le m_{\Gamma}, \max_{1 \le j \le p} \sum_{k=1}^p | A_{jk}|^f \le s_{\Gamma}\}.\]
%
The next Lemma can be proved by using Lemma \ref{cl-l1} and adapting the proof of equation (27) on p.604-605 of \cite{cai11} to our purpose and its proof therefore omitted. Equation (27) is the proof of equation (14) in \cite{cai11}. For the purpose of the next lemma, define the constant $c_f = 1+ 2^{1-f} + 3^{1-f}$.
 \begin{lemma}\label{cl-l2}
Suppose that the conditions of Lemma \ref{cl-l1} hold and that $\Gamma \in U (m_{\Gamma}, f, s_{\Gamma})$. Then, for every $j\in\cbr[0]{1,...,p}$
 \[ \|\hat{\Gamma}_j - \Gamma_j \|_1 \le 2 c_f ( 2 m_{\Gamma} \mu)^{1-f} s_{\Gamma},\]
 \[  \|\hat{\Gamma}_j - \Gamma_j \|_2 \le 2 c_f ( 2 m_{\Gamma} \mu)^{1-f} s_{\Gamma},\] 
 \end{lemma}

The proof in \cite{cai11} also holds for non-symmetric matrices. Now we lower bound the probability of $T_{\Gamma} ( \mu)$. To that end define
\begin{equation}
c_{1n} = \frac{c_n}{(\min_{1 \le l \le q} \sigma_l^2 - c_n) \min_{1 \le l \le q} \sigma_l^2},\label{c1n}
\end{equation}
where $c_n$ is defined in Lemma \ref{l2}. Also recall that 
\[
 t_3 = 2 K [ C \frac{\sqrt{\ln (pq)}}{\sqrt{n}} + \frac{ \sqrt{E M_2^2} \ln (pq)}{n} ] + \sqrt{\frac{\kappa_n}{n}},
 \]
The following inequality, using the notation of Lemma \ref{lineq}, will be useful.
\begin{eqnarray}
\| B F A \|_{\infty} 
\le 
\| B \|_{\infty} \| F A \|_{l_1}  
 \le
 q  \| B \|_{\infty}  \| F \|_{l_{\infty}} \| A \|_{\infty}\label{ineqbfa}
\end{eqnarray}
where we used the dual norm inequality for the first inequality from p.44 of \cite{vdG16} and for the second inequality we used Lemma \ref{lineq}(vi).
We can now introduce the following new lemma for GMM in high dimensional models. The following lemma shows that the event $T_{\Gamma} ( \mu)$ happens with probability approaching one.

\begin{lemma}\label{cl-l3}
Under Assumptions \ref{1}, \ref{2} and \ref{4} one has
\[ P [ \| \Gamma \hat{\Sigma} - I_p \|_{\infty} > \mu ] \le 
10 \exp (- C \kappa_n) + \frac{K [ 2 E M_1^{2} + 5 E M_2^{2} + E M_3^{2} + E M_4^{2} + E M_5^{2}]}{n\kappa_n}\to 0,\]
where 
\[ \mu = m_{\Gamma} \{ (t_3)^2 c_{1n} + 2 C t_3 c_{1n} + C (t_3)^2 + 2 C t_3 + C  c_{1n}\}\to 0.\]

This result is valid uniformly over ${\cal B}_{l_0} (s_0)$ since $\mu$ depends on $c_{1n}$ which depends on $c_n$, and that depends on $\beta_0, s_0$ by Lemma A.5
.
\end{lemma}

\begin{proof}[Proof of Lemma \ref{cl-l3}]
We start by noting that $\Gamma \Sigma = I_p$ such that
\begin{eqnarray}
\| \Gamma \hat{\Sigma} - I_p \|_{\infty} & = &  \| \Gamma ( \hat{\Sigma} - \Sigma) \|_{\infty}  \nonumber \\
& \le & \| \Gamma \|_{l_{\infty}} \| \hat{\Sigma} - \Sigma \|_{\infty} = \| \Gamma \|_{l_1} \| \hat{\Sigma} - \Sigma \|_{\infty}.\label{la91}
\end{eqnarray}
where we used (\ref{ineq3}) and $\Gamma$ being symmetric. From the definitions of $\hat{\Sigma}$ and $\Sigma$, by simple algebra and the fact that the $\max$ norm of a transpose of a matrix is equal to $\max$ norm of a matrix: 
\begin{eqnarray}
\| \hat{\Sigma} - \Sigma \|_{\infty} & \le & \frac{1}{q}  \| (\frac{X'Z}{n} - \Sigma_{xz})(\hat{W}_d - W_d) (\frac{Z'X}{n} - \Sigma_{xz}') \|_{\infty} \label{la92} \\
& + & \frac{2}{q} \| (\frac{X'Z}{n} - \Sigma_{xz})(\hat{W}_d - W_d)  \Sigma_{xz}' \|_{\infty} \label{la93} \\
& + & \frac{1}{q}  \| (\frac{X'Z}{n} - \Sigma_{xz})( W_d) (\frac{Z'X}{n} - \Sigma_{xz}') \|_{\infty} \label{la94}\\
& + & \frac{2}{q} \| (\frac{X'Z}{n} - \Sigma_{xz})(W_d) (\Sigma_{xz}') \|_{\infty} \label{la95}\\
& + & \frac{1}{q} \| (\Sigma_{xz})(\hat{W}_d - W_d) (\Sigma_{xz}') \|_{\infty} \label{la96}
\end{eqnarray}

\noindent Before analyzing the individual terms in the above display note that if $\max_{1 \le l \le q} | \hat{\sigma}_l^2 - \sigma_l^2|\leq c_n$ (an event whose probability we can control by Lemma \ref{l2}, see in particular (\ref{a25})) then
\[ \| \hat{W}_d - W_d \|_{l_{\infty}} = \max_{1 \le l \le q} | \frac{1}{\hat{\sigma}_l^2} - \frac{1}{\sigma_l^2}| \le \frac{\max_{1 \le l \le q} | \hat{\sigma}_l^2 - \sigma_l^2|}{\min_{1 \le l \le q} \hat{\sigma}_l^2 \min_{1 \le l \le q} \sigma_l^2} \le \frac{c_n}{(\min_{1 \le l \le q} \sigma_l^2 - c_n ) \min_{1 \le l \le q} \sigma_l^2} = c_{1n}.\]
Assume furthermore that the following event occurs (the probability of which can be controlled by (\ref{a8}))
\[ \cbr[3]{\| \frac{X'Z}{n} - \Sigma_{xz} \|_{\infty} \le t_3}.\]
Using (\ref{ineqbfa}) we can upper bound (\ref{la92}) as follows on $\cbr[0]{\max_{1 \le l \le q} | \hat{\sigma}_l^2 - \sigma_l^2|\leq c_n}\cap \cbr[0]{\| \frac{X'Z}{n} - \Sigma_{xz} \|_{\infty} \le t_3}$
\begin{eqnarray}
\frac{1}{q}  \| (\frac{X'Z}{n} - \Sigma_{xz})(\hat{W}_d - W_d) (\frac{Z'X}{n} - \Sigma_{xz}') \|_{\infty} & \le & 
\| (\frac{X'Z}{n} - \Sigma_{xz}) \|_{\infty}  \| (\hat{W}_d - W_d)\|_{l_{\infty}} \|(\frac{Z'X}{n} - \Sigma_{xz}') \|_{\infty} \nonumber \\
& \le & (t_3)^2 c_{1n}\label{la97}
\end{eqnarray}
Consider (\ref{la93}). We have $\| \Sigma_{xz}' \|_{\infty} \le C < \infty$ by (\ref{a9}). By the same arguments as the ones that lead to (\ref{la97}) we have 
\begin{equation}
\frac{2}{q} \| (\frac{X'Z}{n} - \Sigma_{xz})(\hat{W}_d - W_d)  \Sigma_{xz}' \|_{\infty} \le 2 C t_3 c_{1n}.\label{la98}
 \end{equation}
 Consider (\ref{la94}). Note that $\|W_d \|_{l_{\infty}} = 1/\min_{1 \le l \le q} \sigma_l^2$. Using the same arguments as in (\ref{la97}) yields
 \begin{equation}
 \frac{1}{q}  \| (\frac{X'Z}{n} - \Sigma_{xz})( W_d) (\frac{Z'X}{n} - \Sigma_{xz}') \|_{\infty} \le (t_3)^2/\min_{1 \le l \le q} \sigma_l^2.\label{la99} 
 \end{equation}
 Consider (\ref{la95}) and (\ref{la96}). By the same analysis as the one that lead to (\ref{la97}) one gets
 \begin{equation}
 \frac{2}{q} \| (\frac{X'Z}{n} - \Sigma_{xz})(W_d) (\Sigma_{xz}') \|_{\infty} \le 2 C t_3/\min_{1 \le l \le q} \sigma_l^2.\label{la910}
  \end{equation}
 \begin{equation}
 \frac{1}{q} \| (\Sigma_{xz})(\hat{W}_d - W_d) (\Sigma_{xz}') \|_{\infty} \le C^2 c_{1n}.\label{la911}
  \end{equation}
 Combine all constants $C, C^2, (\min_{1 \le l \le q} \sigma_l^2)$ as $C$.  Then set via (\ref{la91}) and $\| \Gamma \|_{l_1} \le m_{\Gamma}$
 \[ \mu = m_{\Gamma} [(t_3)^2 c_{1n} + 2 C t_3 c_{1n} + C (t_3)^2  + 2 C t_3 + C c_{1n}].\]
Thus, we have that
 \begin{eqnarray*}
 P [ \| \Gamma \hat{\Sigma} - I _p \|_{\infty}  >   \mu] &\le&
  P [ \| \frac{X'Z}{n} - \Sigma_{xz} \|_{\infty} > t_3]  +P [ \max_{1 \le l \le q} | \hat{\sigma}_l^2 - \sigma_l^2|> c_n] \\
 & \le & \exp(- c \kappa_n) + \frac{K E M_2^{2}}{n\kappa_n} \\
 & + & 9 \exp (- c \kappa_n) + \frac{K [ 2 EM_1^{2} + 4 E M_2^{2} + E M_3^{2} + E M_4^{2} + E M_5^{2}]}{n\kappa_n}\to 0, 
     \end{eqnarray*}
 by (\ref{a8}) and (\ref{a25}) and the comment just above the latter as well as Assumption \ref{2} for the convergence to zero. It remains to be argued that $\mu\to 0$. Using Assumption \ref{2}, we see that by $p \le q$, 
 \[ \frac{\sqrt{E M_2^2}}{n} \ln (pq) = \left(\frac{\sqrt{EM_2^2} \sqrt{\ln pq}}{n^{1/2}}\right)  \frac{\sqrt{\ln pq}}{n^{1/2}} = o(1)  \frac{\sqrt{\ln q}}{n^{1/2}}.\]
Thus, $t_3$ as given in (\ref{t2s}) is $O( \frac{\sqrt{\ln q}}{n^{1/2}})$. Furthermore, $c_n$ as given in (\ref{a25}) is $O (s_0 \frac{\sqrt{\ln q}}{n^{1/2}})$ in Remark 3 of section \ref{8.4}, implying that the same is the case for $c_{1n}$. Therefore,
 \begin{equation}
 \mu = O ( m_{\Gamma} c_{1n}) = O (m_{\Gamma}s_0 \frac{\sqrt{\ln q}}{\sqrt{n}})=o(1).\label{rocmu}
 \end{equation}
where the last assertion is by Assumption \ref{4}.
 \end{proof}

The following lemma combines Lemmas \ref{cl-l2} and \ref{cl-l3}. Since Lemma \ref{cl-l2} conditions on the event $T_{\Gamma} (\mu)$ and Lemma \ref{cl-l3} provides the result that $T_{\Gamma} (\mu)$ happens with probability approaching one, we get Lemma \ref{cl-l4}.

\begin{lemma}\label{cl-l4}
Under Assumptions \ref{1}-\ref{4}, by using the program in section \ref{clime} to get $\hat{\Gamma}$

(i). \[ \max_{1 \le j \le p} \|\hat{\Gamma}_j - \Gamma_j \|_1 = O_p ( (m_{\Gamma} \mu)^{1-f} s_{\Gamma}).\]

(ii).\[ \max_{1 \le j \le p} \|\hat{\Gamma}_j - \Gamma_j \|_2 = O_p ( (m_{\Gamma} \mu)^{1-f} s_{\Gamma}).\]

This result is valid uniformly over ${\cal B}_{l_0} (s_0)$ since $\mu$ depends on $c_{1n}$ which depends on $c_n$, and that depends on $\beta_0, s_0$ by Lemma A.5
.

\end{lemma}

Note that the approximation errors in Lemma \ref{cl-l4} will converge in probability to zero by Assumption \ref{5cl}, and this will be seen in the proof of Theorem \ref{thmcl1}.

\subsection{Proof of Theorem \ref{thmcl1}}

\begin{proof}[Proof of Theorem \ref{thmcl1}]

We prove that $t_{W_d}$ is asymptotically standard normal. This will be done in case of a diagonal weight $W_d$. The case of general symmetric positive definite weight will be discussed afterwards. We divide the proof into several steps. 
 
First, decompose $t_{W_d}$:
\[ t_{W_d}= t_{W_{d1}} + t_{W_{d2}},\]
where 
\[ t_{W_{d1}} = \frac{e_j' \hat{\Gamma} \left( \frac{X'Z}{n} \frac{\hat{W}_d}{q}  \frac{Z'u}{n^{1/2}} \right)}{\sqrt{e_j' \hat{\Gamma} \hat{V}_d \hat{\Gamma} e_j}}.\]
\[ t_{W_{d2}} = -\frac{ e_j' \Delta}{\sqrt{e_j' \hat{\Gamma} \hat{V}_d \hat{\Gamma} e_j}}.\]

{\bf Step 1}. 

In the first step, we introduce an infeasible $t_{W_d^*}$ (it is infeasible since since $V_d = q^{-2} V_1$) and show that it is asymptotically standard normal.
\[ t_{W_{d1}^*} = \frac{e_j' \Gamma \Sigma_{x z}  \frac{W_d}{q}  Z' u /n^{1/2}}{\sqrt{e_j'  \Gamma V_d \Gamma e_j}}
= \frac{e_j' \Gamma \Sigma_{x z}  W_d  Z' u /n^{1/2}}{\sqrt{e_j'  \Gamma V_1 \Gamma e_j}}\]
where
\[
V_1 =  \Sigma_{x z}  W_d \Sigma_{z u} W_d  \Sigma_{x z}'.\]
Recall also that 
$\Sigma_{zu} =  E Z_i  Z_i' u_i^2,$ 
and 
$ \Sigma_{x z} = E X_i Z_i'.$

To establish that $t_{W_d^*}$ is standard normal, we verify the conditions for Lyapounov's central limit theorem. First, note that
\[ E \left[ \frac{e_j' \Gamma \Sigma_{x z} W_d \sum_{i=1}^n  Z_i  u_i /n^{1/2}}{\sqrt{e_j' \Gamma V_1 \Gamma e_j}} \right] = 0\]
since $E Z_i u_i = 0$ by exogeneity of the instruments. Next, 
\[ E \left[ \frac{e_j' \Gamma \Sigma_{x z} W_d \sum_{i=1}^n  Z_i  u_i /n^{1/2}}{\sqrt{e_j' \Gamma V_1 \Gamma e_j}} \right]^2 = 1,\]
where we used that
\[ E \left[ \frac{\sum_{i=1}^n Z_i u_i}{n^{1/2}}\right]\left[ \frac{\sum_{i=1}^n  Z_i u_i}{n^{1/2}}\right]'=  E  Z_i Z_i' u_i^2= \Sigma_{zu},\]
and $V_1$ definition above.
Next, we want to show 
\[ \frac{1}{(e_j' \Gamma V_1 \Gamma' e_j)^{r_u/4}} \sum_{i=1}^n E | e_j' \Gamma \Sigma_{x z}  W_d Z_i u_i /n^{1/2}|^{r_u/2} \to 0.\]
First, since $\Gamma_j$ is the $j$th row vector in $\Gamma$ in section \ref{clime}, with 
$\| \Gamma \|_{l_1} \le m_{\Gamma}$, and $\Gamma$ being symmetric we get
\begin{equation}
\max_{1 \le j \le p} \| \Gamma e_j\|_1= \max_{1 \le j \le p}  \|e_j' \Gamma\|_1 = \| \Gamma \|_{l_{\infty}} = \| \Gamma \|_{l_1} \le m_{\Gamma}.\label{tcl1}
 \end{equation}
\noindent We see that for every $i\in\cbr[0]{1,...,n}$ 
\begin{eqnarray}
 E | e_j' \Gamma \Sigma_{x z}  W_d  Z_i u_i / n^{1/2}|^{r_u/2} &  
\le & 
 E \left[  \| e_j' \Gamma  \|_1 \|\Sigma_{xz} W_d  \frac{Z_i u_i}{n^{1/2}} \|_{\infty} 
  \right]^{r_u/2} \nonumber \\
& \le 
&  E \{ [ m_{\Gamma}  \| \Sigma_{xz} W_d\|_{\infty} \| \frac{Z_i u_i }{\sqrt{n}} ]\|_{1}  \}^{r_u/2}
 \label{eqW} \\ 
 & \le 
&[ m_{\Gamma} \| \Sigma_{xz} \|_{\infty} \| W_d \|_{\infty}/\sqrt{n}  ]^{r_u/2}   E \sbr[3]{\sum_{l=1}^q|Z_{il}u_i| }^{r_u/2} \\
   & 
\le 
& O \left [  \left(\frac{ m_{\Gamma} }{n^{1/2}} \right)^{r_u/2} \right] \frac{q^{r_u/2}}{[\min_{1 \le l \le q} \sigma_l^2]^{r_u/2}}  \max_{1 \le l \le q } E | Z_{il} u_i |^{r_u/2} \nonumber \\
  & 
=
 & O \left(  \frac{m_{\Gamma}^{r_u/2} q^{r_u/2}  }{n^{r_u/4}}
  \right), \nonumber 
  \end{eqnarray}
where we used Hölder's inequality for the first inequality, and Jensen's inequality as well as $W_d$ being diagonal for the third. For the other ones we used Assumption \ref{1}, (\ref{a9}) and $\min_{1 \le l \le q} \sigma_l^2 >0$ with $\max_{1 \le l \le q} E | Z_{1l} u_1 |^{r_u/2}\le C < \infty$ by the Cauchy-Schwarz inequality with $r_u,r_z > 8$.
Therefore,
\begin{equation}
\sum_{i=1}^n E |e_j' \Gamma \Sigma_{x z} W_d  Z_i u_i /n^{1/2}|^{r_u/2} = O \left( \frac{m_{\Gamma}^{r_u/2} q^{r_u/2} }{n^{r_u/4-1}}\right).\label{num1}
\end{equation}
Next recalling that $V_d  q^2 =  V_1$ we get
\begin{eqnarray}
 [e_j' \Gamma V_1 \Gamma e_j ]^{r_u/4}&\ge& [Eigmin(V_1) \| \Gamma e_j \|_2^2]^{r_u/4}  \ge [Eigmin (V_1) Eigmin ( \Gamma)^2 \|e_j \|_2^2]^{r_u/4}
 \nonumber \\
&= & q^{r_u/2} [Eigmin (V_d) \frac{1}{Eigmax (\Sigma)^2}]^{r_u/4} > 0,\label{den1}
\end{eqnarray}
where by Assumption \ref{5cl}(ii), we have that $Eigmin (V_d)$ is bounded away from zero and $Eigmax (\Sigma)$ is bounded from above. Thus, by dividing (\ref{num1}) with (\ref{den1}), the Lyapounov conditions are seen to be satisfied by Assumption \ref{5cl}(i). Therefore, $t_{W_{d1}^*} \stackrel{d}{\to} N(0,1)$. 


{\bf Step 2}.
Here we show $t_{W_{d1}} - t_{W_{d1}^*} = o_p (1)$. 

We do so by showing that the numerators as well as denominators of $t_{W_{d1}} - t_{W_{d1}^*}$ are asymptotically equivalent and by arguing that the denominators are bounded away from 0 in probability.

Step 2a). Regarding the numerators note that

\begin{eqnarray}
&&| e_j' \hat{\Gamma}\left( \frac{X'Z}{n} \right) \frac{\hat{W}_d}{q}  \frac{Z'u}{n^{1/2}}  - e_j' \Gamma \Sigma_{xz} \frac{W_d}{q} \frac{Z'u}{n^{1/2}}| \nonumber \\
& \le  & |e_j' \hat{\Gamma} \left(\frac{X'Z}{n} \right) \frac{\hat{W}_d}{q}  \frac{Z'u}{n^{1/2}} - e_j' \Gamma \left( \frac{X'Z}{n} \right) \frac{\hat{W}_d}{q}  \frac{Z'u}{n^{1/2}}| \label{t21} \\
& + & |e_j' \Gamma \left( \frac{X'Z}{n} \right) \frac{\hat{W}_d}{q}  \frac{Z'u}{n^{1/2}} - e_j' \Gamma \Sigma_{xz} \frac{\hat{W}_d}{q}  \frac{Z'u}{n^{1/2}}| \label{t22} \\
& + & | e_j' \Gamma \Sigma_{xz}  \frac{\hat{W}_d}{q}  \frac{Z'u}{n^{1/2}} - e_j' \Gamma \Sigma_{xz} \frac{W_d}{q} \frac{Z'u}{n^{1/2}}| \label{t23}
\end{eqnarray}

Start with (\ref{t21}). By Hölder's inequality
\begin{equation}
|e_j' \hat{\Gamma} \left(\frac{X'Z}{n} \right) \frac{\hat{W}_d}{q}  \frac{Z'u}{n^{1/2}} - e_j' \Gamma \left( \frac{X'Z}{n} \right) \frac{\hat{W}_d}{q}  \frac{Z'u}{n^{1/2}}|
\le \left[ \| e_j' (\hat{\Gamma} - \Gamma )\|_1  \right] \left[ \| \frac{X'Z}{n} \frac{\hat{W}_d}{q} \frac{Z'u}{n^{1/2}} \|_{\infty} \right].\label{t24}
\end{equation}
Next,

\begin{eqnarray}
\| \frac{X'Z}{n} \frac{\hat{W}_d}{q} \frac{Z'u}{n^{1/2}} \|_{\infty} &\le& [ \max_{1 \le j \le p} \max_{1 \le l \le q} \frac{\sum_{i=1}^n |X_{ij} Z_{il}|}{n} ] \|  \hat{W}_d \frac{Z'u}{n^{1/2}} \|_{\infty} \nonumber \\
& \le & [ \max_{1 \le j \le p} \max_{1 \le l \le q} \frac{\sum_{i=1}^n |X_{ij} Z_{il}|}{n} ] \left( \frac{1}{\min_{1 \le l \le q} \hat{\sigma}_l^2} \right) \| \frac{Z'u}{n^{1/2}} \|_{\infty},\label{t25}
\end{eqnarray}
where we used (\ref{ineq2}) for the first inequality and for the second inequality we used (\ref{ineq3})
and $\|\hat{W}_d \|_{l_{\infty}} = \frac{1}{\min_{1 \le l \le q} \hat{\sigma}_l^2}$.
Use (\ref{t25}) in (\ref{t24}) to get
\begin{eqnarray}
&&|e_j' \hat{\Gamma} \left(\frac{X'Z}{n} \right) \frac{\hat{W}_d}{q}  \frac{Z'u}{n^{1/2}} - e_j' \Gamma \left( \frac{X'Z}{n} \right) \frac{\hat{W}_d}{q}  \frac{Z'u}{n^{1/2}}| \nonumber \\
&\le & \left[ \| e_j' (\hat{\Gamma} - \Gamma )\|_1  \right] [ \max_{1 \le j \le p} \max_{1 \le l \le q} \frac{\sum_{i=1}^n |X_{ij} Z_{il}|}{n} ] \left( \frac{1}{\min_{1 \le l \le q} \hat{\sigma}_l^2} \right) \| \frac{Z'u}{n^{1/2}} \|_{\infty} \nonumber \\
& = & O_p \left(  s_{\Gamma} (m_{\Gamma} \mu)^{1-f} \right) \left[  O_p( 1) \right] O_p (1) \left[ \sqrt{n}  O_p \left( \frac{\sqrt{\ln q}}{n^{1/2}} \right) 
\right]\nonumber \\
& = & O_p \left( s_{\Gamma} (m_{\Gamma} \mu)^{1-f} \sqrt{\ln q} \right) = o_p (1),\label{t26}
\end{eqnarray}
where in the first equality we use Lemma \ref{cl-l4} for the first term on the right side, Lemma \ref{bound}(iv) for the second term, Lemma \ref{l2} for the third term and Lemma \ref{bound}(ii) for the fourth term. The last equality follows by Assumption \ref{5cl}(i). Regarding (\ref{t22}) note first that

\begin{eqnarray}
&&|e_j' \Gamma \left( \frac{X'Z}{n} \right) \frac{\hat{W}_d}{q}  \frac{Z'u}{n^{1/2}} - e_j' \Gamma \Sigma_{xz} \frac{\hat{W}_d}{q}  \frac{Z'u}{n^{1/2}}| \nonumber \\
& \le & \| e_j' \Gamma \|_1 \| (\frac{X'Z}{n} - \Sigma_{xz}) \frac{\hat{W}_d}{q} \frac{Z'u}{n^{1/2}} \|_{\infty}  \nonumber \\
&\le & \|\Gamma_j \|_1  \left[ \max_{1 \le j \le p} \max_{1 \le l \le q } \frac{1}{n} \sum_{i=1}^n |X_{ij} Z_{il} - E X_{ij} Z_{il}| \right] \|\hat{W}_d \frac{Z'u}{n^{1/2}} \|_{\infty} \nonumber \\
& \le &  \|\Gamma_j \|_1  \left[ \max_{1 \le j \le p} \max_{1 \le l \le q } \frac{1}{n} \sum_{i=1}^n |X_{ij} Z_{il} - E X_{ij} Z_{il}| \right] \|\hat{W}_d \|_{l_{\infty}} \| \frac{Z'u}{n^{1/2}} \|_{\infty} \nonumber \\
& = & \|\Gamma_j \|_1  \left[ \max_{1 \le j \le p} \max_{1 \le l \le q } \frac{1}{n} \sum_{i=1}^n |X_{ij} Z_{il} - E X_{ij} Z_{il}| \right] \frac{1}{\min_{1 \le l \le q} \hat{\sigma}_l^2} \| \frac{Z'u}{n^{1/2}} \|_{\infty},
\label{t27}
\end{eqnarray}
where we used Hölder's inequality for the first inequality, (\ref{ineq2}) for the second inequality and for third inequality we used (\ref{ineq3})
Observe that by Assumption \ref{4}(i)
$\max_{1 \le j \le p} \|\Gamma_j \|_1 = O (m_{\Gamma} )$ and by Lemma \ref{bound}(iv), and $p \le q$ such that $\ln (pq) \le 2 \ln q$ and so
\[ \max_{1 \le j \le p} \max_{1 \le l \le q} \frac{1}{n} \sum_{i=1}^n |X_{ij} Z_{il} - E X_{ij} Z_{il} | = O_p \left( \frac{\sqrt{\ln q}}{n^{1/2}}\right ).\]
By Lemma \ref{l2} and \ref{bound}(ii) we have 
\[ \frac{1}{\min_{1 \le l \le q} \hat{\sigma}_l^2} \| \frac{Z'u}{n^{1/2}} \|_{\infty} = \sqrt{n} O_p (1) O_p ( \frac{\sqrt{\ln q}}{n^{1/2}}) = O_p \left( \sqrt{\ln q}\right).\]
Using the above two displays in (\ref{t27}) yields 

\begin{eqnarray}
&&|e_j' \Gamma \left( \frac{X'Z}{n} \right) \frac{\hat{W}_d}{q}  \frac{Z'u}{n^{1/2}} - e_j' \Gamma \Sigma_{xz} \frac{\hat{W}_d}{q}  \frac{Z'u}{n^{1/2}}| \nonumber \\
& = & O (m_{\Gamma}) O_p \left( \frac{\sqrt{\ln q}}{n^{1/2}} \right)   O_p ( \sqrt{\ln q}) \nonumber \\
& = & O_p \left( \frac{m_{\Gamma} \ln q}{n^{1/2}}\right) = o_p (1),\label{t28}
\end{eqnarray}
by Assumption \ref{5cl}(i). Now consider (\ref{t23}):
\begin{eqnarray}
&& | e_j' \Gamma \Sigma_{xz}  \frac{\hat{W}_d}{q}  \frac{Z'u}{n^{1/2}} - e_j' \Gamma \Sigma_{xz} \frac{W_d}{q} \frac{Z'u}{n^{1/2}}| \nonumber \\
& \le&  
\| e_j' \Gamma \|_1 \| \Sigma_{xz} \frac{(\hat{W}_d - W_d)}{q} \frac{Z'u}{n^{1/2}} \|_{\infty} \nonumber \\
& \le & \| \Gamma_j \|_1  [ \max_{1 \le j \le p} \max_{1 \le l \le q} \frac{1}{n} \sum_{i=1}^n E | X_{ij} Z_{il}|] \|(\hat{W}_d - W_d) \frac{Z'u}{n^{1/2}} \|_{\infty} \nonumber \\
& \le & \| \Gamma_j \|_1  [ \max_{1 \le j \le p} \max_{1 \le l \le q} \frac{1}{n} \sum_{i=1}^n E | X_{ij} Z_{il}|] \|(\hat{W}_d - W_d) \|_{l_{\infty}} \| \frac{Z'u}{n^{1/2}} \|_{\infty} \nonumber \\
& = & \| \Gamma_j \|_1  [ \max_{1 \le j \le p} \max_{1 \le l \le q} \frac{1}{n} \sum_{i=1}^n E | X_{ij} Z_{il}|] [ \max_{1 \le l \le q} \frac{1}{|\hat{\sigma}_l^2 - \sigma_l^2|}] \| \frac{Z'u}{n^{1/2}} \|_{\infty} \nonumber \\
&= & O(m_{\Gamma}) O(1) O_p \left( \frac{ \sqrt{\ln q}  s_0}{\sqrt{n}}\right) [ n^{1/2} O_p ( \frac{\sqrt{\ln q}}{n^{1/2}})] \nonumber \\
& = & O_p \left( \frac{m_{\Gamma}  s_0 \ln q }{n^{1/2}} \right) = o_p (1),
\label{t29}
\end{eqnarray}

\noindent where we use Hölder's inequality for the first inequality, (\ref{ineq2}) for the second and for third inequality we used  (\ref{ineq3}).
Assumption \ref{1}, \ref{4}(i), Lemma \ref{bound} as well as the following display are used as well. The last equality is obtained by Assumption \ref{5cl}(i).
In (\ref{t29}) we used that
\begin{eqnarray}
\max_{1 \le l \le q} \left| \frac{1}{\hat{\sigma}_l^2} - \frac{1}{\sigma_l^2} \right| & = & \max_{1 \le l \le q} \left| \frac{\sigma_l^2 - \hat{\sigma}_l^2}{\hat{\sigma}_l^2 \sigma_l^2} \right| \nonumber \\
& \le & \frac{\max_{1 \le l \le q} | \hat{\sigma}_l^2 - \sigma_l^2|}{\min_{1 \le l \le q} \hat{\sigma}_l^2 \min_{1 \le l \le q} \sigma_l^2} \nonumber \\
& = & O_p \left(\frac{\sqrt{\ln q} s_0}{n^{1/2}}\right),\label{t210}
\end{eqnarray}
where we obtained the rate in the last equality from Lemma \ref{rem:sigmas}, Remark \ref{rem:sigmas}. Note that the result is uniform over ${\cal B}_{l_0} (s_0)$ by Lemma \ref{rem:sigmas}, Lemma \ref{cl-l3}, \ref{cl-l4}, and the step 2a proof here.


Step 2b). Here we start analyzing the denominator of $t_{W_{d1}}$. As an intermediate step define the infeasible estimator $\tilde{V}_d = 
(\frac{X'Z}{n} \frac{\hat{W}_d}{q} \tilde{\Sigma}_{zu} \frac{\hat{W}_d}{q} \frac{Z'X}{n})$ of $V_d$, where $\tilde{\Sigma}_{zu} = \frac{1}{n} \sum_{i=1}^n
Z_i Z_i' u_i^2$.
We show that
\begin{equation}
 |e_j' \hat{\Gamma} \hat{V}_d \hat{\Gamma}' e_j - e_j' \Gamma V_d \Gamma' e_j | = o_p (1).\label{t210a}
 \end{equation}
To this end, consider the following three terms:
\begin{equation}
|e_j' \hat{\Gamma} \hat{V}_d  \hat{\Gamma}' e_j - e_j' \hat{\Gamma} \tilde{V}_d  \hat{\Gamma}' e_j |.\label{t211} 
\end{equation} 
\begin{equation}
|e_j' \hat{\Gamma} \tilde{V}_d  \hat{\Gamma}' e_j - e_j' \hat{\Gamma} V_d  \hat{\Gamma}' e_j |.\label{t212} 
\end{equation} 
\begin{equation}
|e_j' \hat{\Gamma} V_d  \hat{\Gamma}' e_j - e_j' \Gamma V_d  \Gamma' e_j |.\label{t213} 
\end{equation}
To establish (\ref{t210a}) we show that the above three terms tend to zero in probability. We start with (\ref{t211}). Use Hölder's inequality twice to get

\begin{equation}
|e_j' \hat{\Gamma} \hat{V}_d  \hat{\Gamma}' e_j - e_j' \hat{\Gamma} \tilde{V}_d  \hat{\Gamma}' e_j | \le \| \hat{V}_d - \tilde{V}_d \|_{\infty} \|\hat{\Gamma}' e_j \|_1^2
.\label{t212c}
\end{equation}

\noindent Then, in (\ref{t212c}), by the definition of $\hat{V}_d$ and $\tilde{V}_d$
\begin{eqnarray}
 \| \hat{V}_d - \tilde{V}_d \|_{\infty}  &= & \| \left( \frac{X'Z}{n} \right) \frac{\hat{W}_d}{q} \hat{\Sigma}_{zu} \frac{\hat{W}_d}{q} \left(  \frac{Z'X}{n}
 \right) - \left( \frac{X'Z}{n} \right) \frac{\hat{W}_d}{q} \tilde{\Sigma}_{zu} \frac{\hat{W}_d}{q} \left(  \frac{Z'X}{n}
 \right)\|_{\infty} \nonumber \\
 & \le & [\| \left( \frac{X'Z}{n} \right) \frac{\hat{W}_d}{q} \|_{l_{\infty}}]^2 \| \| \hat{\Sigma}_{zu} - \tilde{\Sigma}_{zu}\|_{\infty}, \label{t213a}  \end{eqnarray}
where we used Lemma \ref{lineq} (iii). Now
 \begin{equation}
 \| \left( \frac{X'Z}{n} \right) \frac{\hat{W}_d }{q} \|_{l_{\infty}} \le   \| \left( \frac{X'Z}{n} \right) \|_{\infty}  \|\hat{W}_d \|_{l_1},\label{t214}
 \end{equation}
 by Lemma \ref{lineq}(iv). Furthermore,
 \begin{equation}
 \| \left( \frac{X'Z}{n} \right) \|_{\infty} = O_p(1),\label{t216}
 \end{equation}
 by Lemma \ref{bound}(iv). Next by Lemma \ref{l2} and $\hat{W}_d$ being diagonal
 \begin{equation}
 \| \hat{W}_d \|_{l_1} =\| \hat{W}_d \|_{\infty}= O_p (1).\label{t218}
 \end{equation}

\noindent Now insert (\ref{t216}) and (\ref{t218}) into (\ref{t214}) to conclude
 \begin{equation}
 \| \left( \frac{X'Z}{n} \right) \frac{\hat{W}_d}{q} \|_{l_{\infty}} = O_p (1).\label{t220}
 \end{equation}
Recalling that $\hat{u}_i = u_i - X_i' (\hat{\beta}_F - \beta_0)$ one gets
 \begin{equation}
 \hat{\Sigma}_{zu} - \tilde{\Sigma}_{zu} = - \frac{2}{n} \sum_{i=1}^n Z_i Z_i' u_i X_i' (\hat{\beta}_F - \beta_0)
 + \frac{1}{n} \sum_{i=1}^n Z_i Z_i' (\hat{\beta}_F - \beta_0)' X_i X_i' (\hat{\beta}_F - \beta_0).\label{t221}
 \end{equation}
 Consider  the first term on the right side of (\ref{t221}):
 {\small \begin{eqnarray}
\max_{1 \le i \le n}  \max_{1 \le l \le q} \max_{1 \le m \le q} \left| \frac{2}{n} \sum_{i=1}^n Z_{il} Z_{im} u_i X_i' (\hat{\beta}_F - \beta_0) \right| & \le & 
 [ \max_{1 \le i \le n} \max_{1 \le l \le q} \max_{1 \le m \le q} |Z_{il} Z_{im}| ] \left| \frac{2}{n} \sum_{i=1}^n u_i X_i' (\hat{\beta}_F - \beta_0)
 \right| \nonumber \\
 & \le & 2 \max_{1 \le i \le n} \max_{1 \le l \le q} \max_{1 \le m \le q} | Z_{il} Z_{im}|  \|  \frac{u'X}{n}\|_{\infty} \|\hat{\beta}_F - \beta_0 \|_1.\label{t222}
 \end{eqnarray}}
Next, by Markov's inequality and via Lemma A.3 of \cite{ck18} which requires $\max_{1 \le l \le q} E | Z_{il}|^{r_z} \le C < \infty$,
 \[ P \left(\max_{1 \le i \le n}  \max_{1 \le l \le q} \max_{1 \le m \le q} | Z_{il} Z_{im} | > t_7 
 \right) \le \frac{n q^2 C }{t_7^{r_z/2}},\]
 where $t_7= M q^{4/r_z} n^{2/r_z}$ for $r_z>12$. This shows that for a large positive constant $M >0$
 \begin{equation}
\max_{1 \le i \le n}  \max_{1 \le l \le q} \max_{1 \le m \le q} | Z_{il} Z_{im} | = O_p ( q^{4/r_z} n^{2/r_z}).\label{t223}
  \end{equation}
Next, by Lemmas \ref{su2}-\ref{su3} and Assumption \ref{5cl}(iii)
\[
 P \left( \max_{1 \le j \le p} \left| n^{-1} \sum_{i=1}^n X_{ij} u_i - E X_{ij} u_i 
 \right| > t_8 \right) \le \exp(- C \kappa_n) + \frac{ E M_6^{2}}{n \kappa_n} = o(1),
 \]
 where $t_8= O(\frac{\sqrt{\ln p}}{n^{1/2}})$.
 Next, by Assumption \ref{1} and the Cauchy-Schwarz inequality
 \[ \max_{1 \le j \le p} \left|  E X_{ij} u_i \right| = O(1).\]
Combining the above two displays gives
\begin{equation}
\| \frac{1}{n} u'X \|_{\infty} = O(1) + O_p ( \frac{\sqrt{\ln p}}{n^{1/2}}) = O(1) + o_p (1),\label{t224}
\end{equation}
where the $o_p(1)$ term is obtained by Assumption \ref{2}. Now use (\ref{t223}) and (\ref{t224}) in (\ref{t222}) together with Lemma \ref{l1}(ii) with Assumption 2 to get $\lambda_n = O (\sqrt{lnq}/\sqrt{n})$,
\begin{eqnarray}
\left\| \frac{2}{n} \sum_{i=1}^n Z_i Z_i' u_i X_i' (\hat{\beta}_F - \beta_0) \right\|_{\infty} & \le & 
O_p ( q^{4/r_z} n^{2/r_z}) [ O(1) + o_p (1) ] O_p ( \frac{\sqrt{lnq} s_0}{n^{1/2}}) \nonumber \\
& = & O_p (\frac{s_0 (\sqrt{lnq}) q^{4/r_z}  n^{2/r_z}}{n^{1/2}} ) = o_p (1),\label{t225}
\end{eqnarray}
where Assumption \ref{5cl}(iv), $  m_{\Gamma}^2 s_0 q^{4/r_z} n^{2/r_z} \sqrt{lnq}/n^{1/2} = o(1)$ implies  the last equality. Now analyze the following 
in (\ref{t221})

\begin{eqnarray*}
\left\| \frac{1}{n} \sum_{i=1}^n Z_i Z_i' (\hat{\beta}_F - \beta_0)' X_i X_i' (\hat{\beta}_F - \beta_0) \right\|_{\infty} & \le & 
\max_{1 \le i \le n} \max_{1 \le l \le q} \max_{1 \le m \le q} | Z_{il} Z_{im}| (\hat{\beta}_F - \beta_0)' [ \frac{1}{n} \sum_{i=1}^n X_i X_i'] (\hat{\beta}_F - \beta_0) \\
& \le & \max_{1 \le i \le n} 
\max_{1 \le l \le q} \max_{1 \le m \le q} | Z_{il} Z_{im}| \| \frac{X'X}{n} \|_{\infty} \|\hat{\beta}_F - \beta_0 \|_1^2,
\end{eqnarray*}
where we used Hölder's inequality twice for the last estimate. 

\noindent By Lemmas \ref{su2} and \ref{su3}
\[ \| \frac{1}{n} \sum_{i=1}^n [X_i X_i' - E X_i X_i'] \|_{\infty} = O_p ( \frac{\sqrt{\ln p}}{n^{1/2}}).\]
and by the Cauchy-Schwarz inequality with $r_x\ge 6$ bounded moments.
\[ \| E X_i X_i' \|_{\infty}  =  O(1).\]
By Assumption 2 the previous two displays imply
\[ \| \frac{X'X}{n} \|_{\infty} = O (1) + O_p ( \frac{\sqrt{\ln p}}{n^{1/2}}) = O(1) + o_p (1).\]
Then, using (\ref{t223}), the above display and Lemma \ref{l1}
\begin{eqnarray}
\left| \frac{1}{n} \sum_{i=1}^n Z_i Z_i' (\hat{\beta}_F - \beta_0)' X_i X_i' (\hat{\beta} - \beta_0) \right| & \le & 
O_p ( q^{4/r_z} n^{2/r_z}) [ O(1) + o_p (1)] O_p \left( \frac{\ln q s_0^2}{n} \right) \nonumber \\
& = & O_p ( \{\frac{\sqrt{\ln q} s_0 n^{1/r_z} q^{2/r_z}}{n^{1/2}} \}^2) \nonumber \\
&=& o_p (1),\label{t226}
\end{eqnarray}
by Assumption \ref{5cl}(iv).
Using (\ref{t226}) and (\ref{t225}) in (\ref{t221}) thus gives
\begin{equation}
\| \hat{\Sigma}_{zu} - \tilde{\Sigma}_{zu} \|_{\infty}  = O_p ( \frac{q^{4/r_z} s_0 n^{2/r_z} \sqrt{\ln q}}{n^{1/2}}) = o_p (1).\label{t227}
\end{equation}
Insert (\ref{t220}) and (\ref{t227}) into (\ref{t213a}) to get 
\begin{equation}
\| \hat{V}_d - \tilde{V}_d \|_{\infty} = O_p ( \frac{q^{4/r_z} s_0 n^{2/r_z} \sqrt{\ln q}}{n^{1/2}})=o_p(1).\label{t228}
\end{equation}
By the definition of the CLIME program one has for $j=1,..., p$
\[ \| \hat{\Gamma}_j \|_1 \le \| \Gamma_j \|_1.\]
By $\Gamma \in U (m_{\Gamma}, f, s_{\Gamma})$ being symmetric and Lemma \ref{cl-l3} one has with probability approaching one
\begin{equation}
\max_{1 \le j \le p} \| \hat{\Gamma}_j \|_1 = \| \hat{\Gamma}\|_{l_{\infty}} \le \| \Gamma \|_{l_{\infty}} = \| \Gamma\|_{l_1}
\le m_{\Gamma}.\label{gel1}
\end{equation}
Next, use (\ref{t228}) and (\ref{gel1}) in (\ref{t212c}) to bound (\ref{t211}), by Assumption \ref{5cl}(iv)
\begin{eqnarray}
| e_j' \hat{\Gamma} \hat{V}_d \hat{\Gamma}' e_j - e_j' \hat{\Gamma} \tilde{V}_d \hat{\Gamma}' e_j | & \le  & 
O_p ( m_{\Gamma}^2) O_p ( \frac{q^{4/r_z} s_0 n^{2/r_z} \sqrt{\ln q}}{n^{1/2}}) \nonumber \\
& =& O_p ( \frac{m_{\Gamma}^2 s_0 q^{4/r_z} n^{2/r_z} \sqrt{\ln q}}{n^{1/2}})= o_p (1).\label{t229}
\end{eqnarray}

We now turn to (\ref{t212}) and note first that

\begin{eqnarray}
&& |e_j' \hat{\Gamma} \tilde{V}_d  \hat{\Gamma}' e_j - e_j' \hat{\Gamma} V_d  \hat{\Gamma}' e_j | \nonumber \\
& \le & |e_j' \hat{\Gamma} \tilde{V}_d \hat{\Gamma}' e_j - e_j' \hat{\Gamma} \bar{V}_d \hat{\Gamma}' e_j | \label{t230} \\
& + & | e_j' \hat{\Gamma} \bar{V}_d \hat{\Gamma}' e_j -  e_j' \hat{\Gamma} V_d \hat{\Gamma}' e_j | \label{t231},
\end{eqnarray} 
where 
\[ \bar{V}_d = ( \frac{X'Z}{n} \frac{\hat{W}_d}{q} \Sigma_{zu}  \frac{\hat{W}_d}{q} \frac{Z'X}{n} ).\]

We bound (\ref{t230}) and (\ref{t231}) separately. Start with (\ref{t230}).
\begin{equation}
 |e_j' \hat{\Gamma} \tilde{V}_d \hat{\Gamma}' e_j - e_j' \hat{\Gamma} \bar{V}_d \hat{\Gamma}' e_j |  \le \|e_j' \hat{\Gamma} \|_1^2 
 \| \tilde{V}_d - \bar{V}_d \|_{\infty}.\label{t232}
\end{equation}

Then by Lemma \ref{lineq}(iii) 
\begin{equation}
\| \tilde{V}_d - \bar{V}_d \|_{\infty}  \le [ \|\frac{X'Z}{n} \frac{\hat{W}_d}{q}  \|_{l_{\infty}} ]^2\| 
\| \tilde{\Sigma}_{zu} - \Sigma_{zu} \|_{\infty}.\label{t233}
\end{equation}

By Lemmas \ref{su2} and \ref{su3} 
\[
P \left(  \max_{1 \le l \le q} \max_{1 \le m \le q}  \left| n^{-1} \sum_{i=1}^n Z_{il} Z_{im} u_i^2 - E Z_{il} Z_{im} u_i^2 \right| >t_9 
\right) \le \exp (- C \kappa_n) + \frac{K E M_7^{2}}{n \kappa_n} = o(1) ,
\]
for a $t_9 =  O ( \frac{\sqrt{\ln q}}{n^{1/2}})$ via Assumption \ref{5cl}(iii).
Thus, 
 \begin{equation}
 \| \tilde{\Sigma}_{zu} - \Sigma_{zu} \|_{\infty} = O_p ( \frac{\sqrt{\ln q}}{n^{1/2}}).\label{t234}
 \end{equation}
 Then insert (\ref{t220}) and (\ref{t234}) in (\ref{t233}) to get that
 \begin{equation}
 \| \tilde{V}_d  - \bar{V}_d \|_{\infty} = O_p ( \frac{\sqrt{\ln q}}{n^{1/2}}) .\label{t235}
 \end{equation}
 Use (\ref{t235}) in (\ref{t232}) together with (\ref{gel1})
 \begin{eqnarray}
 |e_j' \hat{\Gamma} \tilde{V}_d \hat{\Gamma}' e_j - e_j' \hat{\Gamma} \bar{V}_d \hat{\Gamma}' e_j | =  
 O_p (m_{\Gamma}^2) O_p ( \frac{\sqrt{\ln q}}{n^{1/2}}) 
 =
 o_p (1),\label{t236}
 \end{eqnarray}
 by Assumption \ref{5cl}(iv).

We now turn to (\ref{t231}) and begin by noting that
 \begin{equation}
 | e_j' \hat{\Gamma} \bar{V}_d \hat{\Gamma}' e_j - e_j' \hat{\Gamma} V_d \hat{\Gamma}' e_j |
 \le \|e_j' \hat{\Gamma} \|_1^2 \| \bar{V}_d - V_d \|_{\infty}.\label{t237}
 \end{equation}
Next, by definition of $\bar{V}_d$ and $V_d$, and addition and subtraction
 \begin{eqnarray}
 \bar{V}_d - V_d & = & \left( \frac{X'Z}{n} \right) \frac{\hat{W}_d}{q}  \Sigma_{zu}  \frac{\hat{W}_d}{q} \left( \frac{Z'X}{n} \right) 
 - \Sigma_{xz} \frac{W_d}{q} \Sigma_{zu} \frac{W_d}{q} \Sigma_{xz}' \nonumber \\
 & = & \left[\left( \frac{X'Z}{n}  \frac{\hat{W}_d}{q}   - \Sigma_{xz} \frac{W_d}{q} \right) \Sigma_{zu}  \left( \frac{X'Z}{n}  \frac{\hat{W}_d}{q}  - \Sigma_{xz} \frac{W_d}{q} \right)'   \right]
 \label{t238a}\\
 & + & \left[\left( \frac{X'Z}{n}  \frac{\hat{W}_d}{q}  - \Sigma_{xz} \frac{W_d}{q} \right) \Sigma_{zu} \frac{W_d}{q} \Sigma_{xz}' \right] \label{t238b} \\
 & + & \left[\left( \frac{X'Z}{n}  \frac{\hat{W}_d}{q}  - \Sigma_{xz} \frac{W_d}{q} \right) \Sigma_{zu} \frac{W_d}{q} \Sigma_{xz}' \right]'.\label{t238} 
  \end{eqnarray}
We proceed by bounding each of the terms in the above display. Consider first (\ref{t238a}).
  \begin{eqnarray}
&& \|  \left( \frac{X'Z}{n}  \frac{\hat{W}_d}{q}  - \Sigma_{xz} \frac{W_d}{q} \right) \Sigma_{zu}  \left( \frac{X'Z}{n}  \frac{\hat{W}_d}{q}  - \Sigma_{xz} \frac{W_d}{q} \right)'   
 \|_{\infty} \nonumber \\
 & \le & \| \left( \frac{X'Z}{n}  \frac{\hat{W}_d}{q}  - \Sigma_{xz} \frac{W_d}{q} \right)  \|_{l_{\infty}} ^2
   \| \Sigma_{zu} \|_{\infty},\label{t239}
 \end{eqnarray}
where we used Lemma \ref{lineq} (iii) for the inequality. Consider the first term on the right side of (\ref{t239}):
\begin{eqnarray}
  \| \left( \frac{X'Z}{n}  \frac{\hat{W}_d}{q}  - \Sigma_{xz} \frac{W_d}{q} \right)  \|_{l_{\infty}} & = & 
    \| \left( \frac{X'Z}{n}  \frac{\hat{W}_d}{q}  - \Sigma_{xz} \frac{\hat{W}_d}{q} + \Sigma_{xz} \frac{\hat{W}_d}{q} - \Sigma_{xz} \frac{W_d}{q} \right)  \|_{l_{\infty}} \nonumber \\
    & \le &  \|  \frac{X'Z}{n}  \frac{\hat{W}_d}{q}  - \Sigma_{xz} \frac{\hat{W}_d}{q} \|_{l_{\infty}}  + \| \Sigma_{xz} \frac{\hat{W}_d}{q} - \Sigma_{xz} \frac{W_d}{q} \  \|_{l_{\infty}}         
    \nonumber \\
    & \le & q  \| \frac{X'Z}{n} - \Sigma_{xz} \|_{\infty} \| \frac{\hat{W}_d}{q} \|_{l_1} + 
    q \| \Sigma_{xz} \|_{\infty} \|\frac{\hat{W}_d - W_d}{q} \|_{l_1} \nonumber \\
    & = &  \| \frac{1}{n} \sum_{i=1}^n (X_i Z_i' - E X_i Z_i')  \|_{\infty}  \| \hat{W}_d \|_{l_1}  \nonumber \\
   & + & \| \frac{1}{n} \sum_{i=1}^n E X_i Z_i' \|_{\infty} \|\hat{W}_d - W_d  \|_{l_1},\label{t240}
    \end{eqnarray} 
where we used triangle inequality for the first inequality and Lemma \ref{lineq}(iv) for the second inequality. Consider the terms on the right-side of 
 (\ref{t240}).  By Lemma \ref{bound}(iv), Lemma \ref{l2} and $\hat{W}_d$ being diagonal \begin{eqnarray}
 \| \frac{1}{n} \sum_{i=1}^n (X_i Z_i' - E X_i Z_i')  \|_{\infty}  \| \hat{W}_d \|_{l_1}   =  O_p \left( \frac{\sqrt{\ln q}}{n^{1/2}}
 \right) .\label{t241}   
 \end{eqnarray} 
 Next, arguing as in (\ref{t210}) we obtain
 \begin{eqnarray}
  \| E X_i Z_i' \|_{\infty} \|\hat{W}_d - W_d  \|_{l_1}  = O_p  \left( \frac{\sqrt{\ln q} s_0}{n^{1/2}} \right)\label{t242} 
 \end{eqnarray} 
Using (\ref{t241}) and (\ref{t242}) results in
 \begin{equation}
  \| \left( \frac{X'Z}{n}  \frac{\hat{W}_d}{q}  - \Sigma_{xz} \frac{W_d}{q} \right)  \|_{l_{\infty}} = 
 O_p \left( \frac{\sqrt{\ln q}}{n^{1/2}} \right) +  O_p  \left( \frac{\sqrt{\ln q} s_0}{n^{1/2}} \right) =
  O_p  \left( \frac{\sqrt{\ln q} s_0}{n^{1/2}} \right).\label{t243}
  \end{equation}

By using the generalized version of Hölder's inequality we can bound $\| \Sigma_{zu} \|_{\infty}$ in (\ref{t239}).
  \[  \| \Sigma_{zu} \|_{\infty}  = \max_{1 \le l \le q} \max_{1 \le m \le q}  E Z_{il} Z_{im} u_i^2\leq [ \max_{1 \le l \le q}  E |Z_{il}|^3]^{1/3} [ \max_{1 \le m \le q}  E | Z_{im}|^3]^{1/3} [ E u_i^6]^{1/3}\leq C\]
where we used that $r_z > 12, r_u > 8$. Thus, 
 \begin{equation}
  \| \Sigma_{zu} \|_{\infty} = O(1).\label{t245} 
 \end{equation}
 Insert (\ref{t243}) and (\ref{t245}) into (\ref{t239}) to obtain 
  \begin{equation}
  \|  \left( \frac{X'Z}{n}  \frac{\hat{W}_d}{q}  - \Sigma_{xz} \frac{W_d}{q} \right) \Sigma_{zu}  \left( \frac{X'Z}{n}  \frac{\hat{W}_d}{q}  - \Sigma_{xz} \frac{W_d}{q} \right)'   
 \|_{\infty}  = O_p  \left( \frac{\ln qs_0^2} {n} \right) ,\label{t246} 
 \end{equation}
 which establishes the rate for (\ref{t238a}).
 Now consider (\ref{t238b})  
 \begin{eqnarray}
 \| (\frac{X'Z}{n} \frac{\hat{W}_d}{q} - \Sigma_{xz} \frac{W_d}{q}) \Sigma_{zu} (\frac{W_d}{q} \Sigma_{xz}') \|_{\infty}
 & \le &  \| (\frac{X'Z}{n} \frac{\hat{W}_d}{q} - \Sigma_{xz} \frac{W_d}{q}) \|_{l_{\infty}} \| \Sigma_{zu} (\frac{W_d}{q} \Sigma_{xz}') \|_{\infty} \nonumber \\
 & \le &  \| (\frac{X'Z}{n} \frac{\hat{W}_d}{q} - \Sigma_{xz} \frac{W_d}{q}) \|_{l_{\infty}} \| \Sigma_{zu}\|_{\infty} 
 \| (\frac{W_d}{q} \Sigma_{xz}') \|_{l_1},\label{t247}
    \end{eqnarray}
 where we use (\ref{ineq3}) for the first inequality and the dual norm inequality on p.44 in \cite{vdG16} for the second one.
 Next, by Lemma \ref{lineq}(vi), 
 \begin{eqnarray}
 \| \frac{W_d}{q} \Sigma_{xz}' \|_{l_1}  \le  q \| E X_i Z_i'  \|_{\infty} \| \frac{W_d}{q} \|_{l_{\infty}} 
 = \|  E X_i Z_i' \|_{\infty} \| W_d \|_{l_{\infty}} 
 =  O(1),\label{t248} 
 \end{eqnarray}
 where we use (\ref{a9}), and Assumption 1, with $W_d$ being diagonal, and $\min_{1 \le l \le q} \sigma_l^2$ being bounded away from zero.
Now use (\ref{t243}), (\ref{t245}) and (\ref{t248}) in (\ref{t247}) to get 
 \begin{eqnarray}
&& \| (\frac{X'Z}{n} \frac{\hat{W}_d}{q} - \Sigma_{xz} \frac{W_d}{q}) \Sigma_{zu} (\frac{W_d}{q} \Sigma_{xz}') \|_{\infty} 
 =  O_p \left( \frac{\sqrt{\ln q} s_0 }{n^{1/2}} \right) .\label{t249}
\end{eqnarray}
Next, (\ref{t238}) obeys the same bound as (\ref{t238b}) by the two matrices being each others transposes. Thus,
 \begin{equation}
 \| (\Sigma_{xz} \frac{W_d}{q}) \Sigma_{zu} ( \frac{X'Z}{n} \frac{\hat{W}_d}{q} - \Sigma_{xz} \frac{W_d}{q})' \|_{\infty} 
 = O_p \left( \frac{\sqrt{\ln q} s_0 }{n^{1/2}}\right).\label{t250}
 \end{equation}
 Now  use (\ref{t246}) in (\ref{t238a}),  (\ref{t249}) in (\ref{t238b}) and (\ref{t250}) in (\ref{t238}) to get
 \begin{eqnarray}
 \| \bar{V}_d - V_d \|_{\infty} = O_p \left( [\frac{\sqrt{\ln q} s_0 }{n^{1/2}}]^2 \right) + O_p  \left( \frac{\sqrt{\ln q} s_0 }{n^{1/2}}  \right)
 + O_p \left(\frac{\sqrt{\ln q}  s_0 }{n^{1/2}} \right)
= O_p \left(\frac{\sqrt{\ln q} s_0}{n^{1/2}} \right).\label{t251}
 \end{eqnarray}
 Using (\ref{t251}) in (\ref{t237}) together with (\ref{gel1}) yields
 \begin{eqnarray}
 |e_j' \hat{\Gamma} \bar{V}_d \hat{\Gamma}' e_j - e_j' \hat{\Gamma} V_d \hat{\Gamma}' e_j | 
 =  
 O_p (m_{\Gamma}^2) O_p \left(\frac{\sqrt{\ln q} s_0 }{n^{1/2}}  \right) 
 =
 o_p (1),\label{t252}
  \end{eqnarray}
 since by Assumption \ref{5cl}(iv) $\frac{m_{\Gamma}^2 s_0 \sqrt{\ln q} }{n^{1/2}} = o(1)$. Next, by (\ref{t236}) and (\ref{t252})
 we have that the following holds for (\ref{t212}):
 \begin{equation}
 |e_j' \hat{\Gamma} \tilde{V}_d \hat{\Gamma}' e_j - e_j' \hat{\Gamma} V_d \hat{\Gamma}' e_j | = o_p (1).\label{t253}
 \end{equation}
 
Finally, we bound (\ref{t213}). By Lemma 3.1 in the supplement of \cite{van2014}, we have
 \begin{eqnarray}
 |e_j' \hat{\Gamma} V_d \hat{\Gamma}' e_j - e_j' \Gamma V_d \Gamma' e_j | & \le & 
 Eigmax(V_d)^2 \|\hat{\Gamma}' e_j - \Gamma' e_j \|_2^2 \nonumber \\
 & + & 2 \| V_d \Gamma' e_j \|_2 \|\hat{\Gamma}' e_j - \Gamma' e_j \|_2.\label{t254} 
 \end{eqnarray}
 Now, 
 \begin{eqnarray}
 \| V_d' \Gamma' e_j \|_2 &=& \sqrt{e_j' \Gamma V_d^2 \Gamma' e_j}\nonumber \\
 & \le & \sqrt{ Eigmax (V_d)^2 e_j' \Gamma \Gamma' e_j } \nonumber \\
 & \le & \sqrt{Eigmax(V_d)^2  (Eigmax(\Gamma))^2}  = \frac{Eigmax(V_d)}{Eigmin (\Sigma)} \nonumber \\
 & = & O (1),\label{t255}
 \end{eqnarray}
where we used that $\Gamma$ and $V_d$ are symmetric, $e_j'e_j=1$ and $\Gamma = \Sigma^{-1}$. Finally Assumption \ref{5cl}(ii) was used. Next,
 \begin{equation}
 \| \hat{\Gamma}' e_j - \Gamma' e_j \|_2 \le \max_{1 \le j \le p} \| \hat{\Gamma}_j - \Gamma_j \|_2 = O_p \left(  s_{\Gamma} (m_{\Gamma} \mu )^{1-f}
\right),\label{t256}
\end{equation}
by $e_j'e_j=1$ and Lemma \ref{cl-l4}. Using (\ref{t255}) and (\ref{t256}) in (\ref{t254}) yields, along with the rate of $\mu$ in (\ref{rocmu}), 
 \begin{eqnarray}
 |e_j' \hat{\Gamma} V_d \hat{\Gamma}' e_j - e_j' \Gamma V_d \Gamma' e_j | & \le & 
 O(1) O_p \left(  s_{\Gamma}^2 (m_{\Gamma} \mu)^{2(1-f)}
 \right) + O(1) O_p \left( s_{\Gamma} (m_{\Gamma} \mu)^{1-f}
 \right) \nonumber \\
 & = & 
 O_p \left( s_{\Gamma} m_{\Gamma}^{2-2f} s_0^{1-f} \frac{\ln q^{(1-f)/2}}{n^{(1-f)/2}}
 \right)= o_p (1), \end{eqnarray}
 by Assumption \ref{5cl}(i). Thus, (\ref{t213}) is asymptotically negligible. As (\ref{t211}) and (\ref{t212}) have also been shown to be asymptotically negligible in (\ref{t229}) and (\ref{t253}), (\ref{t210a}) is established.
 
This concludes Step 2 since by Steps 2a-b we have shown that $t_{W_{d1}} - t_{W_{d1}^*} = o_p(1)$. Inspection of the above arguments also shows that (\ref{t210a}) is valid uniformly over ${\cal B}_{l_0}(s_0) = \{ \|\beta_0 \|_{l_0} \le s_0 \}$. 
 

 {\bf Step 3}. Here we  show that $t_{W_{d2}} = o_p (1)$. Note that the denominator of $t_{W_{d2}}$ is the same as that of $t_{W_{d1}}$ which is bounded away from 0 with probability converging to one. It thus suffices to show that the numerator of $t_{W_{d2}}$ vanishes in probability. To this end, note that this numerator is upper bounded by $|e_j' \Delta|$ where 
 \[
 \Delta  
=
\left[ \hat{\Gamma} \hat{\Sigma}  - I_p \right] \sqrt{n} (\hat{\beta} - \beta_0) 
=
 \left[ \hat{\Gamma} \left( \frac{X'Z}{n} \frac{\hat{W}_d}{q} \frac{Z'X}{n} 
 \right) - I_p \right] \sqrt{n} (\hat{\beta} - \beta_0) . \]
 
 
Next, note that 
\begin{eqnarray*}
 |e_j' \Delta | & = &
   \left| \left( e_j' (\hat{\Gamma} \hat{\Sigma} - I_p) \right) \sqrt{n} (\hat{\beta} - \beta_0) \right| \\
 & \le &   \| e_j' (\hat{\Gamma} \hat{\Sigma} - I_p ) \|_{\infty}   \| \sqrt{n} (\hat{\beta} - \beta_0) \|_1 \\
 & \le & \| e_j \|_1 \| \hat{\Gamma} \hat{\Sigma} - I_p \|_{\infty} \sqrt{n} \|\hat{\beta} - \beta_0 \|_1 \\
  &= &  O_p (\mu)  \sqrt{n} \|(\hat{\beta} - \beta_0) \|_1,
  \end{eqnarray*}
 where  we used Hölder's inequality for the inequality. 
 Furthermore, by definition of the CLIME program $\| \hat{\Gamma} \hat{\Sigma} -I_p \|_{\infty} \le \mu$.
 Then, by (\ref{rocmu})
\[  \mu = O \left( m_{\Gamma}\frac{s_0 \sqrt{\ln q} }{\sqrt{n}}\right).\]
which together with with Theorem \ref{1}(ii) gives that
\begin{eqnarray}
 |e_j' \Delta |
=
  O \left( m_{\Gamma} s_0  \frac{\sqrt{\ln q} }{\sqrt{n}}\right) O_p \left( \frac{\sqrt{\ln q}  s_0}{\sqrt{n}} \right)  n^{1/2} 
  =   O_p \left( \frac{ m_{\Gamma} s_0^2 \ln q}{\sqrt{n}} \right) = o_p (1)\label{eq:A1set},
 \end{eqnarray}
  by Assumption \ref{5cl}(i). This concludes Step 3 upon noting that the above estimates are valid uniformly over ${\cal B}_{l_0}(s_0) = \{ \|\beta_0 \|_{l_0} \le s_0 \}$ by Theorem \ref{thm1}(ii). 
\end{proof}
  \newpage

\newpage

\section{Supplementary Appendix}

\setcounter{equation}{0}\setcounter{lemma}{0}\setcounter{assum}{0}\renewcommand{\theequation}{S.%
\arabic{equation}}\renewcommand{\thelemma}{S.\arabic{lemma}}%
\renewcommand{\theassum}{S.\arabic{assum}}%
\renewcommand{\baselinestretch}{1}\baselineskip=15pt

In this part of the paper we present auxiliary technical lemmas and their proofs. We start with some matrix norm inequalities. Let $A$ be a generic $q \times p$ matrix and $x$ a $p \times 1$ vector. Define $a_l'$, $l=1,...,q$ as the $l$'th row of the matrix $A$. Finally, let $B$ be a $p \times q$ matrix, and $F$ a square $q \times q$ matrix while $x$ is a vector of conformable dimension.

\begin{lemma}\label{lineq}

(i). \[  \|A x \|_1 \le \left[ \sum_{l=1}^q \|a_l \|_{\infty} \right] \|x\|_1.\] 

(ii). \[ \|B x \|_{\infty} \le q \| B \|_{\infty} \| x\|_{\infty} \] 

(iii).\[ \| B F B' \|_{\infty} \le \| B \|_{l_{\infty}} \| B' \|_{l_1} \|F \|_{\infty}  = [\| B \|_{l_{\infty}}^2] \|F \|_{\infty}.\]

(iv). \[ \| B F \|_{l_{\infty}} \le q \| B \|_{\infty} \|F \|_{l_1}.\]

(v). \[ \| F B' \|_{l_{\infty}} \le  p \| B \|_{\infty} \|F \|_{l_{\infty}}.\]

(vi). \[ \| F B' \|_{l_1} \le q \| F \|_{l_{\infty}} \|B \|_{\infty} .\]

(vii). \[ \| B A \|_{\infty} \le q \| B \|_{\infty} \| A \|_{\infty}.\]

(viii). \[| x' B F A  x|  \le  q \| x\|_1^2 \| B \|_{\infty} \| F \|_{l_{\infty}} \| A \|_{\infty},\label{mnin1}. \]

\end{lemma}

\begin{proof}

(i) Using Hölder's inequality
\begin{equation}
 \|A x \|_1 = \| \left[ \begin{array}{c}
					a_1'x \\
					\vdots \\
					a_l'x \\
					\vdots \\
					a_q'x \end{array} \right] \|_1 = \sum_{l=1}^q |a_l'x| \le \left[ \sum_{l=1}^q \|a_l \|_{\infty}\right] \|x\|_1.\label{ineq1}
					\end{equation}

(ii) Letting $b_j'$ be the $j$th row of $B$ and $B_{jl}$ the $(j,l)$th entry of $B$, it follows by Hölder's inequality that
\begin{equation}
 \|B x \|_{\infty} = \| \left[ \begin{array}{c}
					b_1'x \\
					\vdots \\
					b_j'x \\
					\vdots \\
					b_p'x \end{array} \right] \|_{\infty} = \max_{1 \le j \le p}  |b_j'x| \le \max_{1 \le j \le p}  \left[ \sum_{l=1}^q |B_{jl}| \right] \|x\|_{\infty} \le 
					 q \left[\max_{1 \le j \le p}
\max_{1 \le l \le q} |B_{jl}| \right] \| x\|_{\infty}.\label{ineq2}
					\end{equation}

(iii) Let $a_j$ be the $j$th column of $A$, $j=1,...,p$. Then, 

\[ B A = \left[ \begin{array}{c}
					b_1' \\
					\vdots \\
					b_j' \\
					\vdots \\
					b_p' \end{array} \right] [a_1, \cdots, a_j, \cdots, a_p] =
					\left[ \begin{array}{ccc}
					b_1'a_1 & \cdots  &  b_1' a_p \\
					\vdots & \vdots & \vdots \\
					b_j' a_1 & \cdots & b_j' a_p \\
					\vdots & \vdots & \vdots \\
					b_p' a_1 & \cdots & b_p' a_p \end{array} \right].\]
					Thus, by Hölder's inequality
					\begin{eqnarray}
					\| B A \|_{\infty} & = & \max_{1 \le j \le p} \max_{1 \le k \le p} |b_j'a_k| \nonumber \\
					& \le & \max_{1 \le j \le p} \max_{1 \le k \le p} \|b_j \|_1 \|a_k \|_{\infty} \nonumber \\
					& = & \left[ \max_{ 1\le j \le p} \| b_j \|_1 \right] \left[ \max_{1 \le k \le p} \|a_k\|_{\infty} \right]  \nonumber \\
					& = & \|B \|_{l_{\infty}} \|A \|_{\infty},\label{ineq3}
					\end{eqnarray}
					
					Next, using $A = F  B'$, where $F$ is a generic $q \times q$ matrix, it follows from the dual norm inequality in section 4.3 
					of \cite{vdG16} that
					\begin{equation}
					\| F B' \|_{\infty} \le  \| F \|_{\infty} \| B' \|_{l_1},\label{ineq4}
					\end{equation}
					Combine (\ref{ineq3}) and (\ref{ineq4}) to get 
					\begin{equation}
					\| B F B' \|_{\infty} \le \| B \|_{l_{\infty}} \| B' \|_{l_1} \|F \|_{\infty},\label{ineq5}
					\end{equation}

(iv) Denoting the $j$th row of $B$ by $b_j'$, for $j=1,..., p$ and the columns of $F$ by $f_l$ for  $l=1,...,q$:
\[ B F = \left[ \begin{array}{c}
				b_1' \\
				\vdots \\
				b_j' \\
				\vdots \\
				b_p' \end{array} \right] [ f_1, \cdots, f_l, \cdots, f_q] = 
				\left[\begin{array}{ccccc}
				b_1' f_1 & \cdots & b_1' f_l & \cdots & b_1' f_q \\
				\vdots & \vdots & \vdots & \vdots & \vdots \\
				b_j' f_1 & \cdots & b_j' f_l & \cdots & b_j' f_q \\
				\vdots & \vdots & \vdots & \vdots & \vdots \\
				b_p' f_1 & \cdots & b_p' f_l & \cdots & b_p' f_q \end{array} \right].\]
				
				Then 
				\begin{eqnarray}
				\| B F \|_{l_{\infty}} & = & \max_{1 \le j \le p} \sum_{l=1}^q |b_j' f_l|  \le \max_{1 \le j \le p} \sum_{l=1}^q \|b_j \|_{\infty} \|f_l \|_1 \nonumber \\
				& \le & \|B \|_{\infty} \left[ q \max_{1 \le l \le q} \| f_l \|_1 \right]  = q \| B \|_{\infty} \|F \|_{l_1}, \label{ineq6}
				\end{eqnarray}
				where we used the definition $\|.\|_{l_1}$ in the last equality.

(v) Denoting the $m$th row of $F$ by $f_m'$, $m=1,..., q$ and the $j$th column of $B'$ by $b_j$, $j=1,..., p$ we first observe
\[  F B' = \left[ \begin{array}{c}
				f_1' \\
				\vdots \\
				f_m' \\
				\vdots \\
				f_q' \end{array} \right] [ b_1, \cdots, b_j, \cdots, b_p] = 
				\left[\begin{array}{ccccc}
				f_1' b_1 & \cdots & f_1' b_j & \cdots & f_1' b_p \\
				\vdots & \vdots & \vdots & \vdots & \vdots \\
				f_m' b_1  & \cdots & f_m' b_j & \cdots & f_m' b_p \\
				\vdots & \vdots & \vdots & \vdots & \vdots \\
				f_q' b_1 & \cdots & f_q' b_j & \cdots &  f_q' b_p \end{array} \right].\]				 
		By Hölder's inequality in the first inequality, and definition of norms afterwards
		\begin{eqnarray}
		\| F B' \|_{l_{\infty}} & = & \max_{1 \le m \le q} \sum_{j=1}^p | f_m' b_j| \le \max_{1 \le m \le q} \sum_{j=1}^p \|f_m\|_1 \|b_j \|_{\infty}	\nonumber \\
		& \le  & p \max_{1 \le j \le p} \| b_j \|_{\infty} \left[ \max_{1 \le m \le q} \| f_m \|_1 \right]  =  p \| B \|_{\infty} \|F \|_{l_{\infty}}.\label{ineq7}
		\end{eqnarray}		

(vi) By Hölder's inequality
		\begin{eqnarray*}
		 \| F B' \|_{l_1} &= &  \max_{1 \le j \le p} \sum_{m=1}^q |f_m' b_j|  \\
		 & \le & \max_{1 \le j \le p} \sum_{m=1}^q  \| f_m \|_1 \| b_j \|_{\infty} \\
		 & \le & \left[ q \max_{1 \le m \le q} \|f_m \|_1 \right]  \left[ \max_{1 \le j \le p} \| b_j \|_{\infty} \right] \\
		 & = & q \| F \|_{l_{\infty}} \| B \|_{\infty},
		 \end{eqnarray*}

(vii) By (\ref{ineq3}), and letting $b_j'$ be the $j$th row of $B$ we obtain:
\begin{eqnarray}
\| B A \|_{\infty} & \le & \| B \|_{l_{\infty}} \| A \|_{\infty} = [ \max_{1 \le j \le p} \| b_j' \|_1] \| A \|_{\infty} \nonumber \\
& \le & [ q \max_{1 \le j \le p} \max_{1 \le l \le q} |B_{jl}| ] \| A \|_{\infty} \nonumber \\
& = & q \| B \|_{\infty} \|A \|_{\infty}\label{ineqsa1}
\end{eqnarray}

(viii) Observe that 
\begin{eqnarray}
| x' B F A  x| & \le & \| x\|_1 \| BFAx \|_{\infty} \nonumber \\
& \le & \| x\|_1^2 \| B F A \|_{\infty} \nonumber \\
& \le &  \| x \|_1^2  \| B \|_{\infty} \| F A \|_{l_1} \nonumber \\
& \le & q \| x\|_1^2 \| B \|_{\infty} \| F \|_{l_{\infty}} \| A \|_{\infty},\label{mnin1}
\end{eqnarray}
where we use Hölder's inequality for the first and second inequalities. The third inequality uses the dual norm inequality of Section 4.3 in \cite{vdG16} while (vi) was used for the last one.

\end{proof}

The following lemma shows that the adaptive restricted eigenvalue, as defined prior to Lemma \ref{l1}, is bounded away from zero with high probability for the first step GMM estimator.

\begin{lemma}\label{evalue1}
Let Assumptions \ref{1} and \ref{2} be satisfied. 
Then, for $n$ sufficiently large, the set 
\[ {\cal A}_3 = \{ \hat{\phi}_{ \hat{\Sigma}_{xz}}^2 (s_0) \ge \phi_{\Sigma_{xz}}^2 (s_0)/2 \},\]
has probability at least $1 - \exp (- C\kappa_n) - \frac{K E M_2^{2}}{n \kappa_n}$, for universal positive constants $C, K$. Furthermore, the probability of $\mathcal{A}_3$ tends to one as $n\to\infty$.
\end{lemma}

\begin{proof}[Proof of Lemma \ref{evalue1}]
We begin by noting that for any $p\times 1$ vector $\delta$
\begin{eqnarray}
\frac{1}{q} \left| \delta' \frac{X'Z}{n}  \frac{Z'X}{n} \delta \right| & = & 
\frac{1}{q} \left| \delta' (\frac{X'Z}{n} - \Sigma_{xz})  ( \frac{Z'X}{n} - \Sigma_{xz}') \delta
+ 2 \delta' (\frac{X'Z}{n} - \Sigma_{xz}) \Sigma_{xz}' \delta+ \delta' \Sigma_{xz} \Sigma_{xz}' \delta \right| \nonumber \\
& \ge & \frac{1}{q} \left| \delta'  \Sigma_{xz} \Sigma_{xz}' \delta \right|  - \frac{2}{q} \left|\delta' (\frac{X'Z}{n} - \Sigma_{xz}) \Sigma_{xz}' \delta \right|
\nonumber \\
& - & \frac{1}{q}  \left| \delta' (\frac{X'Z}{n} - \Sigma_{xz})  ( \frac{Z'X}{n} - \Sigma_{xz}') \delta \right|\label{ev1}
\end{eqnarray}
The second term on the right side of (\ref{ev1}) can be bounded as follows:
\begin{eqnarray}
\frac{2}{q} \left|\delta' (\frac{X'Z}{n} - \Sigma_{xz}) \Sigma_{xz}' \delta \right| & \le & \frac{2}{q} \|\delta\|_1^2  \| (\frac{X'Z}{n} - \Sigma_{xz})\Sigma_{xz}'\|_{\infty}  \nonumber \\
& \le & 2 \| \delta\|_1^2   \| (\frac{X'Z}{n} - \Sigma_{xz}) \|_{\infty}  \| \Sigma_{xz}'\|_{\infty},\label{s10a}
\end{eqnarray}
where for the second inequality we used (\ref{ineqsa1}).
For the third term on the right side of (\ref{ev1}) we get in the same way as above
\begin{equation}
\frac{1}{q}  \left| \delta' (\frac{X'Z}{n} - \Sigma_{xz})  ( \frac{Z'X}{n} - \Sigma_{xz}') \delta \right|
\le \| \delta \|_1^2 \left[ \| \frac{X'Z}{n} - \Sigma_{xz} \|_{\infty} \right]^2.\label{s11}
\end{equation}
Inserting (\ref{s10a})(\ref{s11}) in (\ref{ev1}) yields
\begin{equation}
\frac{1}{q} \| \frac{Z'X}{n} \delta \|_2^2 \ge \frac{1}{q} \|\Sigma_{xz}' \delta \|_2^2 - 2 \| \delta\|_1^2 
\| \frac{X'Z}{n} - \Sigma_{xz} \|_{\infty}  \| \Sigma_{xz}'\|_{\infty} - \|\delta\|_1^2 
\left[ \| \frac{X'Z}{n} - \Sigma_{xz} \|_{\infty} \right]^2.\label{ev2}
\end{equation} 
Note that we have the restriction $\| \delta_{S_0^c} \|_1 \le 3 \sqrt{s_0} \|\delta_{S_0}\|_2$. Add $\|\delta_{S_0}\|_1$ to both sides of this to get $\| \delta \|_1 \le 4 \sqrt{s} \| \delta_{S_0}\|_2$ where we also used the Cauchy-Schwarz inequality. Thus,
\begin{equation}
\frac{\| \delta \|_1^2}{\| \delta_{S_0} \|_2^2} \le 16s_0.\label{ev3}
\end{equation}
Divide (\ref{ev2}) by $\| \delta_{S_0} \|_2^2>0$ and use (\ref{ev3})
\begin{equation}
\frac{1}{q} \frac{ \| \frac{Z'X}{n} \delta \|_2^2}{\| \delta_{S_0}\|_2^2} \ge \frac{1}{q} \frac{ \| \Sigma_{xz}' \delta \|_2^2}{\|\delta_{S_0}\|_2^2}
-32 s_0  \left[ \| \frac{X'Z}{n} - \Sigma_{xz} \|_{\infty} \| \Sigma_{xz}' \|_{\infty} 
\right] - 16s_0 \left(  \| \frac{X'Z}{n} - \Sigma_{xz} \|_{\infty}\right)^2.\label{ev4}
\end{equation}
Using that $\frac{1}{q} \frac{ \| \Sigma_{xz}' \delta \|_2^2}{\|\delta_{S_0}\|_2^2}
\geq \phi_{\Sigma_{xz}}^2(s_0)$ for all $\delta$ satisfying $\enVert[0]{\delta_{S_0^c}}_1\leq 3\sqrt{s_0}\enVert[0]{\delta_{S_0}}_2$ and minimizing the left hand side over these $\delta$ yields 
\begin{align*}
\hat{\phi}_{\hat{\Sigma}_{xz}}^2 (s_0) \geq \phi_{\Sigma_{xz}}^2(s_0) 
-32 s_0  \left[ \| \frac{X'Z}{n} - \Sigma_{xz} \|_{\infty} \| \Sigma_{xz}' \|_{\infty} 
\right] - 16s_0 \left(  \| \frac{X'Z}{n} - \Sigma_{xz} \|_{\infty}\right)^2.
\end{align*}
Note that if with probability approaching one (wpa1)
\begin{equation}
 32 s_0  \left[ \| \frac{X'Z}{n} - \Sigma_{xz} \|_{\infty} \| \Sigma_{xz}' \|_{\infty} 
\right] +16s_0 \left(  \| \frac{X'Z}{n} - \Sigma_{xz} \|_{\infty}\right)^2 \leq  \phi_{\Sigma_{xz}}^2(s_0)/2 .\label{c0n}
\end{equation}
then 
$\hat{\phi}_{\hat{\Sigma}_{xz}}^2 (s_0) \geq \phi_{\Sigma_{xz}}^2(s_0)/2 $ wpa1.
Thus,
\[ P \left( \hat{\phi}_{\hat{\Sigma}_{xz}}^2 (s_0) < \phi_{\Sigma_{xz}}^2 (s_0)/2
\right)
\le P \left( 32 s_0  \left[ \| \frac{X'Z}{n} - \Sigma_{xz} \|_{\infty} \| \Sigma_{xz}' \|_{\infty} 
\right] +16s_0 \left(  \| \frac{X'Z}{n} - \Sigma_{xz} \|_{\infty}\right)^2 > \phi_{\Sigma_{xz}}^2 (s_0)/2\right).\]
Letting $t_3$ as in (\ref{t2s}) define  
\begin{equation}
 \epsilon_{1n} := 32 s_0 t_3 \| \Sigma_{xz}' \|_{\infty} + 16 s_0 (t_3)^2\label{ep1}
 \end{equation}
and note that
\begin{eqnarray}
P \left( 32 s_0  \left[ \| \frac{X'Z}{n} - \Sigma_{xz} \|_{\infty} \| \Sigma_{xz}' \|_{\infty} 
\right] +16s_0 \left(  \| \frac{X'Z}{n} - \Sigma_{xz} \|_{\infty}\right)^2 > \epsilon_n\right) & \le & 
P \left( \| \frac{X'Z}{n} - \Sigma_{xz} \|_{\infty} > t_3\right) \nonumber \\
&\le& \exp(-C\kappa_n) + \frac{K E M_2^{2}}{n\kappa_n},\label{s16a}
\end{eqnarray}
by (\ref{a8}).
Since $\epsilon_{1n}\to 0$ by Assumption \ref{2}, for $n$ sufficiently large, by (\ref{c0n})(\ref{s16a})
\begin{align*}
P \left( \hat{\phi}_{\hat{\Sigma}_{xz}}^2 (s_0) < \phi_{\Sigma_{xz}}^2 (s_0)/2
\right)
\leq
\exp(-C\kappa_n) + \frac{K E M_2^{2}}{n\kappa_n}\to 0
\end{align*}
by Assumption \ref{2}.
\end{proof}

The following lemma  verifies the adaptive restricted eigenvalue condition for the two-step GMM estimator. 

\begin{lemma}\label{evalue2}
Let Assumptions \ref{1} and \ref{2} be satisfied. 
Then, for $n$ sufficiently large, the set 
\[ {\cal A}_4 = \{ \hat{\phi}_{\hat{\Sigma}_{xz\hat{w}}}^2 (s_0) \ge \phi_{\Sigma_{xzw}}^2 (s_0)/2\},\]
has probability at least $1 - 10 \exp (- C \kappa_n) - 
\frac{K[ 2 EM_1^{2} + 5 E M_2^{2} + E M_3^{2} + E M_4^{2} + E M_5^{2}]}{n\kappa_n}$
for universal positive constants $C, K$. Furthermore, the probability of $\mathcal{A}_4$ tends to one as $n\to\infty$.
\end{lemma}

\begin{proof}
(i). By adding and subtracting $\Sigma_{xz}, W_d$,
\begin{eqnarray}
\frac{1}{q} | \delta' \frac{X'Z}{n} \hat{W}_d \frac{Z'X}{n} \delta | & = & 
\frac{1}{q} | \delta' \frac{X'Z - \Sigma_{xz}+ \Sigma_{xz}}{n} [\hat{W}_d - W_d + W_d]  \frac{Z'X- \Sigma_{xz}' + \Sigma_{xz}'}{n} \delta |
\nonumber \\
& \ge & \frac{1}{q} | \delta' \Sigma_{xz} W_d \Sigma_{xz}' \delta | - \frac{1}{q} | \delta' 
(\frac{X'Z}{n} - \Sigma_{xz}) (\hat{W}_d - W_d) ( \frac{Z'X}{n} - \Sigma_{xz}') \delta| \nonumber \\
& - & \frac{1}{q} | \delta' 
(\frac{X'Z}{n} - \Sigma_{xz}) (W_d) ( \frac{Z'X}{n} - \Sigma_{xz}') \delta| \nonumber \\
& - & \frac{1}{q} | \delta' 
(\Sigma_{xz}) (\hat{W}_d - W_d) ( \Sigma_{xz}') \delta| \nonumber \\
& - & \frac{2}{q} | \delta' 
(\frac{X'Z}{n} - \Sigma_{xz}) (\hat{W}_d - W_d) ( \Sigma_{xz}') \delta| \nonumber \\
& - & \frac{2}{q} | \delta' 
(\frac{X'Z}{n} - \Sigma_{xz}) (W_d) ( \Sigma_{xz}') \delta|,\label{ev21}  
\end{eqnarray}
Now we consider the second term on the right side of the inequality in  (\ref{ev21}). 
\[ \frac{1}{q} | \delta' (\frac{X'Z}{n} - \Sigma_{xz}) (\hat{W}_d - W_d) (\frac{Z'X}{n} - \Sigma_{xz}') \delta| 
\le 
\| \delta \|_1^2 \left[ \| \frac{X'Z}{n} - \Sigma_{xz}\|_{\infty}^2 \right] \| \hat{W}_d - W_d \|_{l_{\infty}},\]
by Lemma \ref{lineq} (viii). By the same reasoning,
\[ \frac{1}{q} | \delta' ( \frac{X'Z}{n} - \Sigma_{xz}) W_d (\frac{Z'X}{n} - \Sigma_{xz}') \delta | 
\le \|\delta \|_1^2 \left[ \| \frac{X'Z}{n} - \Sigma_{xz}\|_{\infty}^2\right] \| W_d \|_{l_{\infty}}.\]
\[ \frac{1}{q} | \delta' 
(\Sigma_{xz}) (\hat{W}_d - W_d) ( \Sigma_{xz}') \delta| \le \| \delta \|_1^2 [ \| \Sigma_{xz} \|_{\infty}^2] \| \hat{W}_d - W_d \|_{l_{\infty}}.\]
\[ \frac{2}{q} | \delta' 
(\frac{X'Z}{n} - \Sigma_{xz}) (\hat{W}_d - W_d) ( \Sigma_{xz}') \delta|
\le  2 \| \delta \|_1^2  \| \frac{X'Z}{n} - \Sigma_{xz} \|_{\infty} \| \Sigma_{xz} \|_{\infty} \| \hat{W}_d - W_d \|_{l_{\infty}}.\]
\[ \frac{2}{q} | \delta' 
(\frac{X'Z}{n} - \Sigma_{xz}) (W_d) ( \Sigma_{xz}') \delta| \le 
2 \| \delta \|_1^2  \| \frac{X'Z}{n} - \Sigma_{xz} \|_{\infty} \| \Sigma_{xz} \|_{\infty} \| W_d \|_{l_{\infty}}.\]
By (\ref{ev3}) and $\hat{W}_d$ and $W_d$ being positive definite matrices, (\ref{ev21}) thus yields
\begin{eqnarray}
\frac{\| \hat{W_d}^{1/2} \frac{Z'X}{n} \delta\|_2^2}{q \| \delta_{S_0} \|_2^2} 
& \ge & \frac{\| W_d^{1/2} \Sigma_{xz}' \delta\|_2^2}{q \| \delta_{S_0} \|_2^2} \nonumber \\
& -  & 16s_0  \left[ \| \frac{X'Z}{n} - \Sigma_{xz} \|_{\infty} \right]^2 \left[ \| \hat{W}_d - W_d \|_{l_{\infty}} + \| W_d \|_{l_{\infty}} \right] \nonumber \\
& - & 16s_0  [ \| \Sigma_{xz} \|_{\infty}^2] \| \hat{W}_d - W_d \|_{l_{\infty}} \nonumber \\
& - & 32 s_0 \left( \| \frac{X'Z}{n} - \Sigma_{xz} \|_{\infty} \| \Sigma_{xz} \|_{\infty} \right) 
\left( \| \hat{W}_d - W_d \|_{l_{\infty}} + \| W_d \|_{l_{\infty}} \right).\label{ev22} 
\end{eqnarray}
Since $\frac{\| W_d^{1/2} \Sigma_{xz}' \delta\|_2^2}{q \| \delta_{S_0} \|_2^2}\geq \phi^2_{\Sigma_{xzw}}(s_0)/2$ for all $\delta\in\mathbb{R}^p$ such that $||\delta_{S_0^c}||_1\leq 3\sqrt{s_0}||\delta||_2$ minimizing the left hand side of the above display over such $\delta$ yields
\begin{align*}
\hat{\phi}^2_{\hat{\Sigma}_{xz\hat{w}}}(s_0)
&\geq
\phi^2_{\Sigma_{xzw}}(s_0)
-   a_n 
\end{align*}
for 
\begin{align*}
a_n:&=16s_0  \left[ \| \frac{X'Z}{n} - \Sigma_{xz} \|_{\infty} \right]^2 \left[ \| \hat{W}_d - W_d \|_{l_{\infty}} + \| W_d \|_{l_{\infty}} \right]  
 +  16s_0  [ \| \Sigma_{xz} \|_{\infty}^2] \| \hat{W}_d - W_d \|_{l_{\infty}}\\ 
& +  32 s_0 \left( \| \frac{X'Z}{n} - \Sigma_{xz} \|_{\infty} \| \Sigma_{xz} \|_{\infty} \right) 
\left( \| \hat{W}_d - W_d \|_{l_{\infty}} + \| W_d \|_{l_{\infty}} \right)
\end{align*}
Note that if $a_n\leq \phi^2_{\Sigma_{xzw}}(s_0)/2$ wpa1, then $\hat{\phi}^2_{\hat{\Sigma}_{xz\hat{w}}}(s_0)\geq \phi^2_{\Sigma_{xzw}}(s_0)/2$, wpa1. Thus,
\begin{align*}
P\del[1]{\hat{\phi}^2_{\hat{\Sigma}_{xz\hat{w}}}(s_0)<\phi^2_{\Sigma_{xzw}}(s_0)/2}
\leq
P\del[1]{a_n> \phi^2_{\Sigma_{xzw}}(s_0)/2}
\end{align*}
As argued just after (\ref{la96}) one has that $||\hat{W}_d - W_d||_{l_{\infty}}\leq c_{1n}$ wpa1, if $\max_{1\leq j\leq p}||\hat{\sigma}^2_j-\sigma^2_j||\leq c_n$ wpa1 where $c_{1n}$ is defined in (\ref{c1n}).
Define 
\begin{equation}
\epsilon_{2n}: = 16 s_0 (t_3)^2[ c_{1n} + \| W_d \|_{l_{\infty}}] + 16 s_0 ( \| \Sigma_{xz} \|_{\infty}^2) c_{1n} 
+ 32 s_0 t_3 \|\Sigma_{xz} \|_{\infty} ( c_{1n} + \| W_d \|_{l_{\infty}}).\label{ep2}
\end{equation}
Then 
\begin{align}
P(a_n> \epsilon_{2n})
&\leq
P(a_n> \epsilon_{2n}, \max_{1\leq j\leq p}||\hat{\sigma}^2_j-\sigma^2_j||\leq c_n)+P(\max_{1\leq j\leq p}||\hat{\sigma}^2_j-\sigma^2_j||> c_n) \nonumber \\
&\leq 
P\del[2]{||\frac{X'Z}{n} - \Sigma_{xz}||_{\infty}> t_3}+P(\max_{1\leq j\leq p}||\hat{\sigma}^2_j-\sigma^2_j||> c_n) \nonumber \\
&\leq
10 \exp(-C \kappa_n)+ \frac{K [ 2 E M_1^{2} + 5 E M_2^{2} + E M_3^{2} + E M_4^{2} + E M_5^{2}]}{n \kappa_n} \to 0,\label{s19a}
\end{align}
by (\ref{a8}) and Remark \ref{rem:sigmas} in section \ref{8.4}.  By Assumption 2, the the right hand side of the above display converges to zero. Furthermore, by Lemma \ref{bound}(iv), $t_3 = O (\sqrt{\ln q/n})$ and inspecting the proof of Lemma \ref{l2} yields have $c_{1n} = O (s_0 \sqrt{\ln q/n})$ (upon noting that $c_{1n}=O(c_n)$ in (\ref{c1n})). By 
(\ref{a9}) we have $\| \Sigma_{xz} \|_{\infty} \le C < \infty$, and Assumption 1 gives $\| W_d \|_{l_{\infty}} = O (1)$. Thus
\begin{equation}
\epsilon_{2n} = O( s_0^2 \sqrt{\ln q/n}) \to 0,\label{ep2r}
\end{equation}
by Assumption \ref{2}. Therefore, for $n$ sufficiently large, by (\ref{s19a})
\begin{align*}
P\del[1]{\hat{\phi}^2_{\hat{\Sigma}_{xz\hat{w}}}(s_0)<\phi^2_{\Sigma_{xzw}}(s_0)/2}
&\leq
P\del[1]{a_n> \phi^2_{\Sigma_{xzw}}(s_0)/2}
\\
&\leq 
10 \exp(-C \kappa_n)+ \frac{K [ 2 E M_1^{2} + 5 E M_2^{2} + E M_3^{2} + E M_4^{2} + E M_5^{2}]}{n \kappa_n}
\to 0.
\end{align*}
\end{proof}

\begin{proof}[Proof of Theorem \ref{thm2c}]
We begin with part (i). For $\epsilon>0$ define the events
 \[ A_{1n} = \{ \sup_{\beta_0 \in {\cal B}_{l_0} }|e_j' \Delta| < \epsilon \},\]
 \[ A_{2n} = \left\{ \sup_{\beta_0 \in {\cal B}_{l_0}} \left| \frac{\sqrt{e_j' \hat{\Gamma} \hat{V}_d \hat{\Gamma}' e_j}}{
 \sqrt{e_j' \Gamma V_d \Gamma' e_j}} - 1 \right| < \epsilon \right\},\]
 \[ A_{3,n} = \{ \sup_{\beta_0 \in {\cal B}_{l_0}} | e_j' \hat{\Gamma} \frac{X'Z}{nq} \hat{W}_d \frac{Z'u}{n^{1/2}} - e_j' \Gamma \Sigma_{xz} \frac{W_d}{q} \frac{Z'u}{n^{1/2}} | < \epsilon \}.\]
The probability of $A_{1n}$ converges to one by (\ref{eq:A1set}) while the probability of $A_{2n}$  tends to one by (\ref{t210a}) and $e_j' \Gamma V_d \Gamma' e_j$ being bounded away from zero. Finally, $A_{3n}$ converges  to one in probability by step 2a in the proof of Theorem \ref{thmcl1}. Thus, every $t \in \mathbb{R}$,
 \begin{eqnarray*}
 && \sup_{\beta_0 \in {\cal B}_{l_0}} \left| P \left( \frac{n^{1/2} (\hat{b}_j -  \beta_{j0})}{\sqrt{e_j' \hat{\Gamma} \hat{V}_d \hat{\Gamma}' e_j}} \le t \right) - \Phi (t) \right| \\
 & = &  \sup_{\beta_0 \in {\cal B}_{l_0}}
  \left| P \left( \frac{e_j' \hat{\Gamma} \frac{X'Z}{nq} \hat{W}_d  \frac{Z'u}{n^{1/2}}}{\sqrt{e_j' \hat{\Gamma} \hat{V}_d \hat{\Gamma}' e_j}}
 - \frac{e_j' \Delta}{\sqrt{e_j' \hat{\Gamma} \hat{V}_d \hat{\Gamma}' e_j}} \le t \right)  - \Phi (t) \right| \\
 & \le & \sup_{\beta_0 \in {\cal B}_{l_0}}
  \left| P \left( \frac{e_j' \hat{\Gamma} \frac{X'Z}{nq} \hat{W}_d  \frac{Z'u}{n^{1/2}}}{\sqrt{e_j' \hat{\Gamma} \hat{V}_d \hat{\Gamma}' e_j}}
 - \frac{e_j' \Delta}{\sqrt{e_j' \hat{\Gamma} \hat{V}_d \hat{\Gamma}' e_j}} \le t, A_{1n}, A_{2n}, A_{3n} \right)  - \Phi (t) \right| 
 + P ( \cup_{i=1}^3 A_{in}^c).
 \end{eqnarray*}
Using that $e_j' \Gamma V_d \Gamma' e_j$ is bounded away from zero and does not depend on $\beta_0$ it follows that there exists a universal $D>0$ such that
\begin{eqnarray*}
&& \sup_{\beta_0 \in {\cal B}_{l_0}}
 P  \left( \frac{e_j' \hat{\Gamma} \frac{X'Z}{nq} \hat{W}_d  \frac{Z'u}{n^{1/2}}}{\sqrt{e_j' \hat{\Gamma} \hat{V}_d \hat{\Gamma}' e_j}}
 - \frac{e_j' \Delta}{\sqrt{e_j' \hat{\Gamma} \hat{V}_d \hat{\Gamma}' e_j}} \le t, A_{1n}, A_{2n}, A_{3n} \right) \\
 & = & \sup_{\beta_0 \in {\cal B}_{l_0}} 
P \left( \frac{e_j' \hat{\Gamma} \frac{X'Z}{nq} \hat{W}_d  \frac{Z'u}{n^{1/2}}}{\sqrt{e_j' \Gamma V_d \Gamma' e_j}}
 - \frac{e_j' \Delta}{\sqrt{e_j' \Gamma V_d \Gamma' e_j}} \le t \frac{\sqrt{e_j' \hat{\Gamma} \hat{V}_d \hat{\Gamma}' e_j}}{\sqrt{e_j' \Gamma V_d 
 \Gamma' e_j}}
 , A_{1n}, A_{2n}, A_{3n} \right) \\
 & \le & 
 P \left(  \frac{e_j' \Gamma \Sigma_{xz} \frac{W_d}{q}  \frac{Z'u}{n^{1/2}}}{\sqrt{e_j' \Gamma V_d \Gamma' e_j}} \le t(1+ \epsilon) + 2 D \epsilon \right),
 \end{eqnarray*}
Thus, as the right hand side of the above display does not depend on $\beta_0$ it follows from the asymptotic normality of $\frac{e_j' \Gamma \Sigma_{xz} W_d \frac{Z'u}{n^{1/2}}}{\sqrt{e_j' \Gamma V_d \Gamma' e_j}}$ that for $n$ sufficiently large
 \begin{eqnarray}
 && \sup_{ \beta_0 \in {\cal B}_{l_0}}
 P \left( \frac{e_j' \hat{\Gamma} \frac{X'Z}{nq} \hat{W}_d  \frac{Z'u}{n^{1/2}}}{\sqrt{e_j' \hat{\Gamma} \hat{V}_d \hat{\Gamma}' e_j}}
 - \frac{e_j' \Delta}{\sqrt{e_j' \hat{\Gamma} \hat{V}_d \hat{\Gamma}' e_j}} \le t, A_{1n}, A_{2n}, A_{3n} \right)\nonumber  \\
   & \le & 
    P \left(  \frac{e_j' \Gamma \Sigma_{xz} \frac{W_d}{q} \frac{Z'u}{n^{1/2}}}{\sqrt{e_j' \Gamma V_d \Gamma' e_j}} \le t(1+ \epsilon) + 2 D \epsilon \right) 
    \nonumber \\
    & \le & \Phi (t(1+ \epsilon) + 2 D \epsilon)+  \epsilon,\label{pt31}
 \end{eqnarray}
   Using the continuity of $q\mapsto \Phi(q)$ it follows that for any $\delta>0$ there exists a sufficiently small $\epsilon$ such that
\begin{align*}
 \sup_{ \beta_0 \in {\cal B}_{l_0}}
 P \left( \frac{e_j' \hat{\Gamma} \frac{X'Z}{nq} \hat{W}_d  \frac{Z'u}{n^{1/2}}}{\sqrt{e_j' \hat{\Gamma} \hat{V}_d \hat{\Gamma}' e_j}}
 - \frac{e_j' \Delta}{\sqrt{e_j' \hat{\Gamma} \hat{V}_d \hat{\Gamma}' e_j}} \le t, A_{1n}, A_{2n}, A_{3n} \right)
\leq
\Phi(t)+\delta+\epsilon
\end{align*}
   Following a similar reasoning as above one can also show that for any $\delta>0$ and $\epsilon>0$ sufficiently small
      \begin{eqnarray}
 && \inf_{ \beta_0 \in {\cal B}_{l_0}}
 P \left( \frac{e_j' \hat{\Gamma} \frac{X'Z}{nq} \hat{W}_d  \frac{Z'u}{n^{1/2}}}{\sqrt{e_j' \hat{\Gamma} \hat{V}_d \hat{\Gamma}' e_j}}
 - \frac{e_j' \Delta}{\sqrt{e_j' \hat{\Gamma} \hat{V}_d \hat{\Gamma}' e_j}} \le t, A_{1n}, A_{2n}, A_{3n} \right)
    \ge  \Phi (t) - 2 \epsilon- \delta.\label{pt32}
   \end{eqnarray}
From (\ref{pt31}) and (\ref{pt32}) it can be concluded that
   \[ \sup_{ \beta_0 \in {\cal B}_{l_0}} \left| P \left( \frac{n^{1/2} e_j' (\hat{b} - \beta_0)}{\sqrt{ e_j'  \hat{\Gamma} \hat{V}_d \hat{\Gamma}' e_j}}
   \le t \right) - \Phi (t) \right| \to 0.\]

Part (ii) can be established in a similar fashion as the proof of Theorem 3ii in \cite{ck18}.
   
We now turn to part (iii).
   \begin{eqnarray*}
   && n^{1/2} \sup_{\beta_0 \in {\cal B}_{l_0}} diam ( [ \hat{b}_j - z_{1 - \alpha/2} \frac{\hat{\sigma}_{bj}}{n^{1/2}}, \hat{b}_j
   + z_{1 - \alpha/2} \frac{\hat{\sigma}_{bj}}{n^{1/2 }}]) \\
   & = & \sup_{\beta_0 \in {\cal B}_{l_0}} 2 \hat{\sigma}_{bj} z_{1 - \alpha/2}   \\
   & = & 2 [  \sup_{\beta_0 \in {\cal B}_{l_0}} \sqrt{e_j' \Gamma V_d \Gamma' e_j} + o_p (1)] z_{1 - \alpha/2} \\
   & = & O_p (1),
   \end{eqnarray*}
   by Theorem \ref{thmcl1}(ii) for the second equality, and Assumption \ref{5cl} (ii) for the last equality.   
\end{proof}
. 

\bibliographystyle{chicagoa}
\bibliography{refer1}

  \end{document}